\documentclass[11pt,leqno,twoside]{amsart}
\usepackage{amsfonts,amssymb, dsfont, mathrsfs}
\usepackage{tikz, subfigure, xcolor} 
\usetikzlibrary{patterns}
\usepackage{epigraph}

\usepackage{mathdots}%

\usepackage{amsmath}
\usepackage{amsthm}
\usepackage{qsymbols}
\usepackage{latexsym}
\usepackage{chngcntr}
\usepackage[noadjust]{cite}
\usepackage{paralist}
\usepackage{doi, url, hyperref}
\usepackage{comment}
\usepackage{enumerate}

\usepackage{esint}

\usepackage{color}

\newtheorem{theorem}{Theorem}[section]

\newtheorem{lemma}[theorem]{Lemma}
\newtheorem{proposition}[theorem]{Proposition}
\newtheorem{example}[theorem]{Example}
\newtheorem{corollary}[theorem]{Corollary}
\theoremstyle{definition}
\newtheorem{definition}[theorem]{Definition}
\newtheorem{remark}[theorem]{Remark}
\newtheorem{remarks}[theorem]{Remarks}

\providecommand{\norm}[1]{\lVert#1\rVert} 
\providecommand{\abs}[1]{\lvert#1\rvert} 
\newcommand{\IR}{\mathbb{R}}
\newcommand{\IC}{\mathbb{C}}
\newcommand{\IN}{\mathbb{N}}
\newcommand{\IZ}{\mathbb{Z}}

\newcommand{\II}{\mathbb{I}}

\newcommand{\au}{\underline{a}}

\newcommand{\cH}{\mathcal{H}}
\newcommand{\cK}{\mathcal{K}}

\newcommand{\cM}{\mathcal{M}}

\newcommand{\cO}{\mathcal{O}}

\newcommand{\cE}{\mathcal{E}}

\newcommand{\cG}{\mathcal{G}}
\newcommand{\cS}{\mathcal{S}}
\newcommand{\cP}{\mathcal{P}}
\newcommand{\cR}{\mathcal{R}}

\newcommand{\cC}{\mathcal{C}}
\newcommand{\cL}{\mathcal{L}}

\newcommand{\cT}{\mathcal{T}}
\newcommand{\cI}{\mathcal{I}}
\newcommand{\cV}{\mathcal{V}}

\renewcommand{\L}{\mathrm{L}}

\renewcommand{\S}{\mathrm{S}}

\newcommand{\cf}{\emph{cf.}}
\newcommand{\ie}{{\emph{i.e.}}}

\newcommand{\diag}{\mathrm{diag}}      

\renewcommand\Re{\operatorname{Re}}
\renewcommand\Im{\operatorname{Im}}

\DeclareMathOperator{\supp}{supp}
\DeclareMathOperator{\dist}{dist}

\DeclareMathOperator{\Hom}{Hom}

\DeclareMathOperator{\Ran}{Ran} 
 
\DeclareMathOperator{\Rank}{Rank} 
\DeclareMathOperator{\Ker}{Ker}
\DeclareMathOperator{\Dom}{Dom}
\DeclareMathOperator{\dom}{dom}

\DeclareMathOperator{\Gr}{Gr}

\numberwithin{equation}{section} 

\title{Hidden symmetries in non-self-adjoint graphs}

\author{Amru Hussein} 
\address{Department of Mathematics,
	TU Kaiserslautern,  Paul-Ehrlich-Stra\ss e 31, 67663 Kaiserslautern, Germany}
\email{hussein@mathematik.uni-kl.de}



\thanks{The author was partially supported by the MathApp – Mathematics Applied to Real-World Problems, part of the Research Initiative of the Federal State of Rhineland-Palatinate}

\begin{document}

\begin{abstract}
On finite metric graphs the set of all realizations of the Laplace operator in the edgewise defined $L^2$-spaces are studied. These are defined by coupling boundary conditions at the vertices most of which define non-self-adjoint operators. In [Hussein, Krej\v{c}i\v{r}\'{\i}k, Siegl, \emph{Trans. Amer. Math. Soc.}, 367(4):2921--2957, 2015] a notion of regularity of boundary conditions by means of the Cayley transform of the parametrizing matrices has been proposed. The main point presented here is that not only the existence of this Cayley transform is essential for basic spectral properties, but also its poles and its asymptotic behaviour. It is shown that these poles and asymptotics can be characterized using the quasi-Weierstrass normal form which exposes some ``hidden'' symmetries of the system. Thereby, one can analyse not only the spectral theory of these mostly non-self-adjoint Laplacians, but also the well-posedness of the time-dependent  heat-, wave- and Schrödinger equations on finite metric graphs as initial-boundary value problems.
In particular, the generators of $C_0$- and analytic semigroups and $C_0$-cosine operator functions can be characterized. On star-shaped graphs a characterization of generators of bounded $C_0$-groups and thus of operators similar to self-adjoint ones is obtained.
\end{abstract}

\maketitle
\section{Introduction}

Laplace operators on finite metric graphs are realization of minus the second derivative operator on each edge
with coupling boundary conditions on the vertices. The boundary conditions encode the geometry of the graph and model the transmission of information between the edges, and while the edgewise defined Laplacian is a quite elementary operator, the boundary conditions can contribute at times rather involved features.
Such graph or quasi-one-dimensional models are used for instance in quantum mechanics, where they are called quantum graphs, see e.g. \cite{BeKu_book} and references therein, and also in the modelling of a variety of wave and diffusion dynamics, compare e.g. \cite{Mug_book, Ali1994, Kuchment2002} just to mention a few. 

In the context of quantum mechanics \emph{self-adjoint} Laplacians are admissible Hamiltonians defining a unitary time evolution. In contrast, when studying the wave- or heat equation  certain \emph{non-self-adjoint} Laplacians are admissible as well. Even for the Schrödinger equation \emph{quasi-self-adjointness} of the Hamiltonian, that is, its similarity to a self-adjoint operator instead of self-adjointness, might be sufficient, and this possibility has been explored theoretically by some physicists and  mathematicians starting with the work \cite{SGH}, see also \cite{Znojil_book} for recent developments and further references.
The aim here is now to analyse the set of all possible  boundary conditions, and to study
how these general boundary conditions on the vertices influence spectral properties of the mostly \emph{non-self-adjoint} Laplacians on finite metric graphs in particular with regard to the well-posedness of the time-dependent linear heat-, wave-, and Schr\"odinger equations on finite metric graphs. 
The main result is a characterization of the generators of $C_0$-semigroups, analytic semigroups and $C_0$-$\cos$-families on general finite metric graphs by means of the coupling boundary conditions. For the case of star graphs, such a characterization of generators of $C_0$-groups is also obtained, but this does not carry over to  general  finite metric graphs.


\subsection{The theory of self-adjoint boundary conditions}
In order to develop a systematic theory for non-self-adjoint boundary conditions, the self-adjoint case is a natural starting point.
The case of Laplace operators on finite metric graphs with \emph{self-adjoint} boundary conditions is well-studied, and the set of self-adjoint boundary conditions can be characterized by classical extension theory, and also using quadratic forms.

Considering for instance a star-shaped graph consisting of $d$ copies of $[0,\infty)$ glued together at the origins as sketched in Figure~\ref{fig:finitemetricgraphs} (a), then 
linear boundary conditions of the form
\begin{align}\label{eq:AB}
A\psi(0) + B\psi^\prime(0)=0
\end{align}
for matrices $A,B\in \IC^{d\times d}$ define a self-adjoint Laplacian $-\Delta(A,B)$ in the edgewise defined $L^2$-space $L^2(\cG)=\oplus_{i=1}^dL^2([0,\infty);\IC)$ over the graph, \ie , a realization of minus the second derivative operator on each edge with the boundary condition \eqref{eq:AB}, if and only if 
\begin{align}\label{eq:KostrykinSchrader}
AB^\ast = BA^\ast \quad \hbox{and} \quad \Rank (A\, B)=d,
\end{align}
cf. \cite{KS1999, Harmer2000}. Here, $d$ is also equal to the deficiency indices for the minimal Laplacian defined on  $\oplus_{i=1}^dC_0^\infty((0,\infty);\IC)$, and by $\Ker (A\, B)$ these boundary conditions are related to maximal isotropic subspaces of $\IC^{2d}$. 


The link to the classical extension theory where self-adjoint extensions are parameterized by unitary maps is given by the local scattering matrix 
\begin{align*}
\mathfrak{S}(k,A,B) = - (A+ik B)^{-1} (A-ik B), \quad k>0.
\end{align*}
This is unitary and boundary conditions can equivalently be described by
\begin{eqnarray*}
	A_{\mathfrak{S}}:= - \frac{1}{2} \left(\mathfrak{S}(k,A,B) -\mathds{1}\right) & \mbox{and} & B_{\mathfrak{S}}:=  \frac{1}{2ik} \left(\mathfrak{S}(k,A,B) +\mathds{1}\right) \quad \hbox{for some } k>0,
\end{eqnarray*}
and the other way round this defines self-adjoint boundary conditions replacing $\mathfrak{S}(k,A,B)$ by any unitary matrix $U$.

Looking at the associated quadratic form, it is convenient to consider a parameterization using the orthogonal projector $P$ onto $\ker B$, and equivalent boundary conditions defined by
\begin{align*}
A' = P+L \quad \hbox{and} \quad B'=P^\perp \quad \hbox{with} \quad  L= (B\vert_{(\ker B)^\perp})^{-1}A.
\end{align*} 
Then $L=P^\perp L P^\perp$ and the boundary conditions \eqref{eq:AB} decompose into 
\begin{align*}
L\psi(0) + P^\perp\psi^\prime(0)=0 \quad \hbox{and} \quad P \psi(0)=0,
\end{align*}
that is into a Robin/Neumann part and a Dirichlet part, and the associated quadratic form can be determined to be given by 
\begin{equation*}
\int_{\cG} \abs{ \psi^{\prime}}^2 
- \langle L P^{\perp}\underline{\psi},P^{\perp}\underline{\psi}\rangle_{\cK} \quad \hbox{for }
\psi\in  \{ \varphi\in H^1(\cG,\au) \mid P\underline{\varphi}=0\}.
\end{equation*}
For the case of self-adjoint boundary conditions, these three parametrizations -- by maximal isotropic subspaces, unitary matrices, and quadratic forms --  are equivalent, compare \cite[Theorem 1.4.4]{BeKu_book}. However, for non-self-adjoint boundary conditions each of the three viewpoints presented above leads to a different subclass of non-self-adjoint operators. 

 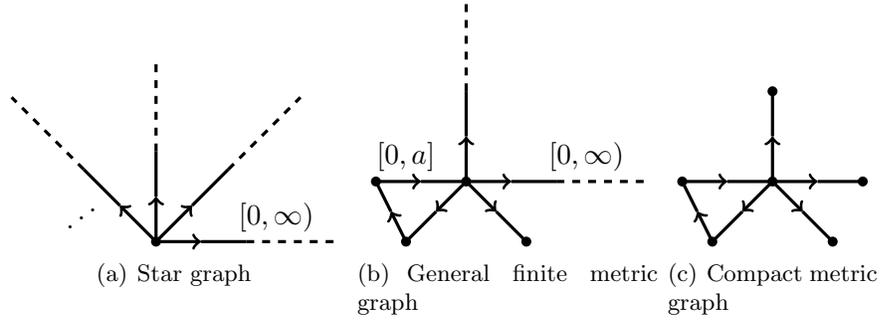
\begin{figure}[h]
	\begin{center}
		\subfigure[Star graph]{
			\begin{tikzpicture}[scale=0.4]
			\fill[black] (0,0) circle (1ex);
			\draw[->, black, very thick] (0,0) -- (1.5,0);
			\draw[black, very thick] (1.5,0) -- (3,0);
			\draw[black, very thick, dashed] (3,0) -- (6,0);
			\draw[->, black, very thick] (0,0) -- (0, 1.5);
			\draw[black, very thick] (0, 1.5) -- (0,3);
			\draw[black, very thick, dashed] (0,3) -- (0,6);
			\node[above] at (4, 0) {$[0,\infty)$};
			\draw[->, black, very thick] (0,0) -- (1.2, 1.2);
			\draw[black, very thick] (1.2, 1.2) -- (2.4,2.4);
			\draw[black, very thick, dashed] (2.4,2.4) -- (4.8,4.8);
			\draw[->, black, very thick] (0,0) -- (-1.2, 1.2);
			\draw[black, very thick] (-1.2, 1.2) -- (-2.4,2.4);
			\draw[black, very thick, dashed] (-2.4,2.4) -- (-4.8,4.8);
			\node[above] at (-2.5, 0) {$\iddots$};
			\end{tikzpicture}
		}
		\subfigure[General finite metric graph]{
			\begin{tikzpicture}[scale=0.4]
			\fill[black] (0,0) circle (1ex);
			\draw[black, very thick] (0,0) -- (-1.5,0);
			\draw[<-, black, very thick] (-1.5,0) -- (-3,0);
			\fill[black] (-3,0) circle (1ex);
			\draw[->, black, very thick] (0,0) -- (1,-1);
			\draw[black, very thick] (1,-1) -- (2,-2);
			\fill[black] (2,-2) circle (1ex);
			\draw[->, black, very thick] (-2,-2) -- (-2.5,-1);
			\draw[black, very thick] (-2.5,-1) -- (-3,0);
			\draw[->, black, very thick] (0,0) -- (-1,-1);
			\draw[black, very thick] (-1,-1) -- (-2,-2);
			\fill[black] (-2,-2) circle (1ex);
			\draw[->, black, very thick] (0,0) -- (1.5,0);
			\draw[black, very thick] (1.5,0) -- (3,0);
			\draw[black, very thick, dashed] (3,0) -- (6,0);
			\draw[->, black, very thick] (0,0) -- (0, 1.5);
			\draw[black, very thick] (0, 1.5) -- (0,3);
			\draw[black, very thick, dashed] (0,3) -- (0,6);
			\node[above] at (4, 0) {$[0,\infty)$};
			\node[above] at (-2, 0) {$[0,a]$};
			\end{tikzpicture}
		}
		\subfigure[Compact metric graph]{	
			\begin{tikzpicture}[scale=0.4]
			\fill[black] (0,0) circle (1ex);
			\draw[black, very thick] (0,0) -- (-1.5,0);
			\draw[<-, black, very thick] (-1.5,0) -- (-3,0);
			\fill[black] (-3,0) circle (1ex);
			\draw[->, black, very thick] (0,0) -- (1,-1);
			\draw[black, very thick] (1,-1) -- (2,-2);
			\fill[black] (2,-2) circle (1ex);
			\draw[->, black, very thick] (-2,-2) -- (-2.5,-1);
			\draw[black, very thick] (-2.5,-1) -- (-3,0);
			\draw[->, black, very thick] (0,0) -- (-1,-1);
			\draw[black, very thick] (-1,-1) -- (-2,-2);
			\fill[black] (-2,-2) circle (1ex);
			\draw[->, black, very thick] (0,0) -- (1.5,0);
			\draw[black, very thick] (1.5,0) -- (3,0);
			\fill[black] (3,0) circle (1ex);
			\draw[->, black, very thick] (0,0) -- (0, 1.5);
			\draw[black, very thick] (0, 1.5) -- (0,3);
			\fill[black] (0,3) circle (1ex);	
			\end{tikzpicture}
		}
		\caption{Finite metric graphs} \label{fig:finitemetricgraphs}
	\end{center}
\end{figure}

\subsection{Some features of non-self-adjointness}
To illustrate the difficulties with general \emph{non-self-adjoint} boundary conditions, consider the Laplacians $-\Delta(A_{\tau},B_{\tau})$ defined on two copies of $[0,\infty)$ by boundary conditions of the form \eqref{eq:AB} with
\begin{eqnarray}\label{eq:AtauBtau}
	A_{\tau}=\begin{bmatrix} 1 & -e^{i\tau} \\ 0 & 0\end{bmatrix} &\mbox{and} & B_{\tau}=\begin{bmatrix}  0 & 0 \\ 1 & e^{-i\tau} \end{bmatrix} \quad \hbox{for } \tau \in [0,\pi/2].
\end{eqnarray} 
Identifying the two copies of $[0,\infty)$ with the real line, this translates to
\begin{align*}
	\psi(0+) = e^{i\tau} \psi(0-)
	\quad\mbox{and}\quad
	\psi^{\prime}(0+) = e^{-i\tau} \psi^{\prime}(0-) \quad \hbox{for } \tau \in [0,\pi/2],
\end{align*}
where $0+$ and $0-$ denote the limits from the right and the left hand side, respectively.
This is an example for so-called $\mathcal{PT}$-symmetric point interactions, and it has been studied e.g. in~\cite{Kurasov02, Petr, Kuzel05}, and it has been the model case for the analysis in \cite{HKS2015}.  

So, for $\tau=0$, $-\Delta(A_{0},B_{0})$ corresponds to the Laplacian on $\IR$, while for $\tau \in (0,\pi/2]$ the form
\begin{equation*}
	\langle -\Delta(A_{\tau},B_{\tau})\psi,\psi\rangle = \int_{\IR} \abs{\psi^{\prime}}^2 + (1-e^{2i\tau}) \psi(0+)\overline{\psi^{\prime}(0+)},
\end{equation*}
is not symmetric, and therefore $-\Delta(A_{\tau},B_{\tau})$ cannot be self-adjoint for $\tau \in (0,\pi/2]$. Moreover, due to the trace term $\psi^{\prime}(0+)$, which cannot be balanced by lower order terms, one can show that this form is not even closable, and therefore the approach by quadratic forms in the usual $L^2(\IR)$-space fails for this operator. Nevertheless, one can show that the operators $-\Delta(A_{\tau},B_{\tau})$ for $\tau \in (0,\pi/2)$ are similar to the self-adjoint Laplacian $-\Delta(A_{0},B_{0})$, while for $\tau=\pi/2$, the spectrum of $-\Delta(A_{\pi/2},B_{\pi/2})$ is the entire  complex plain, see e.g. \cite[Section 6.4]{HKS2015}.

For self-adjoint $-\Delta(A,B)$ the heat-, wave, and Schrödinger equations  are well-posed which in terms of operator theory means that $\Delta(A,B)$ generates a $C_0$-semigroup and a $C_0$-$\cos$-family, and that $i\Delta(A,B)$ generates a $C_0$-group, respectively. However, apart from these self-adjoint Laplacians there are also many \emph{non-self-adjoint} Laplacians for which these equations are well-posed. 
In the abstract situation, these generation properties can be  characterized in terms of resolvent estimates as sketched roughly in Figure~\ref{fig:resolvent_estimates}.
To study generation properties -- instead of verifying resolvent estimates of the type sketched in Figure~\ref{fig:resolvent_estimates} directly --  one  relies often on convenient criteria using forms and the numerical range of an operator. Taking the operator $-\Delta(A_{\tau},B_{\tau})$ discussed above as a guiding example, one sees that the form approach is not directly applicable here, but nevertheless -- since it is similar to the Laplacian on the real line -- each of the heat-, wave-, and Schrödinger equation for $-\Delta(A_{\tau},B_{\tau})$ is well-posed for $\tau \in [0,\pi/2)$.
This illustrates the well-known fact that the numerical range is only stable under unitary similarity transforms but not under general similarity transforms while resolvent estimates and spectra -- and therefore generation properties  -- are stable under general similarity transforms. In particular, in order to classify all generators  of $C_0$-semigroups,  $C_0$-$\cos$-families, and $C_0$-groups the direct application of form methods disregarding similarity relations is in general not sufficient, and cases just as the operators $-\Delta(A_{\tau},B_{\tau})$ for $\tau \in (0,\pi/2)$ might be overlooked at first sight.

%
%
%

%
%
%

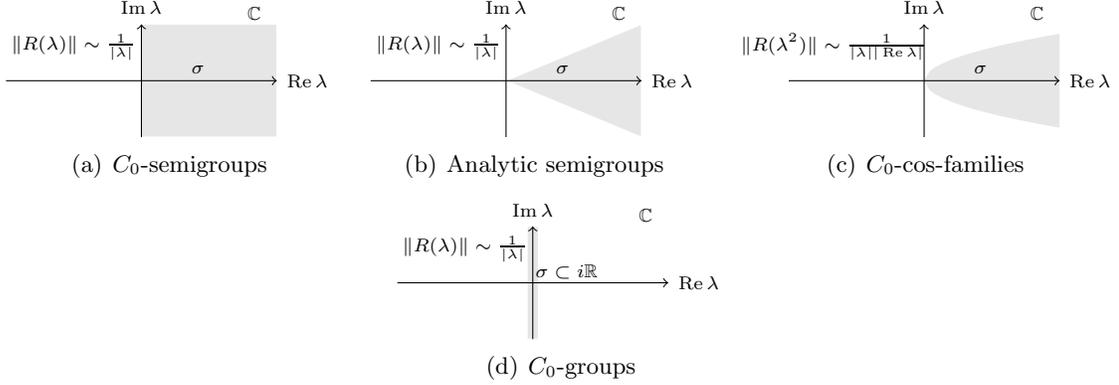
\begin{figure}[h]
	\begin{center}
		\subfigure[$C_0$-semigroups]{
			\begin{tikzpicture}[scale=1.5]
			\filldraw [draw=white, fill=gray!20!white] (0,-0.5) -- (0,0.5) -- (1.2,0.5) -- (1.2,-0.5) -- (0,-0.5) -- cycle;
			\draw[->] (-1.2,0) -- (1.2,0) node[right] {\tiny{$\Re \lambda$}};
			\draw[->] (0,-0.5) -- (0,0.5) node[above] {\tiny{$\Im \lambda$}};
			\draw (1,0.6) node {\tiny{$\IC$}};
			\draw (-0.6,0.3) node {\tiny{$\norm{R(\lambda)} \sim \tfrac{1}{|\lambda|}$}};
			\draw (0.5,0.1) node {\tiny{$\sigma$}};
			\end{tikzpicture}
		}	
		\subfigure[Analytic semigroups]{
			\begin{tikzpicture}[scale=1.5]
			\filldraw [draw=white, fill=gray!20!white] (0,0) -- (1.2,0.5) -- (1.2,-0.5) -- (0,0) -- cycle;
			\draw[->] (-1.2,0) -- (1.2,0) node[right] {\tiny{$\Re \lambda$}};
			\draw[->] (0,-0.5) -- (0,0.5) node[above] {\tiny{$\Im \lambda$}};
			\draw (1,0.6) node {\tiny{$\IC$}};
			\draw (-0.6,0.3) node {\tiny{$\norm{R(\lambda)} \sim \tfrac{1}{|\lambda|}$}};
			\draw (0.5,0.1) node {\tiny{$\sigma$}};
			\end{tikzpicture}
		}	
		\subfigure[$C_0$-$\cos$-families]{
			\begin{tikzpicture}[scale=1.5]
			\fill [gray!20!white, domain=0:1.2, variable=\x]
			(0, 0)
			-- plot ({\x}, {0.38*sqrt(\x)})
			-- (1.2, 0)
			-- cycle;
			\fill [gray!20!white, domain=0:1.2, variable=\x]
			(0, 0)
			-- plot ({\x}, {-0.38*sqrt(\x)})
			-- (1.2, 0)
			-- cycle;
			\draw[->] (-1.2,0) -- (1.2,0) node[right] {\tiny{$\Re \lambda$}};
			\draw[->] (0,-0.5) -- (0,0.5) node[above] {\tiny{$\Im \lambda$}};
			\draw (1,0.6) node {\tiny{$\IC$}};
			\draw (-0.8,0.3) node {\tiny{$\norm{R(\lambda^2)} \sim \tfrac{1}{|\lambda| |\Re \lambda|}$}};
			\draw (0.5,0.1) node {\tiny{$\sigma$}};
			\end{tikzpicture}
		}	
		\subfigure[$C_0$-groups]{
			\begin{tikzpicture}[scale=1.5]
			\filldraw [draw=white, fill=gray!20!white] (0.05, -0.5) -- (0.05,0.5) -- (-0.05,0.5) -- (-0.05,-0.5) -- (0.05, -0.5) -- cycle;
			\draw[->] (-1.2,0) -- (1.2,0) node[right] {\tiny{$\Re \lambda$}};
			\draw[->] (0,-0.5) -- (0,0.5) node[above] {\tiny{$\Im \lambda$}};
			\draw (1,0.6) node {\tiny{$\IC$}};
			\draw (-0.6,0.3) node {\tiny{$\norm{R(\lambda)} \sim \tfrac{1}{|\lambda|}$}};
			\draw (0.3,0.1) node {\tiny{$\sigma\subset i\IR$}};
			\end{tikzpicture}
		}	
		\caption{Location of spectra $\sigma$ (inside the grey regions) and resolvent estimates (outside the grey regions) for different types of generators}\label{fig:resolvent_estimates}
	\end{center}
\end{figure}

\subsection{A classification of non-self-adjoint boundary conditions}
After the direct application of the form approach has been ruled out,    
another possible starting point to study general boundary conditions including all \emph{non-self-adjoint} ones, could be to attempt a generalization of the Kostrykin-Schrader parametrization \eqref{eq:KostrykinSchrader}.
Then one first observes that  the condition $\Rank (A\, B)=d$, \ie , that one has $d$ linearly independent boundary conditions, is necessary for basic spectral properties such as a non-empty resolvent set. However this is not sufficient, 
as for the example discussed above $\Rank (A_{\pi/2}\, B_{\pi/2})=2$, while the resolvent set is empty, and one can construct further examples where even the spectrum is empty, see e.g. Example~\ref{ex:emptyspec} below.
To exclude such cases, in addition some kind of regularity assumption is needed. 

For self-adjoint boundary conditions of the form \eqref{eq:AB} defined by matrices $A,B$ satisfying \eqref{eq:KostrykinSchrader},  the local scattering matrix is given by minus the Cayley transform of these matrices
\begin{align*}
\mathfrak{S}(k,A,B) = - (A+ik B)^{-1} (A-ik B), \quad k>0,
\end{align*}
and this object appears in many instances such as in the Green's function or in the eigenvalue equation for the corresponding Laplacian, cf. e.g. \cite{KS1999, KS2006}.
This motivated the notion of regularity of boundary conditions proposed in \cite{HKS2015} by  Krej\v{c}i\v{r}\'{\i}k, Siegl, and the author, where boundary conditions defined by matrices $A,B$ are \emph{regular} if their Cayley transform exists for some $k\in \IC$. 
This regularity assumption implies many basic spectral properties while the  example $-\Delta(A_{\pi/2}, B_{\pi/2})$ mentioned above is defined by irregular boundary conditions.
More specifically, here it is shown that \emph{irregular boundary conditions} are exactly those  which define operator with exponentially growing resolvent norm, \ie,  
\begin{align*}
\cO(e^{c\sqrt{|\lambda|}}) \lesssim \norm{(-\Delta(A,B)-\lambda)^{-1}} \quad \hbox{as } \lambda\to -\infty,
\end{align*}
and in particular these cannot be generators of $C_0$-semigroups,  $C_0$-$\cos$-families, and $C_0$-groups.

\subsection{Regularity assumptions and hidden structures}
Turning  now to regular boundary  conditions, the key observation communicated here is that the quasi-Weierstrass normal form reveals some additional structure or ``hidden'' symmetry of the operator, and thereby it allows for a systematic analysis of Laplacians defined by regular boundary conditions. The quasi-Weierstrass normal form has been introduced in the context of algebraic differential equations by Berger,  Ilchmann, and Trenn, see~\cite{BIT2012}, and for regular boundary conditions it implies that there are equivalent boundary conditions defined by
	\begin{align}\label{eq:qWnf}
A= G^{-1}\begin{bmatrix}
L & 0 \\ 0 & \mathds{1}_{\IC^{d-m}}
\end{bmatrix}G \quad \hbox{and} \quad 
B= G^{-1}\begin{bmatrix}
\mathds{1}_{\IC^{m}} & 0 \\ 0 & N_B
\end{bmatrix}G, 
\end{align} 
for a similarity transform $G$, a block decomposition $\IC^{d}=\IC^{d-m}\times \IC^{m}$ with $0\leq m\leq d$, a nilpotent matrix $N_B\in \IC^{(d-m)\times (d-m)}$, and a matrix $L\in \IC^{m\times m}$. Therefrom one can deduce that the Cayley transform is uniformly bounded outside its poles if and only if $N_B\equiv 0$, and the order of the poles of $\mathfrak{S}(\cdot,A,B)$ is related to $N_B$ and the Jordan normal form of $L$, see Subsection~\ref{subsect:quasiWeierstrass}. 
Taking then advantage of this structure 
one can characterize the generators of $C_0$-semigroups in $L^2$-spaces. These are defined exactly by those regular boundary conditions with $N_B\equiv 0$, see Theorem~\ref{thm:semigroup} below which is one of the major findings of this article. Moreover, these $C_0$-semigroups extend to analytic semigroups, and its generators also generate $C_0$-cosine operator function. Thus the question of the well-posedness of the heat and wave equation is completely answered in terms of boundary conditions. A heuristic explanation for the effect of $N_B\neq 0$ is discussed in Subsection~\ref{exDS}. In particular, if $N_B\neq 0$, then the Cayley transform is not uniformly bounded as $|k|\to \infty$, and as a result the resolvent does not exhibit the decay sketched in Figure~\ref{fig:resolvent_estimates}.

There has been a number of results on some classes of boundary conditions defining generators of analytic semigroups where in forms of the type 
\begin{equation*}
\langle -\psi'', \psi \rangle
=\int \abs{ \psi^{\prime}}^2 
+ \langle \psi'(0),\psi(0)\rangle_{\IC^d}
\end{equation*}
the trace of the derivative $\psi'(0)$ can be balanced by terms involving only the trace $\psi(0)$. More concretely,  if the boundary condition 
\eqref{eq:AB} can be represented equivalently  by 
\begin{align*}
A= L + P \quad \hbox{and} \quad B=P^\perp
\end{align*}
for an orthogonal projector $P$ and an operator $L$ acting in the range of $P^\perp=\mathds{1}-P$, then \eqref{eq:AB} becomes
$P^{\perp}\psi'(0)+L\psi(0)=0$  and  $P\psi(0)=0$,
see \cite{Kramar2007, Mug07, Mug_book}. However, as it turns out, not all generators of $C_0$-semigroups fall under this category, e.g. the boundary conditions defined by \eqref{eq:AtauBtau} above.  It seems that these cases have not been addressed  in the literature so far, and the characterization  presented here closes this gap.

\subsection{A generalized  reflection principle}
The methods to prove the generation properties for general regular boundary conditions with $N_B=0$ is to alter the scalar product such that the operators $-\Delta(A,B)$ can still be associated with closed sectorial forms of Lions type which implies the generation of analytic semigroups and $C_0$-cosine operator functions. This is elaborated in Subsection~\ref{subsec:forms}. One key observation in \cite{HKS2015} has been  that on star graphs with $d$ edges, matrices $G\in \IC^{d\times d}$ induce maps in function spaces by
\begin{align*}
 L^2(\cG) \rightarrow L^2(\cG), \quad  \psi  \mapsto G \psi, \quad \psi = \begin{bmatrix}
\psi_1 \\ \vdots \\ \psi_d
\end{bmatrix} \in L^2(\cG)=\oplus_{i=1}^dL^2([0,\infty);\IC), 
\end{align*} 
and therefore relations of the type \eqref{eq:qWnf} induces a similarity relation for the corresponding operators on star graphs. In the simplest case of the Laplacian on the real line, one example for such a similarity relation are the even and odd reflections, and therefore this way to construct similarity transforms is referred to here as \emph{generalized  reflection principle}.  
In general this generalized  reflection principle however cannot be translated to graphs with internal edges, because then the number of the edges and of their endpoints do not coincide any more. However, using a cut-off and extension procedure, this can be localized in a sense to construct an equivalent scalar product.     
Note that in contrast to regular boundary conditions, for irregular boundary conditions the resolvent has even an exponential growth and for boundary conditions with $N_B\neq 0$, the resolvent does not decay fast enough or even grows polynomially as $\lambda \to -\infty$, see Subsections~\ref{subsec:resolvent_reg} and \ref{subsec:irr}.  

\subsection{Quasi-Hermitian operators on star graphs}
It should be highlighted that the characterization of generators of $C_0$-semigroups involves -- just as for self-adjointness -- only the boundary conditions and not any other geometric feature. 
In contrast, for the question if a Laplacian is the generator of a bounded $C_0$-group, \ie , if it is similar to a self-adjoint operator, the geometry matters. Here, it is shown, that for the relatively simple geometry of a star graph, where each edge is identified with $[0,\infty)$, similarity to a self-adjoint operator can be characterized in terms of the poles of the Cayley transform, see Theorem~\ref{thm:similarity} below. The proof of this theorem relies on the generalized reflection principle on star graphs, and for other geometries the characterization is in general no longer true, that is, using the same boundary conditions in a different geometric setting can destroy the similarity relation, see Subsection~\ref{subsec:geometry}. 
The characterization obtained here can be seen as a generalization to finite star graphs of theorems on particular point interactions on the real line by Mostafazadeh, see \cite{Mostafazadeh2006}, and Grod and Kuzhel \cite{GrodKuzhel2014}. 
The study of operators similar to self-adjoint ones is relevant in the so-called non-Hermitian quantum mechanics, see e.g. \cite{Bender07, SGH, Mostafazadeh2010} and the references therein. Some $\cP\cT$-symmetric operator fall also in this class, see e.g. \cite{Kurasov02, Kuzel05, Krej2010, Krej2019} and the many references given therein. A generalization of $\cP\cT$-symmetry to star graphs is proposed and analysed in \cite{Astudillo2015}. The general situation of  non-self-adjoint
extensions of symmetric operators with only absolutely continuous spectrum have been studied in \cite{Kiselev2009, Kiselev2011, Kiselev2000}, see also the references therein.

\subsection{Structure and outline}
In the subsequent Section~\ref{sec:Laplace} basic concepts and definitions are recapitulated. The notion of regularity of boundary conditions and  properties of  Cayley transforms are elaborated in Section~\ref{Sec:bc}, and some differences to the classical Birkhoff-Tamarkin theory and its notion of regularity are exemplified in Subsection~\ref{subsec:birkhoff}. The characterization of semigroup generators is discussed in Section~\ref{sec:semigroup}. There also the completeness of the root vectors for compact graphs is addressed. Section~\ref{sec:similarity} deals with the similarity to self-adjoint operators on star graphs.

\section{Laplacians on finite metric graphs}\label{sec:Laplace}

\subsection{Finite metric graphs}
A \emph{graph} is a $4$-tuple 
$$\cG = \left( \cV, \cI,\cE, \partial \right)$$  
with the set of \textit{vertices} $\cV$, 
 the set of \textit{internal edges} $\cI$,  
and  the set of \textit{external edges} $\cE$, 
and summing up $\cE \cup \cI$ is the set of \textit{edges}.
The \textit{boundary map} $\partial$ assigns to each internal edge $i\in \cI$ 
an ordered pair of vertices $\partial (i)=\left(\partial_-(i),\partial_+(i)\right)\in \cV \times \cV$, 
where $\partial_-(i)$ is called its \textit{initial vertex} 
and $\partial_+(i)$ its \textit{terminal vertex}. 
Each external edge $e\in \cE$ is mapped by $\partial$ onto a single, 
its initial vertex. 
A graph is called \textit{finite} if $\abs{\cV}+\abs{\cI}+\abs{\cE}<\infty$. Instead of using a boundary map, the structure of the graph can also be encoded by the adjacency matrix, cf. e.g. \cite[Section 1.1]{BeKu_book}.

A graph $\cG$ can be endowed with the following metric structure. 
Each internal edge $i\in \cI$ is identified 
with an interval $[0,a_i]$, with $a_i>0$, 
such that its initial vertex corresponds to~$0$ 
and its terminal vertex to~$a_i$. 
Each external edge $e\in \cE$ is identified with the half line $[0,\infty)$ 
such that $\partial(e)$ corresponds to~$0$. 
The numbers~$a_i$ are called \textit{lengths} of the internal edges $i\in \cI$, 
and they are summed up into the vector 
$$\au=\{a_i\}_{i\in \cI}\in (0,\infty)^{\abs{\cI}}.$$
The $2$-tuple consisting of a finite graph endowed with a metric structure is called a \textit{metric graph} $(\cG,\au)$. 
The metric on $(\cG,\au)$ is defined via minimal path lengths, and $(\cG,\au)$ is \textit{compact} if $\cE=\emptyset$. The notations used here largely parallel those in the works of Kostrykin and Schrader, see e.g. \cite{KS2006}. The construction of a metric graph as quotient space is discussed by Mugnolo in \cite{Mugnolo_whatisagraph}, where a metric graph consists of intervals glued together at their end points according to the structure of a given graph. For further details see e.g. \cite[Section 1.3]{BeKu_book}.

\subsection{Function spaces on finite metric graphs}
A function $\psi \colon (\cG,\au) \rightarrow \IC$ can be written as 
\begin{eqnarray*}
	\psi(x)= \psi_j(x), & \mbox{where} & \psi_j \colon I_j \rightarrow \IC \quad \hbox{with} \quad I_j= \begin{cases} [0,a_j], & \mbox{if} \ j\in \cI, \\ [0,\infty), &\mbox{if} \ j\in \cE,  \end{cases}
\end{eqnarray*}
where with a slight abuse of notation ambiguities on the vertices are admitted.
Occasionally, also $\psi_j(x)=\psi_j(x_j)$  is written  where $x=x_j\in I_j$. Equipping each edge of the finite metric graph with
the one-dimensional Lebesgue measure 
one obtains a measure space, and then one defines 
\begin{equation*} 
\int_{\cG} \psi := \sum_{i\in \cI}  \int_{0}^{a_i}  \psi(x_i) \, dx_i +  \sum_{e\in \cE}  \int_{0}^{\infty}  \psi(x_e)\, dx_e, 
\end{equation*}
where $dx_{i}$ and $dx_e$ refers to integration with respect to the Lebesgue measure on the intervals $[0,a_i]$ and $[0,\infty)$, respectively. 
Given a finite metric graph $(\cG,\au)$, one then considers the Hilbert space
\begin{eqnarray*}
L^2(\cG,\au)=  \bigoplus_{j\in\cI\cup \cE} L^2(I_j;\IC) \quad \hbox{with} \quad \langle \psi,\varphi\rangle =\int_{\cG} \psi\, \overline{\varphi}.
\end{eqnarray*}    
Using the Sobolev spaces $H^1(I_j;\IC)$ and  $H^2(I_j;\IC)$ on each edge, one defines
\begin{align*}
	H^1(\cG,\au)&=  \bigoplus_{j\in\cI\cup \cE} H^1(I_j;\IC) \quad \hbox{with} \quad \langle \psi,\varphi\rangle_{H^1} =\langle \psi,\varphi\rangle + \langle \psi',\varphi'\rangle, \\
	H^2(\cG,\au)&=  \bigoplus_{j\in\cI\cup \cE} H^2(I_j;\IC) \quad \hbox{with} \quad \langle \psi,\varphi\rangle_{H^2} =\langle \psi,\varphi\rangle_{H^1} + \langle \psi'',\varphi''\rangle, \\
	H_0^2(\cG,\au)&=  \bigoplus_{j\in\cI\cup \cE} H^2_0(I_j;\IC),
\end{align*}    
where  $\psi'$ and $\psi''$ denote the edgewise defined first and second distributional derivatives,
\begin{align*}
	\left( \psi'\right)_j (x) 
= \tfrac{d}{dx}\psi_j(x) \quad \hbox{and} \quad	\left( \psi''\right)_j (x) 
= \tfrac{d^2}{dx^2}\psi_j(x) \quad \hbox{for }j\in \cE\cup \cI,  \quad x\in I_j,
\end{align*}
 and
$H_0^2(I_j;\IC)$ with $j\in \cE \cup \cI$ denotes the set of all $\psi_j\in H^2(I_j;\IC)$  with
\begin{align*}
	\psi_j(0)=0, \ \psi_j^{\prime}(0)=0,  \mbox{for}  \ j\in \cE,  \ 
	\psi_j(0)=0,  \ \psi_j^{\prime}(0)=0, \ \psi_j(a_j)=0, \  \psi_j^{\prime}(a_j)=0, \ \mbox{for} \ j\in \cI. 
\end{align*}
The notations are sometimes shortened to $L^2(\cG), H^1(\cG), H^2(\cG)$, and $H^2_0(\cG)$. These edgewise defined Sobolev spaces are suitable  for the study of general coupling boundary conditions on finite metric graphs, see Subsection~\ref{subsec:Laplacians} below.  
However, in the literature Soblev spaces are often defined more naturally with respect to the underlying metric structure of $(\cG,\au)$. For instance the Soblev space of order one is defined as the space  of functions $\psi \in H^1(\cG)$ such that $\psi$ is continuous  on the metric space $(\cG,\au)$, compare e.g. \cite[Definition 1.3.6]{BeKu_book} and  \cite[Section 3]{Mugnolo_whatisagraph}.

\subsection{Laplacians and boundary conditions}\label{subsec:Laplacians}
One defines maximal and minimal Laplace operators in $L^2(\cG,\au)$ by
\begin{align*}
\Delta^{\max} \psi &=\psi'' \hbox{ with } \Dom(\Delta^{\max})=H^2(\cG,\au) \quad \hbox{and}\quad \\
\Delta^{\min} \psi &=\psi'' \hbox{ with } \Dom(\Delta^{\min})=H_0^2(\cG,\au). 
\end{align*}
It is known that the operator $\Delta^{\min}$ is a closed symmetric operator with deficiency indices $(d,d)$, where one sets
\begin{equation}\label{d-notation}
d=\abs{\cE}+ 2\abs{\cI},
\end{equation}
and its Hilbert space adjoint is $(\Delta^{\min})^{\ast}=\Delta^{\max}$,
see e.g. \cite[Section~4.8]{Exnerbook}. 
The scope here are general realizations $\Delta'$ of the Laplacian on metric graphs with
\begin{equation}\label{extensions}
\Delta^{\min} \subset \Delta' \subset \Delta^{\max}.
\end{equation}
These realizations $\Delta'$ 
can be discussed in terms of boundary or matching conditions imposed at the endpoints of the edges. To this end, one defines the trace vectors
\begin{eqnarray*}
	\underline{\psi}= \begin{bmatrix} \{\psi_{e}(0)\}_{e\in\cE} \\ 
		\{\psi_{i}(0)\}_{i\in\cI} \\
		\{\psi_{i}(a_i)\}_{i\in\cI}
	\end{bmatrix}, \quad 
	\underline{\psi}^{\prime}= \begin{bmatrix} \{\psi_{e}^{\prime}(0)\}_{e\in\cE} \\ 
		\{\psi_{i}^{\prime}(0)\}_{i\in\cI} \\
		\{-\psi_{i}^{\prime}(a_i)\}_{i\in\cI} 
	\end{bmatrix}, \quad \mbox{and} \quad
[\psi]= \begin{bmatrix}
	\underline{\psi} \\ \underline{\psi^{\prime}}\end{bmatrix}
\quad \hbox{for } \psi\in H^2(\cG,\au),
\end{eqnarray*}
and one introduces the auxiliary Hilbert space
\begin{equation*}
\cK \equiv \cK(\cE, \cI) = \cK_{\cE}  \oplus \cK_{\cI}^- \oplus \cK_{\cI}^+
\end{equation*}
with $\cK_{\cE} = \IC^{\abs{\cE}}$ and $\cK_{\cI}^{\mp} = \IC^{\abs{\cI}}$. Then for $\psi\in H^2(\cG,\au)$ one has $\underline{\psi}, \underline{\psi}^{\prime}\in \cK$, and $[\psi]\in \cK^2$. 
Since the quotient space $\Dom (\Delta^{\max})/ \Dom (\Delta^{\min})\cong \cK^2$, any realization $\Delta'$ satisfying~\eqref{extensions}  
is associated with a subspace $\cM\subset \cK^2$ of the space of boundary values such that 
\begin{equation*}
\Delta'=\Delta(\cM), \quad \hbox{where} \quad \Delta(\cM)\psi = \psi'' \hbox{ and } \Dom(\Delta(\cM))= \{ \psi \in H^2(\cG,\au) \colon [\psi] \in \cM   \}.
\end{equation*}

\section{Boundary conditions and Cayley transforms}\label{Sec:bc}

\subsection{Parametrization of boundary conditions}
Recall that 
there are various ways to parametrise the subspaces $\cM\subset \cK^2$, and for the self-adjoint case some have been summarized in \cite[Section 3]{HKS2015}, see also \cite[Theorem 1.4.4]{BeKu_book}. The number of linearly independent boundary conditions imposed is equal to $\dim \cM$.  In particular if $\dim \cM \geq d$, for the deficiency index $d$ given by \eqref{d-notation}, then there exist  $A,B\in \Hom(\cK)$ such that
\begin{equation*}
\cM  =\cM(A,B)=\Ker (A, \, B)  \quad \hbox{and} \quad \Delta(A,B)=\Delta(\cM(A,B)),
\end{equation*}
where
\begin{eqnarray*}
	(A, \, B)\colon \cK^2 \rightarrow \cK, & (A, \, B) (\chi_1, \chi_2)^T = A\chi_1 + B \chi_2 & \mbox{for } \chi_1,\chi_2\in\cK.
\end{eqnarray*} 
An equivalent description of $\Delta(A,B)$ given
by the 
linear 
boundary conditions defined by $A,B \in \Hom(\cK)$ is then
\begin{equation}\label{AB}
\Delta(A,B)\psi = \psi'' \hbox{ and } \Dom(\Delta(A,B))=\{\psi\in H^2(\cG,\au)\colon A \underline{\psi} + B \underline{\psi}^{\prime} = 0 \}.
\end{equation}
Note that the parametrisation by the matrices $A$ and $B$ is not unique. 
Indeed, operators $\Delta(A,B)$ and $\Delta(A^{\prime},B^{\prime})$ agree if and only if the corresponding spaces $\cM(A,B)$ and $\cM(A^{\prime},B^{\prime})$ agree. Boundary conditions defined by $A,B\in \Hom(\cK)$ and $A^{\prime},B^{\prime}\in \Hom(\cK)$ are called \emph{equivalent} if $$\cM(A,B)=\cM(A^{\prime},B^{\prime}),$$
compare \cite[Definition 2.2]{KSP2008} where in addition a rank condition is assumed. Notice that the boundary conditions 
are equivalent if there exists 
an invertible operator $C\in GL(\cK)$ such that simultaneously 
\begin{eqnarray*}
	A^{\prime} = CA & \mbox{and} & B^{\prime} = CB.
\end{eqnarray*}

\subsection{Regular boundary conditions and the quasi-Weierstrass normal form}\label{subsect:quasiWeierstrass}

 Operator pencils  $ A-\lambda B$ with $\lambda \in \IC$ appear in many contexts such as generalized eigenvalue problems or  differential algebraic equations, see e.g. \cite{BIT2012, Markus_book} and the references therein. A basic regularity assumption is the following, cf. \cite{BIT2012}.
\begin{definition}[Regularity of operator pencils]
	For  $d\in \IN$ let $A,B\in \IC^{d\times d}$.	Then the operator pencil $A-\lambda B$, $\lambda\in \IC$, is called {\em regular} if $\det (A + \lambda B) \neq 0$ for some $\lambda\in \IC$ and {\em irregular} otherwise.   
\end{definition} 

For matrices $A,B$ defining a regular operator pencil, one can define the Cayley transform
\begin{equation*}
\mathfrak{S}(k,A,B):= - \left(A+ik B \right)^{-1}\left(A-ik B \right)
\end{equation*}
for all $k\in \IC$ except an at most finite set, because $\det(A-ik B)$ is a non-vanishing polynomial of degree at most $d$. 
In fact the Cayley transform has been the starting point for the introduction of a notion of \emph{regular boundary conditions} in \cite[Definition 3.2]{HKS2015} not knowing that the analogous notion of \emph{regular operator pencils} already had  existed in a different context. 

The Cayley transform is an objects which appears in many instances of the analysis of Laplacians on graphs. For example for self-adjoint Laplacians on star graphs with $\cI=\emptyset$, $\mathfrak{S}(k,A,B)$ for $k>0$ has the interpretation of a scattering matrix, and for general finite metric graphs it is referred to as local scattering matrix, see e.g. \cite{KS1999}, and it also appears naturally in the analysis of non-self-adjoint Laplacians when one formulates the eigenvalue equation or the Green's function, see \cite{KSP2008}.

\begin{definition}[Regularity of boundary condition, cf. Definition 3.2 in \cite{HKS2015}]\label{def:irregularbc} Let $A,B\in \Hom(\cK)$  with $\Rank (A,\, B)=d$. Then the boundary conditions defined by $A,B$ are called {\em regular} if	 $A-\lambda B$ is regular, and {\em irregular} otherwise. 
\end{definition} 
Note that equivalent regular boundary conditions have the same Cayley transform, and in turn the boundary conditions can be recovered  via
\begin{eqnarray*}
A_{\mathfrak{S}}:= - \frac{1}{2} \left(\mathfrak{S}(k,A,B) -\mathds{1}\right) & \mbox{and} & B_{\mathfrak{S}}:=  \frac{1}{2ik} \left(\mathfrak{S}(k,A,B) +\mathds{1}\right),
\end{eqnarray*}
compare e.g. \cite[Theorem 1.4.4]{BeKu_book} for the self-adjoint case and \cite[Proposition 3.7.]{HKS2015} for the more general case.

\begin{remark}[The rank condition in the two notions of regularity]
Regularity of the operator pencil $A-\lambda B$ already implies that $\Rank (A\, B)=d$, cf. \cite[Proposition 4.2]{HKS2015} for the case $\cI=\emptyset$.  
		The notion of regularity of boundary conditions deals with a subclass of the boundary conditions which are composed of $d$ linearly independent boundary conditions, \ie , $\Rank (A\, B)=\dim \cM(A,B)=d$. This slight redundancy in Definition~\ref{def:irregularbc} seems to be justified by the  fact that $\dim \cM < d$ leads to an under-determined  while $\dim \cM > d$ leads to an over-determined system. This is reflected by the spectrum since $\sigma(\Delta(\cM))=\IC$ if $\dim \cM \neq d$, see \cite[Proposition 4.2]{HKS2015}. The implication of the rank condition becomes more apparent in the irregular cases.
		Definition~\ref{def:irregularbc} requires $\dim \cM=d$ for irregular boundary conditions, while the cases with $\dim \cM\neq d$ are out of scope here.   However, having $d$ linearly independent boundary conditions, \ie , $\dim \cM=d$, is not sufficient for many spectral properties since irregular boundary conditions can lead to wild spectral features with empty spectrum, empty resolvent set or cases in between, see \cite[Section 3.4]{HKS2015} where some examples are discussed. In contrast, irregular operator pencils are all pencils with $\det (A+\lambda B)=0$ for all $\lambda\in \IC$ irrespective of the rank of $\Rank (A\, B)$.
\end{remark}

The quasi-Weierstrass normal form for regular operator pencils  $A-\lambda B$ has been introduced in  \cite{BIT2012} in the context of  differential algebraic equations. Its proof is based on Wong sequences of subspaces. It is the key ingredient to the analysis presented here, and it reveals many properties of the Cayley transform such as its poles and its asymptotic behaviour.
\begin{proposition}[Quasi-Weierstrass normal form, \cf~Theorem 2.6 in \cite{BIT2012}]\label{prop:quasiweierstrass}
	Let $A,B\in \IC^{d\times d}$ and let the operator pencil $A-\lambda B$, $\lambda\in \IC$, be regular. Then there exists $0\leq m\leq d$ and invertible matrices $F,G$ such that
	\begin{align*}
	A= F\begin{bmatrix}
	L & 0 \\ 0 & \mathds{1}_{\IC^{d-m}}
	\end{bmatrix}G \quad \hbox{and} \quad 
		B= F\begin{bmatrix}
	\mathds{1}_{\IC^{d}} & 0 \\ 0 & N_B
	\end{bmatrix}G, 
	\end{align*} 
	where $N_B\in \IC^{(d-m)\times (d-m)}$ is nilpotent and $L\in \IC^{m\times m}$. These matrices are in general not uniquely determined. 
\end{proposition}
Recall that for self-adjoint boundary conditions one has a unique parametrization, \cf~\cite[Theorem 6]{Kuchment2004}, by
\begin{align}\label{eq:bc_LP}
A= L+P, \quad B=P^\perp=(\mathds{1}-P) \quad \hbox{with} \quad P^\ast=P=P^2,  \quad P^\perp L = L P^\perp,
\end{align}
and $L=L^\ast$, 
then the boundary conditions decompose into
\begin{align*}
L\underline{\psi}  + P^\perp \underline{\psi}'=0 \quad \hbox{and} \quad P \underline{\psi}=0,
\end{align*} 
that is, into a Robin/Neumann part and a Dirichlet part.
 In particular, $-\Delta(A,B)$ is
associated with the closed symmetric form $\delta_{P,L}$ 
defined by 
\begin{equation}\label{deltaPL}
\delta_{P,L}[\psi]
=\int_{\cG} \abs{ \psi^{\prime}}^2 
- \langle L P^{\perp}\underline{\psi},P^{\perp}\underline{\psi}\rangle_{\cK}, \quad
\dom(\delta_{P,L}) 
= \{ \varphi\in H^1(\cG,\au) \mid P\underline{\varphi}=0\}.
\end{equation}
For general possibly non-selfadjoint $L$ with $P^\perp L = L P^\perp$ the form $\delta_{P,L}$ is sectorial, though in general non-symmetric, and it has been investigated in detail in the context of semigroups on networks, see e.g. \cite[Chapter 6 and 7]{Mug_book} and the references therein, and compare also \cite{Hussein2014} where a characterization is given when general non-self-adjoint boundary conditions defined by matrices $A,B$ are of this form.   
From Proposition~\ref{prop:quasiweierstrass} it follows in particular that regular boundary conditions in general are not of the form \eqref{eq:bc_LP} needed to define $\delta_{P,L}$ as in \eqref{deltaPL}. However, at least if $N_B= 0$, there are equivalent boundary conditions which are similar to such sectorial boundary conditions, \ie , 
\begin{align*}
A= G^{-1}(L+P)G, \quad B=G^{-1} P^\perp G,  \quad \hbox{with} \quad P^\ast=P=P^2, \quad P^\perp=(\mathds{1}-P), \quad P^\perp L = L P^\perp,
\end{align*}
for $G$ invertible and $P$ an orthogonal projection
 This motivates the following definition, and it  turns out, see Theorem~\ref{thm:semigroup} below,  that this notion suggesting a relation to sectorial operators is indeed justified. 
\begin{definition}[Quasi-sectorial boundary conditions]\label{def:quasisec}
	Regular boundary conditions defined by $A,B\in \Hom(\cK)$ are called \emph{quasi-sectorial} if $N_B= 0$ in the quasi-Weierstrass normal form.
\end{definition}
Here, the condition that $N_B=0$  includes the case of $m=d$, and $m=d-1$ already implies $N_B=0$.

For a matrix $M \in \IC^{n\times n}$ and $\lambda\in \sigma(M)$ denote by $\gamma_M(\lambda)$ the maximal length of  Jordan chains corresponding to the eigenvalue $\lambda$ or equivalently the size of the largest Jordan block corresponding to the eigenvalue $\lambda$.
\begin{lemma}[Poles of the Cayley transform]\label{lemma:poles}
		Let $A,B\in \Hom(\cK)$ define regular boundary conditions, where $m$, $L$ and $N_B$ are given by the quasi-Weierstrass normal form  in Proposition~\ref{prop:quasiweierstrass}. Then the  function 
		\begin{align*}
		\mathfrak{S}(\cdot,A,B)\colon \IC \rightarrow \Hom(\cK), \quad k\mapsto \mathfrak{S}(k,A,B)
		\end{align*}
		is meromorphic with at most $m+1$ but not more than $d$ different poles, and each pole falls under one of the following cases with poles of order zero being removable singularities:
		\begin{enumerate}[(a)]
			\item If $k\in \sigma(iL)\setminus\{0\}$, then $k$ is a pole of order $\gamma_L(k)$; 
			\item If $0\in \sigma(iL)$, then $0$  is a singularity of order $\max\{\gamma_L(0)-1, \gamma_{N_B}(0)-1, 0\}$;
			\item If $0\notin \sigma(iL)$, then $0$  is a singularity of order $\max \{\gamma_{N_B}(0)-1,0\}$.
		\end{enumerate}
\end{lemma}
\begin{proof}
	By the quasi-Weierstrass normal form one has  for regular boundary conditions
	\begin{align}\label{eq:S_blockmatrix}
	\mathfrak{S}(k,A,B)= G^{-1}\begin{bmatrix}
	\mathfrak{S}(k,L,\mathds{1}) & 0 \\ 0 & \mathfrak{S}(k,\mathds{1}, N_B)\end{bmatrix} G.
	\end{align}
		Now, without loss of generality one can assume that $L$ and $N_B$ have the Jordan normal form, where the necessary similarity transforms can be incorporated into $G^{-1}$ and $G$.
	Recall that for a nilpotent matrix $N$ with $N^{n}=0$ and $N^{n-1}\neq 0$, \ie , $\gamma_N(0)=n$, one has that
	\begin{align*}
	(\lambda \mathds{1} +N)^{-1} =  \tfrac{1}{\lambda}\left(\mathds{1} + \tfrac{N}{-\lambda} + \ldots + \tfrac{N^{n-1}}{(-\lambda)^{n-1}}\right) \quad \hbox{for } \lambda \neq 0.
	\end{align*}
	Hence if $n\in \IN$ is such that $N_B^{n}=0$ and $N_B^{n-1}\neq 0$, then
	\begin{align}\begin{split}\label{eq:SN}
	\mathfrak{S}(\mathds{1}, N_B;k) &= -(\mathds{1} + ik N_B)^{-1}(\mathds{1} + ik N_B - 2ik N_B ) \\
	&= -\mathds{1} + 2ik(\mathds{1} + ik N_B)^{-1} N_B \\
	&=
	-\mathds{1} + 2ik N_B + 2 (ik)^2(-1)N_B^2 + \ldots +  2(ik)^{n-1} (-1)^{n-2}N_B^{n-1} \quad \hbox{for } k\neq 0,
	\end{split}
	\end{align}
	and if $N_B=0$, then
		\begin{align}\label{eq:SN0}
	\mathfrak{S}(\mathds{1}, N_B;k) = -\mathds{1} \quad \hbox{for all } k\in \IC. 
	\end{align}
		If $L= \lambda \mathds{1} + N$ for $\lambda\in \IC$ with $n\in \IN_0$ such that $N^{n}=0$ and $N^{n-1}\neq 0$, then for $\lambda\neq 0$
	\begin{align}\begin{split}\label{eq:S_NL}
	\mathfrak{S}(L, \mathds{1};k) &= -((\lambda +ik)\mathds{1} +  N)^{-1}((\lambda +ik)\mathds{1} +  N - 2ik \mathds{1}) \\
	&= -\mathds{1} + 2ik((\lambda +ik)\mathds{1} +  N)^{-1} \\
	&=-\mathds{1} + \frac{2ik}{\lambda +ik}  \left(\mathds{1} + \frac{N}{-(\lambda +ik)} + \ldots+ \frac{N^{n-1}}{(-1)^{n-1}(\lambda +ik)^{n-1}}  \right), \quad \lambda +ik\neq 0,
	\end{split}
	\end{align}
	and for $\lambda=0$
		\begin{align}\label{eq:S_N0}
	\mathfrak{S}(L, \mathds{1};k) =
	-\mathds{1} + 2  \left(\mathds{1} + \frac{N}{- ik} + \ldots+ \frac{N^{n-1}}{(-1)^{n-1}(ik)^{n-1}}  \right)\quad \hbox{for } ik\neq 0.
	\end{align}
		If $L= \lambda \mathds{1}$ for $\lambda\in \IC$, this simplifies to become 
	\begin{align}\label{eq:SL0}
	\mathfrak{S}(L, \mathds{1};k) = -\frac{(\lambda -ik)}{(\lambda +ik)}\mathds{1}, \quad \hbox{if } \lambda\neq 0\quad \hbox{and} \quad 	\mathfrak{S}(L, \mathds{1};k)= \mathds{1}, \quad \hbox{if } \lambda= 0.
 	\end{align}
 	
So, in case of (a), that is, $k\in \sigma(iL)\setminus \{0\}$, then by \eqref{eq:S_NL} each Jordan block $J_{L,\lambda}$ of $L$ to the eigenvalue $\lambda=-ik$ contributes a pole at $k$ of order $\gamma_{J_{L,\lambda}}(\lambda)$, and the Jordan blocks of $L$ to different eigenvalues and also $\mathfrak{S}(\mathds{1}, N_B;k)$ do not have a pole at $k$ by \eqref{eq:SN} and \eqref{eq:SN0} which proves (a).

In  case (b), the Jordan blocks $J_{L,0}$ of $L$ to $0$ contribute poles of order $\max\{0,\gamma_{J_{k,0}}(0)-1\}$ by \eqref{eq:S_N0}. This includes the case where the geometric and algebraic multiplicity of the eigenvalue $0$ of $L$ agree since in this case by \eqref{eq:SL0} there is no pole. The Jordan blocks of $J_{N_B,0}$  	 
contribute poles at $0$ of order $\max \{0, \gamma_{N_B}(0)-1\}$ by \eqref{eq:SN}, which includes the case $N_B=0$ by \eqref{eq:SN0}. 

In the case (c), only   the Jordan blocks of $J_{N_B,0}$  	 
contribute poles at $0$ of order $\max \{0, \gamma_{N_B}(0)-1\}$.
 	
Moreover, in \eqref{eq:S_blockmatrix} the block	$\mathfrak{S}(k,L,\mathds{1})$ can have at most $m$ different poles since $L$ can have at most $m$ pairwise disjoint non-zero eigenvalues, and  $\mathfrak{S}(k,\mathds{1}, N_B)$ by \eqref{eq:SN} and \eqref{eq:SN0} can have at most one pole at zero. So the set of  poles consists of at most $m+1$ but not more than $d$ points. Outside its poles the function $\mathfrak{S}(\cdot,A,B)$ is holomorphic.
\end{proof}

\begin{lemma}[Uniform boundedness of the Cayley transform]\label{lemma:Nilpotent}
Let $A,B$ define regular boundary conditions. Then 
the following statements are equivalent:
\begin{enumerate}[(a)] 
	\item The boundary conditions defined by $A,B$ are quasi-sectorial.
	\item For fixed $\varepsilon>0$, the function 
	\begin{align*}
 \IC\setminus \{B_{\varepsilon}(p_1), \ldots, B_{\varepsilon}(p_l)\} \rightarrow \Hom(\cK), \quad k \mapsto \mathfrak{S}(k,A,B),
	\end{align*}
	where $p_1, \ldots, p_l$ are the not removable poles of $\mathfrak{S}(\cdot, A, B)$, is uniformly bounded.
\end{enumerate} 
\end{lemma}

\begin{proof}
	For regular boundary conditions  the quasi-Weierstrass normal form  is available, and by \eqref{eq:S_blockmatrix} it is sufficient to consider  
	$\mathfrak{S}(k,L,\mathds{1})$ and  $\mathfrak{S}(k,\mathds{1}, N_B)$. 
	
	For the first block one has
		\begin{align*}
	\mathfrak{S}(k,L,\mathds{1}) = -(\tfrac{1}{ik}L + \mathds{1})^{-1}(\tfrac{1}{ik} L - \mathds{1}),
	\end{align*}
	and hence using the Neumann series
	\begin{align*}
	\norm{(L + ik\mathds{1})^{-1}(L - ik\mathds{1})} \leq \frac{2}{1-\norm{L}/k} \leq 4 \quad \hbox{for } |k|\geq 2\norm{L},
	\end{align*}
	and due to continuity $\mathfrak{S}(\cdot,L,\mathds{1})$ is bounded on $\overline{B_{2\norm{L}}(0)}\setminus \{B_{\varepsilon}(p_1), \ldots, B_{\varepsilon}(p_1)\}$.
	
	Considering the second block,  if there exists a cyclic vector $y$ with  $N_By\neq 0$ and $N_B^2y=0$ -- which is equivalent to $N_B\neq 0$ -- one has by \eqref{eq:SN} that $\mathfrak{S}(k,\mathds{1}, N_B)y=-y - 2ik Ny$,
	and hence for some $c>0$
	\begin{align*}
	\norm{\mathfrak{S}(A,B;k)} \geq c(1+|k|) \quad k\in \IC,
	\end{align*}
	while for $m-d>0$ and $N_B= 0$ one has $\mathfrak{S}(k,\mathds{1}, N_B)=-\mathds{1}$. So, the uniform boundedness is characterized by $N_B$.
\end{proof}

\begin{remark}
	In the case $d=1$, \ie , $|\cE|=1$ and $|\cI|=0$, all boundary conditions with $\dim \cM=1$ are of the form
	\begin{align}\label{eq:d=1}
	\psi(0) =0 \quad \hbox{or}\quad \alpha\cdot \psi(0) + \psi'(0) =0, \quad \alpha \in \IC,
	\end{align}
	and so all boundary conditions are regular and define operators associated with forms as in \eqref{deltaPL}  of the type $\delta_{1,0}$ and $\delta_{\alpha,1}$, respectively.
	
	 Irregular and non-quasi-sectorial boundary conditions occur only for $d\geq 2$. These cases appear in the analysis of Sturm-Liouville operators on a compact interval with boundary conditions coupling both endpoints, \ie , $|\cI|=1$ and $\cE=\emptyset$ where $d=2$. Problems with such couplings are studied for instance in context of the classical Birkhoff-Tamarkin theory discussed below in Subsection~\ref{subsec:birkhoff}. The case $d=2$ with $\cI=\emptyset$ and $|\cE|=2$, \ie , the case of Laplacians with point interaction on the real line, is discussed in detail by Mugnolo and the author in \cite{HusMug2019}, yet without the theory of the quasi-Weierstrass normal form at hand. 
	
	Second order elliptic boundary value problems in domains in $\IR^n$ for $n\geq 2$ are often
	localized and then drawn back to the case of the half-space $\IR^n_+=\IR^{n-1}\times [0,\infty)$. Then  by a partial Fourier transform the problem is reduced to a one dimensional problem where mostly boundary conditions of the form  \eqref{eq:d=1} are considered. Thus the cases where $d\geq 2$ are usually not studied in the theory of partial differential equations. However, for vector valued problems, these more general conditions might be relevant as well.
\end{remark}

\begin{example}[$\cP\cT$-symmetric point interaction]\label{ex1_quasi_weiserstrass}
	Consider the boundary conditions defined by 
	\begin{eqnarray*}
		A_{\tau}=\begin{bmatrix} 1 & -e^{i\tau} \\ 0 & 0\end{bmatrix} &\mbox{and} & B_{\tau}=\begin{bmatrix}  0 & 0 \\ 1 & e^{-i\tau} \end{bmatrix} \quad \hbox{for} \quad \tau \in [0,\pi/2].
	\end{eqnarray*} 
	Let $\cG$ be a graph consisting of two external edges $\cE=\{e_1,e_2\}$ and one vertex $\partial(e_1)=\partial(e_2)$.
	Identifying the graph with the real line 
	and the vertex with zero,
	the boundary conditions correspond to
	$\psi(0+) = e^{i\tau} \psi(0-)$
	and 
	$\psi^{\prime}(0+) = e^{-i\tau} \psi^{\prime}(0-)$.
	This example is included in the study of 
	$\mathcal{PT}$-symmetric point interactions in~\cite{Kurasov02} 
	and was further investigated in  \cite{Kuzel05} and \cite{Petr, PetrPhD}.
	
	For $\tau=\pi/2$, these boundary conditions are irregular, and	
	for $\tau\in[0,\pi/2)$ they define regular boundary conditions with $k$-independent
	Cayley transform
	\begin{eqnarray*}
		\mathfrak{S}(A_{\tau},B_{\tau},k)= \frac{1}{\cos(\tau)} \begin{bmatrix} i \sin(\tau) & 1 \\ 1 & -i\sin(\tau) \end{bmatrix}.
	\end{eqnarray*}
	Moreover, one has
	\begin{eqnarray*}
		\mathfrak{S}(A_{\tau},B_{\tau},k)= \frac{-1}{2\cos(\tau)}\begin{bmatrix} 1 & 1 \\ e^{-i\tau} & -e^{i\tau} \end{bmatrix} 
		\begin{bmatrix} 1 & 0 \\ 0 & -1 \end{bmatrix} \begin{bmatrix} -e^{i\tau} & -1 \\ -e^{-i\tau} & 1 \end{bmatrix}\quad \hbox{for } \tau\in[0,\pi/2).
	\end{eqnarray*}
	Therefore for the equivalent boundary conditions given by
	\begin{eqnarray*}
		A_{\mathfrak{S}}= - \frac{1}{2} \left(\mathfrak{S} -\mathds{1}\right) & \mbox{and} & B_{\mathfrak{S}}=  \frac{1}{2} \left(\mathfrak{S} +\mathds{1}\right),
	\end{eqnarray*}
	one has the quasi-Weierstrass normal form
	\begin{eqnarray*}
		A_{\mathfrak{S}}= G_{\tau}^{-1}
		\begin{bmatrix} 0 & 0 \\ 0 & 1 \end{bmatrix}
		G_{\tau} \hbox{ and }
		B_{\mathfrak{S}}= G_{\tau}^{-1}
		\begin{bmatrix} 1 & 0 \\ 0 & 0 \end{bmatrix}
		G_{\tau} \hbox{ with } G_{\tau}=\frac{i}{\sqrt{2\cos(\tau)}}\begin{bmatrix} -e^{i\tau} & -1 \\ -e^{-i\tau} & 1 \end{bmatrix}.
	\end{eqnarray*}
	In particular, this is an example of quasi-sectorial boundary conditions which are not in the form ~\eqref{eq:bc_LP}, compare also \cite[Example 3.5 ff.]{HKS2015}. 
\end{example}

\begin{example}[Intermediate boundary conditions]\label{ex:intermediate}
	Consider regular boundary conditions defined by 
	\begin{eqnarray*}
		A=\begin{bmatrix}1 & 0 \\ 0 & 1  \end{bmatrix}  & \mbox{and} & B=\begin{bmatrix} 0 & 0 \\ -1 & 0  \end{bmatrix}.
	\end{eqnarray*}
	These boundary conditions are already in quasi-Weierstrass normal form with $m=0$ and nilpotent $N_B=B$, i.e., $B^2=0$.
	On the interval $[0,1]$ 
	this corresponds to
	$\psi(0)=0$ and $\psi(1)-\psi^{\prime}(0)=0$, and 
	since $\dim\cM(A,B)=2$ and $\det(A+ikB)=1$  for all $k\in\IC$, \ie, $A,B$ are regular with Cayley transform
	\begin{align*}
	\mathfrak{S}(k;A,B) = -\begin{bmatrix}
	1 & 0 \\ 2ik & 1
	\end{bmatrix}, \quad k\in \IC.
	\end{align*} 
	This example will be discussed further in Subsection~\ref{exDS} as a typical case for regular but non-quasi-sectorial boundary conditions.
	It can be found in \cite[p.383]{Birkhoff2} as well as in \cite[Ex.~XIX.6(d)]{DSIII}, and it has  been discussed also in \cite[Ex.~3.6]{HKS2015}.
	There, it is an example of so-called \emph{intermediate boundary conditions}.
\end{example}

\subsection{Topology of boundary condition}
The complex Grassmanian $\Gr(n,m)$ is the set of all $n$-dimensional complex subspaces of $\IC^m$. Using unitary groups, it can be identified with
\begin{align*}
\Gr(n,m) = U(m)/(U(n)\times U(m-n))
\end{align*}
which induces also a differentiable structure on $\Gr(n,m)$ which makes it a compact manifold with real dimension $\dim \Gr(n,m) = n(m - n).$ 
Here, each boundary condition defined by $\cM=\cM(A,B)$ can be identified with a point $\cM\in \Gr(d,2d)$. A metric on $\Gr(d,2d)$ is defined by
\begin{align*}
d(\cM, \cM')= \norm{P_{\cM}- P_{\cM'}}, \quad \cM,\cM'\in \Gr(d,2d),
\end{align*}
where $P_{\cM}$ and $P_{\cM'}$ are the orthogonal projections in $\cK^2$ onto $\cM$ and $\cM'$ respectively. 
\begin{remark}
	It is a classical result from extension theory that the set of self-adjoint boundary conditions can be parametrized by $U(d)$ which is  a submanifold of $\Gr(d,2d)$ with real dimension $d^2$ while the manifold of all boundary conditions with $\dim  \cM=d$ is $\Gr(d,2d)\cong U(2d)/(U(d)\times U(d))$ which has  real dimension $2d^2$.
\end{remark}
Parts of the following lemma are covered by  \cite[Lemma 3.2]{KSP2008}. It allows one to draw back convergence in $\Gr(d,2d)$ to convergence in $\Hom(\cK)$. 
\begin{lemma}[Convergence of boundary conditions]\label{lemma:projmperp} Let $A,B\in\Hom(\cK)$.
	\begin{enumerate}
		\item[(a)] Then	one has that $\dim \cM(A,B)=d$ if and only if 
		$A A^\ast  + BB^\ast$ is invertible.
		\item[(b)]  For $\dim \cM(A,B)=d$  
		the orthogonal projection in $\cK^2$ on $\cM(A,B)^\perp$ is given by 
		\begin{align*}
		P_{\cM(A,B)^\perp}= \begin{bmatrix}
		A^\ast \\ B^\ast
		\end{bmatrix} (AA^\ast + B B^\ast)^{-1} \begin{bmatrix}
		A &  B
		\end{bmatrix}.
		\end{align*}
		\item[(c)] If $(A_n)_{n\in \IN}$ and $(B_n)_{n\in \IN}$ are sequences in $\Hom(\cK)$ with $\dim \cM(A_n,  B_n)=d$ for all $n\in \IN$ and  $A_n\to A$ and $B_n\to B$ as $n\to \infty$ in $\Hom(\cK)$, where $\dim \cM(A,  B)=d$,  then  $P_{\cM(A_n,B_n)}\to P_{\cM(A,B)}$ as $n\to \infty$.
	\end{enumerate}
\end{lemma}

\begin{proof}
	Lemma 3.2 in \cite{KSP2008} deals with the necessary condition in (a). For completeness, the proof is repeated here. First, note that
	\begin{align*}
	\cM(A,B)^\perp=\Ker (A,\,B)^\perp= \Ran \begin{bmatrix}
	A^\ast \\ B^\ast
	\end{bmatrix}\subset \cK^2,
	\end{align*}
	and since  $\dim \cK^2=2d$, one has that $\dim\cM(A,B)=d$  if and only if  $\dim \cM(A,B)^\perp=d$. By the dimension formula
	\begin{align*}
	\dim \Ran \begin{bmatrix}
	A^\ast \\ B^\ast
	\end{bmatrix} = d-\dim \Ker \begin{bmatrix} A^\ast \\  B^\ast\end{bmatrix} ,
	\end{align*} 
	and hence $\dim\cM(A,B)=d$ if and only if	$$\Ker \begin{bmatrix} A^\ast \\  B^\ast\end{bmatrix}= \Ker A^\ast \cap \Ker B^\ast=\{0\},$$
	which is equivalent to the invertibility of $A A^\ast  + BB^\ast$. This can be seen for instance from  the form defined by 
	$\xi \mapsto \langle B^\ast\xi,B^\ast\xi \rangle_{\cK} +  \langle A^\ast\xi,A^\ast\xi \rangle_{\cK}$.
	
	For (b) it is now straight forward to check, that  
	\begin{align*}
	\Ran P_{\cM(A,B)^\perp}=\cM(A,B)^\perp, \quad P_{\cM(A,B)^\perp}^2=P_{\cM(A,B)^\perp}, \quad \hbox{and} \quad P_{\cM(A,B)^\perp}=P_{\cM(A,B)^\perp}^\ast.
	\end{align*}
	
The convergence in (c) follows directly from the formula for the projection since $P_{\cM(A,B)}=\mathds{1}-P_{\cM(A,B)^\perp}$ and the continuity of the composition and inversion operators. 
\end{proof}

\begin{corollary}[Characterization of regular boundary conditions]\label{cor:regboudnary conditions}
	The boundary conditions defined by $A,B\in \Hom(\cK)$ are regular if and only if both 
	\begin{align*}
	AA^\ast + BB^\ast\quad \hbox{and} \quad A^\ast A + B^\ast B
	\end{align*}
	are invertible. 
\end{corollary}	
\begin{proof}
	In \cite[Proposition 3.3]{HKS2015} it has been shown that $A,B$ define  regular boundary conditions if and only if $\dim(\cM(A,B))=d$ and $\Ker A \cap \Ker B\neq \{0\}$, where by 
	Lemma~\ref{lemma:projmperp} the condition $\dim(\cM(A,B))=d$ is equivalent to the invertibility of
	$AA^\ast + BB^\ast$. Moreover, $\Ker A \cap \Ker B\neq \{0\}$ is equivalent to the invertibility of
	$A^\ast A + B^\ast B$.
\end{proof}

\begin{remark}
To illustrate that there is no implication between the two conditions in Corollary~\ref{cor:regboudnary conditions}, consider $A,B$ which define irregular boundary conditions, \ie , $\dim(\cM(A,B))=d$, or equivalently by Lemma~\ref{lemma:projmperp}  $\Ker A^\ast\cap \Ker B^\ast= \{0\}$, and $\Ker A\cap \Ker B\neq \{0\}$, for example 
\begin{align*}
A=\begin{bmatrix}
1 & 0 \\ 0 & 0
\end{bmatrix} \quad \hbox{and} \quad B=\begin{bmatrix}
0 & 0 \\ 1 & 0
\end{bmatrix}.
\end{align*}
The other way round considering the boundary conditions defined by $A^\ast, B^\ast$, one observes that since $A,B$ define irregular boundary conditions the pencil $A-\lambda B$ is irregular as well, and it follows by taking adjoints that  $\det(A^\ast-\overline{\lambda} B^\ast)=0$ for all $\lambda\in \IC$, and hence also the pencil $A^\ast-\lambda B^\ast$ is irregular, and  
$\Ker A^\ast \cap \Ker B^\ast = \{0\}$, 
but these boundary conditions violate the rank condition, i.e., 
by Lemma~\ref{lemma:projmperp} $\Ker A \cap \Ker B \neq \{0\}$.
Note that in general  $\Delta(A,B)^\ast \neq \Delta(A^\ast, B^\ast)$ while equality holds for boundary conditions of the form \eqref{eq:bc_LP}, compare \cite[Subsections 3.6 and 3.7]{HKS2015}. Taking $A=B=0$, one sees that there are boundary conditions satisfying neither of both conditions.
\end{remark}

\begin{proposition}\label{prop:topology_bc}
	The set of regular boundary conditions is an open, connected and dense submanifold of $\Gr(d,2d)$. For $d\geq 2$ the set of quasi-sectorial boundary conditions is a dense subset, and neither open nor closed. 
\end{proposition}
\begin{proof}	
By Corollary~\ref{cor:regboudnary conditions} 
\begin{align*}
\IC^{d \times 2d}_{\hbox{reg}}:&= \{(A,\, B)\in \IC^{d \times 2d}\colon A,B \hbox{ define regular boundary conditions}   \} \\
&=\{(A,\, B)\in \IC^{d \times 2d}\colon |\det(AA^\ast + BB^\ast)|+ |\det(A^\ast A + B^\ast B)| \neq 0  \},
\end{align*}
and since $(A,\, B) \mapsto |\det(AA^\ast + BB^\ast)|+ |\det(A^\ast A + B^\ast B)|$ is continuous $\IC^{d \times 2d}_{\hbox{reg}}\subset \IC^{d \times 2d}$ is an open submanifold. Now, consider the quotient map
\begin{align*}
q\colon \IC^{d \times 2d}_{d}:= \{(A,B)\in \IC^{2d \times d}\colon \Rank(A,\,B)=d  \}  \rightarrow  \Gr(d,2d), \quad (A,\,B) \mapsto \cM(A,B),
\end{align*}
which identifies $q(A,B)=q(A',B')$ if $\cM(A,B)=\cM(A',B')$. By Lemma~\ref{lemma:projmperp} this is continuous, and in fact the topology on $\Gr(d,2d)$ coincides with the quotient topology induced by $q$.
Hence the set of regular boundary conditions $q(\IC^{d \times 2d}_{\hbox{reg}})\subset \Gr(d,2d)$ is open.

	Following the idea in the proof  \cite[Theorem 3.10]{HKS2015}, one can show that
	any boundary conditions $A,B$ with $\dim \cM(A,B)=d$ can be approximated by
	\begin{align*}
	A_\varepsilon :=A \quad \hbox{and} \quad B_\varepsilon := B + \varepsilon P_{\Ker B}, \quad \varepsilon>0,
	\end{align*}
	where $P_{\Ker B}$ is the orthogonal projection in $\cK$ to $\Ker B$. In particular $A_\varepsilon, B_\varepsilon$ define quasi-sectorial boundary conditions which implies the density of quasi-sectorial boundary conditions.   
	
	To prove that the set of regular boundary conditions in $\Gr(d,2d)$ is connected, consider for regular boundary conditions $A,B$ the  decomposition of $A$ with respect to $\Ker A$ and $(\Ker A)^\perp$ and define
	\begin{align*}
	A_{\varepsilon}:=\begin{bmatrix}
	P_{(\Ker A)^\perp} A P_{(\Ker A)^\perp} & 0 \\
	P_{\Ker A} A P_{(\Ker A)^\perp} & \varepsilon P_{\Ker A}
	\end{bmatrix}
	\quad \hbox{and}\quad B_{\varepsilon}:=(1-\varepsilon)B, \quad \varepsilon \in [0,1]
	\end{align*} 
where	$A_0=A$, $B_0=B$ and
	$A_\varepsilon$ is invertible for all $\varepsilon\in (0,1]$, hence $(A_\varepsilon, B_\varepsilon)$ is regular for all $\varepsilon\in [0,1]$. The  map 
	\begin{align*}
	\gamma\colon [0,1] \rightarrow \IC^{d \times 2d}_{d}, \quad \varepsilon \mapsto (A_\varepsilon, B_\varepsilon),\quad \hbox{hence also} \quad q\circ \gamma\colon [0,1] \rightarrow \Gr(d,2d), \quad \varepsilon \mapsto \cM(A_\varepsilon, B_\varepsilon)
	\end{align*}   
	is continuous.
	Moreover, $A_1$ is invertible and $B_1=0$, so  $\cM(A_1,B_1)$ defines the regular Dirichlet boundary conditions, and hence
 the manifold of regular boundary conditions is connected, and in fact even star shaped.	
 
 To show that the set of quasi-sectorial boundary conditions is not closed for  $d\geq 2$, consider for instance
 \begin{align*}
 A= \mathds{1}, \quad
  B=  N, 
\quad \hbox{and}\quad
 A_{\varepsilon}= \mathds{1}, \quad 
B_{\varepsilon}= 
N + \varepsilon \mathds{1} \hbox{ for } \varepsilon >0.
 \end{align*}  
 Then for nilpotent $N$, the matrices $A,B$ define regular but not quasi-sectorial boundary conditions while $A_\varepsilon, B_\varepsilon$ define quasi-sectorial boundary conditions for $\varepsilon>0$  because then  $B_\varepsilon$ is invertible. However $\cM(A_\varepsilon, B_\varepsilon)\to\cM(A, B)$ in $\Gr(d,2d)$ by Lemma~\ref{lemma:projmperp}, see also the proof of \cite[Theorem 3.10]{HKS2015}.
 
 Conversely, to show that the set of non-quasi-sectorial boundary conditions in $\Gr(d,2d)$ is not closed for  $d\geq 2$, let
   \begin{align*}
   A= \mathds{1}, \quad
   B= 0, 
   \quad \hbox{and}\quad
   A_{\varepsilon}= 
   \mathds{1}, \quad
   B_{\varepsilon}=  \varepsilon N
  \hbox{ for } \varepsilon >0.
   \end{align*}  
Then $A,B$ define quasi-sectorial boundary while for nilpotent $N$ the matrices  $A_\varepsilon, B_\varepsilon$ define regular but non-quasi-sectorial boundary conditions for $\varepsilon>0$ since then  $\varepsilon N$ is nilpotent, but $\cM(A_\varepsilon, B_\varepsilon)\to\cM(A, B)$ in $\Gr(d,2d)$ again by Lemma~\ref{lemma:projmperp}.
\end{proof}

\subsection{Comparison to the Birkhoff-Tamarkin theory}\label{subsec:birkhoff}
A now classical theory for boundary value problems on intervals has been initiated at the beginning of the 20th century by Birkhoff in the works \cite{Birkhoff1,Birkhoff2} and then continued by Tamarkin \cite{Tamarkin} and many others. It is elaborated in classical textbooks like \cite[Chapter XIX]{DSIII} or \cite{Naimark}. The Birkhoff-Tamarkin theory focuses on the convergence of the eigenfunction expansion for non-self-adjoint boundary conditions. This theory is still developing, see for instance \cite{Locker2000}, and also \cite{Freiling2012} for a survey on Birkhoff-irregular problems,  and the respective  references therein, and it also ramified including parameter dependent eigenvalue problems, see e.g. \cite{Mennicken2003}.

Central to the Birkhoff-Tamarkin theory is a notion of regularity of boundary conditions often referred to as Birkhoff-regularity which is related to the behaviour of determinants related to the eigenvalue equation. The Birkhoff-Tamarkin theory applies to general $n$-th order differential operators on intervals. Note that for the second derivative operator on an interval $[0,a]$ this is different from the notion of regularity promoted here, compare \cite[Section 3.3]{HKS2015}. The difference can be illustrated also using \cite[page 2344f.]{DSIII} where  all Birkhoff-irregular boundary conditions consisting of two linearly independent conditions are characterized
to be of the form
\begin{align*}
\psi'(0) + \gamma \psi'(a) + \alpha \psi(0) + \beta \psi(a)&=0 \quad \hbox{and}\quad
\psi(0) - \gamma \psi(a) =0, \hbox{ or } \\
\psi'(a) +\gamma  \psi'(0) + \alpha \psi(0) + \beta \psi(a)&=0 \quad \hbox{and}\quad
\psi(a) - \gamma \psi(0) =0, \quad \alpha, \beta, \gamma \in \IC,
\end{align*}
where the terms $+ \ldots$ in \cite[page 2344f.]{DSIII} are interpreted here as the terms with $\alpha$ and $\beta$.
In matrix form the first line translates to
\begin{align*}
A= \begin{bmatrix}
1 & -\gamma  \\
\alpha & \beta
\end{bmatrix} 
\quad \hbox{and}\quad
B= \begin{bmatrix}
0 & 0 \\
1 & -\gamma 
\end{bmatrix}.
\end{align*}
This is irregular in the sense of Definition~\ref{def:irregularbc} if $\Ker A\cap \Ker B \neq\{0\}$ since $\Rank (A\, B)=2$ for any $\alpha,\beta,\gamma \in \IC$. One has $\Ker A\cap \Ker B \neq\{0\}$ if and only if  $\alpha=\beta=0$ or $\beta= -\gamma\alpha$. If $\beta+\alpha\gamma \neq 0$ then $A$ is invertible, and equivalent boundary conditions are given by
\begin{align*}
A'=\mathds{1} \quad \hbox{and} \quad B'=A^{-1}B = \frac{1}{\beta+\alpha \gamma}\begin{bmatrix}
\gamma & - \gamma^2 \\ 1 & -\gamma 
\end{bmatrix},
\end{align*}
where $(B')^2=0$, and hence these boundary conditions are regular but not quasi-sectorial.

The other way round, in \cite[page 2344f.]{DSIII} it is also elaborated that all Birkhoff-regular boundary conditions fall under one of the following categories:
\begin{enumerate}[(a)]
	\item[(a)] $\psi'(0)+ \hbox{first order terms} =0$ and $\psi'(a)+ \hbox{first order terms} =0$, that is , in the formalism presented here this translates to $B=\mathds{1}$, where $A$ can be arbitrary. Hence these boundary conditions is quasi-sectorial.
	\item[(b)] $\psi(0)=0$ and $\psi(a)=0$, that is, one has the self-adjoint Dirichlet boundary conditions, which are in particular quasi-sectorial.
	 \item[(c1)]  $\psi'(0)+ \gamma \psi'(1) + \hbox{first order terms} =0$ 
	 and $\psi(a) =0$ for $\gamma \in \IC$, that is, in the first case
	 \begin{align*}
	 A=\begin{bmatrix}
	 a_{11} & 0 \\ 
	 0 & 1
	 \end{bmatrix}
	 \quad \hbox{and} \quad
	 B=\begin{bmatrix}
	 1 & -\gamma \\ 
	 0 & 0
	 \end{bmatrix},
	 \end{align*}
where for $a_{11}=0$ one has $m=d=2$ in the quasi-Weierstrass normal form, and hence quasi-sectorial boundary conditions, and for $a_{11}\neq 0$, one has equivalent boundary conditions $A'=\mathds{1}$ and $B'=\tfrac{1}{a_{11}}B$. Here, one observes that $B'$ is not nil potent and hence this is already the quasi-Weierstrass normal form with $m=d=2$ and $L=B'$, and in particular  these are quasi-sectorial.
	 \item[(c2)]  $\psi'(0)+ \gamma \psi'(1) + \hbox{first order terms} =0$  
	 and $\psi(0)+ \beta\psi(a) =0$ for $\beta,\gamma \in \IC$, with $\beta \neq -\gamma$,
	  that is, 
	 \begin{align*}
	 A=\begin{bmatrix}
	 0 & a_{12} \\ 
	 1 & \beta
	 \end{bmatrix}
	 \quad \hbox{and} \quad
	 B=\begin{bmatrix}
	 1 & -\gamma \\ 
	 0 & 0
	 \end{bmatrix}.
	 \end{align*}
	 If $a_{12}\neq 0$, then $A$ is invertible and one can consider equivalently
	 \begin{align*}
	 A'=\mathds{1} \quad \hbox{and} \quad B'=A^{-1}B = \frac{1}{-a_{12}}\begin{bmatrix}
	 \beta & -\beta\gamma \\ -1 & -\gamma 
	 \end{bmatrix},
	 \end{align*} 
	 where $B'$ is not nil potent, and using the same argument as above, it follows that  these are then quasi-sectorial.
	 If $a_{12}= 0$, then $\dim \Ker A=1$, and therefore $m=d=2$ in the quasi-Weierstrass normal form, and hence  one is in the quasi-sectorial case again. In case (c1) and case (c2), the corresponding conditions interchanging the endpoints of the interval, lead to the same result.	 
\end{enumerate}
Summarizing the above, one concludes that for the Laplacian on a compact interval, the notions of Birkhoff-regularity and quasi-sectoriality agree.

Furthermore, there is a characterization of Birkhoff-regularity by means of  estimates on the Green's function $r_{A,B}(\cdot,\cdot;k^2)$ to the boundary value problem $(-\Delta(A,B)-k^2)\psi = f$ for $f\in L^2([0,a];\IC)$ and $k^2\in \IC$, where one considers pointwise estimates of the type
\begin{align}\label{eq:Birkhoff-Stone}
|r_{A,B}(x,y;k^2)| \leq C |k|^{(\alpha-1)/2}, 
\end{align}
for a constant $C>0$ and $\alpha\geq 0$.
Then, the boundary value problem is Birkhoff-regular if \eqref{eq:Birkhoff-Stone} holds for $\alpha=0$ and
for $k^2$ along a halfline in the upper half plane, see for instance \cite[Lemma 2.2]{Freiling2012}. If such estimate hold for $\alpha>0$, then it is called Stone-regular
(or also almost regular \cite{Shkalikov1983} or mildly irregular \cite{Vass1936}) of order $\alpha$, see  \cite[Section 2]{Freiling2012} for a comprehensive survey. 
From the discussion of the Green's function or resolvent kernel $r_{A,B}$ in Subsection~\ref{subsec:resolvent_reg} below, one concludes that for the Laplacian on a compact interval, the notions of Birkhof- regularity and quasi-sectoriality agree, while the regular but non-quasi-sectorial boundary conditions are Birkhoff-irregular but Stone-regular for some $\alpha \in \IN$, where the minimal $\alpha$ is determined by the nilpotent part in the quasi-Weierstrass normal form. Irregular boundary conditions correspond to Stone-irregular boundary conditions.



%

The above considerations are restricted to a compact interval. However, there is also a system version of the Birkhoff-Tamarkin theory which is presented for instance in the book by Naimark \cite[Chapter III $\S$ 8.4]{Naimark}. To highlight the difference to the theory presented here,
consider the second derivative operator on two intervals $[0,1]$ with boundary conditions
\begin{align*}
A= \begin{bmatrix}
1 & 0 & 0 & 0 \\
0 & 1 & 0 & 0 \\
0 & 0 & 1 & -1 \\
0 & 0 & 0 & 0 \\
\end{bmatrix} 
\quad \hbox{and}\quad
B= \begin{bmatrix}
0 & 0 & 0 & 0 \\
0 & 0 & 0 & 0 \\
0 & 0 & 0 & 0 \\
0 & 0 & 1 & 1 \\
\end{bmatrix}.
\end{align*}
This corresponds to
\begin{align*}
\varphi_1(0)=0, \quad\varphi_2(0)=0, \quad \hbox{and} \quad \varphi_1(1)=\varphi_2(1), \quad \varphi_1'(1)=-\varphi_2'(1),
\end{align*}
and $-\Delta(A,B)$ is equivalent to the Dirichlet Laplacian on $[0,2]$, where one identifies the two copies of $[0,1]$ glued together at their endpoints with the interval $[0,2]$. 

Translating this into the formalism of Naimark, one has in its notation  $n=2$ (order of the operator) and $m=2$ (number of edges), and the boundary conditions are of the form
\begin{align*}
U_1(\varphi) = U_{10}(\varphi) + U_{11}(\varphi), \quad \hbox{and}\quad
U_2(\varphi) &= U_{20}(\varphi) + U_{21}(\varphi), 
\end{align*}
where for the orders of the boundary conditions  $k_0=0$ and $k_1=1$ one has
\begin{eqnarray*}
	U_{10}(\varphi) = A_1 \varphi(0), & & U_{11}(\varphi)=B_1 \varphi(1),  \hbox{ and}\\
	U_{20}(\varphi)= A_2 \varphi'(0)+ A_{20} \varphi(0), & &
	U_{21}(\varphi) = B_2\varphi'(1) + B_{20}\varphi(1)
\end{eqnarray*}
with
\begin{align*}
A_1 = \begin{bmatrix}
1 & 0 \\ 0 & 1
\end{bmatrix}, \quad B_1=\begin{bmatrix}
0 & 0 \\ 0 & 0 
\end{bmatrix}, \quad A_2=\begin{bmatrix}
0 & 0 \\ 0 & 0 
\end{bmatrix}, \quad A_{20}=
\begin{bmatrix}
0 & 0 \\ 0 & 0 
\end{bmatrix}, \quad
B_2=  \begin{bmatrix}
0 & 0 \\ 1 & 1
\end{bmatrix}, \quad B_{20}= \begin{bmatrix}
1 & -1 \\ 0 & 0 
\end{bmatrix}.
\end{align*}
Then regularity of the boundary conditions is defined with $\omega_1=i$ and $\omega_2=-i$ via
\begin{align*} 
\Phi(s)=\begin{bmatrix}
(A_1 + sB_1)\omega_1^{k_1} & (A_1 + (1/s)B_1)\omega_2^{k_1} \\
(A_2 + sB_12)\omega_1^{k_2} & (A_2 + (1/s)B_2)\omega_2^{k_2}
\end{bmatrix} = \begin{bmatrix}
1 & 0 & 1 & 0 \\
0 & 1 & 0 & 1 \\
0 & 0 & 0 & 0 \\
i s & i s & -i s & -i s
\end{bmatrix}, \quad s\in \IC, 
\end{align*}
where boundary conditions are regular if for
\begin{align*}
\det \Phi(s) = \varphi_{-2}s^{-2} + \varphi_{-1}s^{-1} + \varphi_{0} + \varphi_{1}s^{1} + \varphi_{2}s^{1}
\end{align*}
one has $\varphi_{-2}\neq 0$ and $\varphi_{2}\neq 0$. This is not satisfied here since $\det \Phi(s) =0$.

Herefrom, one sees that the notion of Birkhoff-regularity for systems can produce artifacts when inserting artificial  vertices. 
This indicates that this notion of regularity is not always compatible with a geometric interpretation of a system of ordinary differential equations such as graphs, and when looking for generators of $C_0$-semigroups it leaves out many relevant cases.  Moreover, while the Birkhoff-regularity is a powerful tool to check for Riesz basis properties of eigenfunctions, its verification is usually quite laborious. In contrast using for instance Corollary~\ref{cor:regboudnary conditions}, one can check for regularity in the sense of Definition~\ref{def:irregularbc} on the level of linear algebra. 

\begin{remark}
As pointed out above,  for the case of one compact interval the notions of Birkhoff-regularity and quasi-sectoriality introduced here agree. This equivalence has been quite a surprise to the author, and it invites to some speculative outlook. First, the generalisation of quasi-sectorial boundary conditions to differential operators of even order might be feasible, and replacing the quite complicated Birkhoff regularity condition by some simple criteria on the line of Corollary~\ref{cor:regboudnary conditions}
might be advantageous for applications. Second, one of the key features of Birkhoff-regular operators is the existence of a Riesz basis of eigenvectors. Interpreting operators with quasi-sectorial boundary conditions as a generalization of Birkhoff-regular operators on intervals, it is tempting to conjecture  the existence of Riesz basis (more likely with than without parenthesis) for such operators on compact finite metric graphs.
\end{remark}

\section{Semigroup generation}\label{sec:semigroup}
Using the notion of quasi-sectorial boundary conditions introduced in Definition~\ref{def:quasisec}, one obtains a complete picture for which Laplacians the first and second order Cauchy problems 
\begin{align*}
\partial_t\psi  - \Delta(A,B)\psi&=0, \quad t>0,\quad \psi(0)=\psi_0, \hbox{ and }\\
 \partial_t^2\psi  - \Delta(A,B)\psi&=0, \quad t>0,\quad \psi(0)=\psi_0, \quad \psi_t(0)=\varphi_0,
\end{align*}
are well-posed in terms of $C_0$-semigroups, analytic semigroups, and $C_0$-cosine operator functions, see e.g. \cite{AreBatHie01} for the definitions, and also \cite{Bobrowski2013} for a discussion of the differences between cosine families and semigroups. Note that the same set of boundary conditions can be imposed on finite metric graphs with different  geometries as illustrated in Figure~\ref{fig:graph_geometries}, where moreover, the length of the internal edges can be varied. However, the characterization given below depends only on the boundary conditions and not on the geometry.

\begin{theorem}[Characterisation of generators of $C_0$-semigroups]\label{thm:semigroup}
	Let $(\cG,\au)$ be a finite metric graph.
	Then $\Delta(A,B)$ is the generator of a $C_0$-semigroup in $L^2(\cG,\au)$ if and only if $A,B$ define quasi-sectorial boundary conditions in the sense of Definition~\ref{def:quasisec}.
If $\Delta(A,B)$ is the generator of a $C_0$-semigroup in $L^2(\cG,\au)$, then 
this semigroup extends to an analytic semigroup, and  $\Delta(A,B)$ generates also a $C_0$-cosine operator function. 
\end{theorem}
  
 The proof of this theorem needs some preparations, and it is completed eventually in Subsection~\ref{subsec:proof_thm_semigroup} below.   
  
\begin{remarks}
	\begin{enumerate}[(a)]
		\item 	Generators of $C_0$-cosine operator functions generate also $C_0$-semigroups, see \cite[Theorem 3.14.17 and its proofs]{AreBatHie01}. Therefore, Theorem~\ref{thm:semigroup} characterizes also the generators of $C_0$-cosine operator functions, and since analytic semigroups are in particular $C_0$-semigroups, it characterizes also generators of analytic semigroups, where in both cases $\Delta(A,B)$ defines a generator if and only if the boundary conditions defined by $A,B$ are quasi-sectorial.
		\item 	Note that Theorem~\ref{thm:semigroup} deals with general $C_0$-semigroups with
		\begin{align*}
		\norm{e^{\Delta(A,B)t}} \leq C e^{\mu t}, \quad t\geq 0, \quad \hbox{for some}\quad  C>0 \hbox{ and } \mu \in \IR.
		\end{align*}
		The question when one can chose $C=1$, \ie , when $\Delta(A,B)$ is quasi-dissipative (or equivalently $-\Delta(A,B)$ is quasi-accretive) has been discussed in \cite[Theorem 3.1]{Hussein2014}, and this is the case if and only if there are equivalent boundary conditions of the form \eqref{eq:bc_LP}. For this case, bounds on $\mu>0$ are discussed in \cite{Hus2013}. The semigroup is bounded if one can chose $\mu=0$, and the case $\mu=0$ and $C=1$ has been characterized in  \cite[Theorem 3.2]{Hussein2014}.
		The boundedness of the semigroup, that is the case for general $C>0$ and $\mu\leq 0$, has  been addressed at least for point interactions in \cite[Theorem 3.1 (d)]{HusMug2019}.       
	\end{enumerate}
\end{remarks}

 There has been a number of results on some classes of boundary conditions defining generators of analytic semigroups, see e.g. \cite{Kramar2007, Mug07, Mug_book, KSP2008} which however do not include a characterization. Some  boundary conditions of the form \eqref{eq:bc_LP} defining generators of $C_0$-cosine operator function are considered also in \cite{Mug_book}, and also in \cite{EngKra19} using different methods and considering more general elliptic second order operators on graphs. Apart from this, there is an extensive literature on the wave equation on networks, see e.g.
 \cite{Ali1994, Ali1984, Lagnese1994, Kuchment2002,  Zuazua2006, Kramar2007, Jacob2015, Kloss2012, KPS2012}
   and the references therein, and even more so 
 for the heat equations, see e.g. \cite{Below1988, Bobrowski2012, KSP2008, KPS2012_b, KPS2007, Banasiak2016} and the many  the references given therein. 

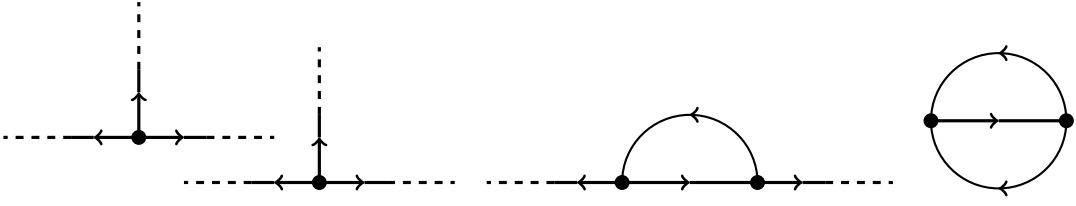
\begin{figure}
		\subfigure{
		\begin{tikzpicture}[scale=0.6]
			\fill[black] (0,0) circle (1ex);
	\draw[->, black, very thick] (0,0) -- (-1,0);
	\draw[black, very thick] (-1,0) -- (-1.5,0);
	\draw[black, very thick, dashed] (-1.5,0) -- (-3,0);
	\draw[->, black, very thick] (0,0) -- (1,0);
	\draw[black, very thick] (1,0) -- (1.5,0);
	\draw[black, very thick, dashed] (1.5,0) -- (3,0);
		\draw[->, black, very thick] (0,0) -- (0,1);
	\draw[black, very thick] (0,1) -- (0,1.5);
	\draw[black, very thick, dashed] (0,1.5) -- (0,3);
	\fill[black] (-4,1) circle (1ex);
		\draw[->, black, very thick] (-4,1) -- (-5,1);
	\draw[black, very thick] (-5,1) -- (-5.5,1);
	\draw[black, very thick, dashed] (-5.5,1) -- (-7,1);
	\draw[->, black, very thick] (-4,1) -- (-3,1);
	\draw[black, very thick] (-3,1) -- (-2.5,1);
	\draw[black, very thick, dashed] (-2.5,1) -- (-1,1);
		\draw[->, black, very thick] (-4,1) -- (-4,2);
	\draw[black, very thick] (-4,2) -- (-4,2.5);
	\draw[black, very thick, dashed] (-4,2.5) -- (-4,4);
			\end{tikzpicture}
	}
	\subfigure{
	\begin{tikzpicture}[scale=0.6]
				\fill[black] (0,0) circle (1ex);
						\fill[black] (-3,0) circle (1ex);
							\draw[->, black, very thick] (-3,0) -- (-1.5,0);
						\draw[black, very thick] (-1.5,0) -- (0,0);
							\draw[thick] (0,0) arc (0:180:1.5);
						\draw[<-, black, very thick] (-1.5,1.5) -- (-1.45,1.5);
							\draw[->, black, very thick] (0,0) -- (1,0);
						\draw[black, very thick] (1,0) -- (1.5,0);
						\draw[black, very thick, dashed] (1.5,0) -- (3,0);
							\draw[->, black, very thick] (-3,0) -- (-4,0);
						\draw[black, very thick] (-4,0) -- (-4.5,0);
						\draw[black, very thick, dashed] (-4.5,0) -- (-6,0);
			\end{tikzpicture}
}
		\subfigure{
	\begin{tikzpicture}[scale=0.6]
	\fill[black] (0,0) circle (1ex);
\fill[black] (-3,0) circle (1ex);
\draw[->, black, very thick] (-3,0) -- (-1.5,0);
\draw[black, very thick] (-1.5,0) -- (0,0);
\draw[thick] (0,0) arc (0:180:1.5);
\draw[<-, black, very thick] (-1.5,1.5) -- (-1.45,1.5);	
\draw[thick] (0,0) arc (0:-180:1.5);
\draw[<-, black, very thick] (-1.5,-1.5) -- (-1.45,-1.5);	
	\end{tikzpicture}
}
	\caption{Same boundary conditions for different geometries}\label{fig:graph_geometries}
\end{figure}

\subsection{Nilpotent matrices and resolvent estimates}\label{exDS}
Before discussing the proof of Theorem~\ref{thm:semigroup}, it is instructive to have  a closer look on the mechanism which prevents certain regular non-quasi-sectorial boundary conditions from defining a generator of a $C_0$-semigroup.
Consider the interval $[0,1]$ 
	and as in Example~\ref{ex:intermediate} the regular boundary conditions defined by $\psi(0)=0$ and $\psi(1)-\psi^{\prime}(0)=0$, i.e.,
	\begin{eqnarray*}
		A=\begin{bmatrix}1 & 0 \\ 0 & 1  \end{bmatrix}  & \mbox{and} & B=\begin{bmatrix} 0 & 0 \\ -1 & 0  \end{bmatrix}, \quad \hbox{where } \mathfrak{S}(k;A,B) = -\begin{bmatrix}
			1 & 0 \\ 2ik & 1
		\end{bmatrix}, \quad k\in \IC.
	\end{eqnarray*}
	%
		By Lemma~\ref{lemma:Nilpotent} and  Theorem~\ref{thm:riesz}  $\Delta(A,B)$ does not generate a $C_0$-semigroup on $L^2(\cG,\au)$. 
			This is an illustrative example for regular but non-quasi-sectorial boundary conditions.  It stands here as an exemplification of boundary conditions of the type $A=\mathds{1}$ and $B=N$ where $N$ is nilpotent.

		The same boundary conditions defined by $A,B$ given above can also be considered on a graph with only two external edges, \ie, $\cI=\emptyset$ and $|\cE|=2$ with $\cE=e_1 \cup e_2$. Then the resolvent kernel 
	   is given by
		\begin{align*}
		r_{\cM}(x,y;k) = \frac{i}{2k}\left\{\begin{bmatrix}
		e^{ik|x_1-y_1|} & 0 \\ 0 & e^{ik|x_2-y_2|}
		\end{bmatrix}
		+
		\begin{bmatrix}
		e^{ikx_1} & 0 \\ 0 & e^{ikx_2}
		\end{bmatrix}
		\mathfrak{S}(k;A,B)
		\begin{bmatrix}
		e^{iky_1} & 0 \\ 0 & e^{iky_2}
		\end{bmatrix}
		\right\}.
		\end{align*}
		Therefrom, the strategy of the proof of Theorem~\ref{thm:semigroup} becomes apparent when one takes into account Lemma~\ref{lemma:Nilpotent}, see also \cite[Lemma 5.3]{HusMug2019}. In particular in the case considered here, for $k=i\kappa$ with $\kappa \to \infty$, the resolvent does not exhibits the 
		decay $(-\Delta(A,B)+\kappa^2)^{-1}\lesssim \cO(1/\kappa^2)$  necessary for semigroup generators. 
	
For a  more heuristic interpretation  
	one can consider the resolvent problem $$(-\Delta(A,B)-k^2)\psi=f$$ more directly using the Dirichlet Laplacian $-\Delta(\mathds{1},0)$ on $[0,1]$. Then using the \textit{Ansatz}
	\begin{align*}
	\psi= \psi_1+ \psi_0, \quad \hbox{where} \quad \psi_1=(-\Delta(\mathds{1},0)-k^2)^{-1}f, \quad \Im k>0, 
	\end{align*}
	one finds that 
	\begin{align*}
	\psi_0(x) =  \sin(kx) \frac{\psi_1'(0)}{\sin(k)-1}, \quad \Im k>0,
	\end{align*}
	and since $-\Delta(\mathds{1},0)\psi_1= f + k^2 \psi_1$, where  $-\Delta(\mathds{1},0)= D_0^\ast D_0$ with $D_0 \psi = \psi'$ and $\Dom (D_0)= H_0^1([0,1];\IC)$,
	\begin{align*}
	\psi_1'(x) = (D_0^\ast)^{-1}(f + k^2 \psi_1)(x) = \int_{x}^1 f(y) + k^2 \psi_1(y) dy,
	\end{align*}
	where $(D_0^\ast)^{-1}\colon L^2([0,1])\rightarrow \Ran D_0$.
	Using the trace or evaluation operator $\gamma_0$ which maps $\psi \mapsto \psi(0)$, this implies that
		\begin{align*}
	|\psi_1'(0)| = |\gamma_0\circ (D_0^\ast)^{-1}(f + k^2 \psi_1)|&\leq \norm{\gamma_0\circ(D_0^\ast)^{-1}}_{\cL(L^2;\IC)}\left( \norm{f}_{L^2} +  \norm{k^2(-\Delta(\mathds{1},0)-k^2)^{-1}f}_{L^2} \right) \\
	&\leq C \norm{f}_{L^2}
	\end{align*}
	for some $C>0$ and any $k^2\in \IC\setminus \Sigma_{0,\theta}$ for $\theta \in (0,\pi/2)$, where $\Sigma_{0,\theta}:=\{z\in \IC \colon |\arg (z)| \leq \theta \}$. Here one uses that the trace operator $\gamma_0\colon H^1([0,1];\IC)\rightarrow \IC$ is bounded, $$\Ran (D_0^\ast)^{-1}= \{\psi\in H^1([0,1];\IC)\colon \int_{[0,1]} \psi =0\},$$ and hence $\gamma_0\circ (D_0^\ast)^{-1}\in \cL(L^2(\cG,\au);\IC)$,  and  that
the Dirichlet Laplacian generates an analytic semigroup. Summarizing, the trace $\psi_1'(0)$ is given as a $0$-th order operator applied to $f$ the norm of which can be estimated independent of $k$.
	
	More concretely, choosing for instance  an eigenfunction of the Dirichlet Laplacian as right hand side
	\begin{align*}
	f_{\kappa}(x)= \sin (\pi \kappa x), \quad \kappa\in \IN, \quad \hbox{then}\quad \psi_1(x;\kappa) = \frac{\sin (\pi \kappa  x)}{\pi^2 \kappa^2 -k^2}, \quad \hbox{and}\quad \psi_1'(0;\kappa) = \frac{\pi l}{\pi^2 \kappa^2 -k^2},
	\end{align*}
	if $k^2 \neq \kappa^2 \pi^2$. Then for $k=i\kappa$
	\begin{align*}
	\norm{\psi_1(\cdot;\kappa)} = \cO(1/\kappa^2) \quad \hbox{while}\quad  \norm{\psi_0(\cdot;\kappa)} = |\psi_1'(0;\kappa)|\cO(1/\sqrt{\kappa})=\cO(1/\kappa^{3/2}) \quad \hbox{as} \quad \kappa \to \infty.
	\end{align*}
	
The role of the Dirichlet problem and the behaviour of the trace of the derivative of its solution  become even more apparent when $A,B$ given above are considered on a graph with only two external edges, \ie, $\cI=\emptyset$ and $|\cE|=2$ with $\cE=e_1 \cup e_2$.  Then the resolvent problem can be rewritten to a coupled system 
\begin{align*}
(-\tfrac{d^2}{dx_1^2}-k^2) \psi_1 &= f_1, \quad \psi_1(0)=0, \\
(-\tfrac{d^2}{dx_1^2}-k^2) \psi_2 &= f_2, \quad \psi_2(0)=\psi_1'(0). 
\end{align*}
This can be solved iteratively, to obtain first, using the Dirichlet Laplacian $-\Delta(1,0)$ on $[0,\infty)$, that
\begin{align*}
\psi_1 = (-\Delta(1,0)-k^2)^{-1}f_1\quad \hbox{if} \quad \Im k>0.
\end{align*}
The second equation is then a boundary value problem with inhomogeneous Dirichlet boundary conditions a solution of which is given by
\begin{align*}
 \psi_2(x) = e^{ikx}\psi_1'(0)+  ((-\Delta(1,0)-k^2)^{-1}f_2)(x), \quad x\in e_2, \quad \hbox{if} \quad \Im k>0,
\end{align*} 
and $\psi_1'(0) = \gamma_0 \circ \partial_{x_1}(-\Delta(1,0)-k^2)^{-1}f_1$, hence $|\psi_1'(0)|\leq  \cO(\norm{f_1})$.
Now for $k=i\kappa$ with $\kappa>0$ 
\begin{align*}
\norm{e^{-\kappa x}\psi_1'(0)}_{L^2}  \lesssim \cO(\norm{f_2}_{L^2}/\kappa) \quad \hbox{as}\quad \kappa \to \infty.
\end{align*}
This resembles the resolvent estimate of a first order operator and not that of a second order operator. 

Interchanging $A$ and $B$, \ie, one can consider
	\begin{eqnarray*}
	A=\begin{bmatrix} 0 & 0 \\ -1 & 0  \end{bmatrix}  & \mbox{and} & B=\begin{bmatrix}1 & 0 \\ 0 & 1  \end{bmatrix},
\end{eqnarray*}
and then the resolvent behaviour is rather different,
Since $B$ is invertible, this defines quasi-sectorial boundary conditions. The resolvent problem can be rewritten to become
\begin{align*}
(-\tfrac{d^2}{dx_1^2}-k^2) \psi_1 &= f_1, \quad \psi_1'(0)=0, \\
(-\tfrac{d^2}{dx_1^2}-k^2) \psi_2 &= f_2, \quad \psi_2'(0)=\psi_1(0),
\end{align*}
which can be solved iteratively, where now using the Neumann Laplacian $-\Delta_1(0,1)$ on $[0,\infty)$ 
\begin{align*}
\psi_1 = (-\Delta_1(0,1)-k^2)^{-1}f_1\quad \hbox{for} \quad \Im k>0.
\end{align*}
The second equation is then a boundary value problem with inhomogeneous boundary conditions a solution of which is given by
\begin{align*}
\psi_2 = \frac{e^{ikx}}{ik}\psi_1(0)+  (-\Delta_2(0,1)-k^2)^{-1}f_2 \quad \hbox{for} \quad \Im k>0,
\end{align*} 
and $\psi_1(0)= \gamma_0(-\Delta_1(0,1)-k^2)^{-1}f_1$. For $k=i\kappa$ with $\kappa >0$ it follows that
\begin{align*}
|\psi_1(0)|= |\gamma_0 (-\Delta_1(0,1)+1)^{-1/2} (-\Delta_1(0,1)+1)^{1/2}(-\Delta_1(0,1)+\kappa^2)^{-1}f_1|\lesssim \cO(\norm{f_1}/\kappa), \quad \kappa \to \infty,
\end{align*}
where one uses that the operator $\gamma_0 (-\Delta_1(0,1)+1)^{-1/2}$  is bounded, and that
 by interpolation 
 $\norm{(-\Delta_1(0,1)+1)^{1/2}(-\Delta_1(0,1)+\kappa^2)^{-1}}\lesssim \cO(1/\kappa)$ as $\kappa\to\infty$.
Hence  $\norm{e^{-\kappa x}\psi_1'(0)}_{L^2}\lesssim \cO(\norm{f_2}_{L^2}/\kappa^2)$ as $\kappa \to \infty$ which is compatible with the properties of $C_0$-semigroup generators. 

\subsection{Quadratic forms for quasi-sectorial boundary conditions}\label{subsec:forms}
In many instances spectral properties of operator can be drawn back to quadratic forms. Here, integration by parts gives  
\begin{equation*}
\langle -\Delta(A,B)\psi, \psi \rangle
=\int_{\cG} \abs{ \psi^{\prime}}^2 
+ \langle \underline{\psi}',\underline{\psi}\rangle_{\cK}, \quad \psi \in \Dom(\Delta(A,B)).
\end{equation*}
If the boundary conditions are of the form \eqref{eq:bc_LP}, then
\begin{align*}
 \langle \underline{\psi}',\underline{\psi}\rangle_{\cK}= \langle -LP^\perp\underline{\psi},P^\perp\underline{\psi}\rangle_{\cK} \quad  \hbox{for }\psi \in \Dom(\Delta(A,B)),
\end{align*}
 and hence this defines a closable sectorial form the closure of which is associated with $-\Delta(A,B)$. However if this is not the case, then
not all the derivative terms at the vertices  cancel out, and 
therefore the numerical range is not confined to a sector, see for instance \cite{Hussein2014}. This can be illustrated in the following example.
\begin{example}[$\cP\cT$-symmetric point interaction -- continuation of Example~\ref{ex1_quasi_weiserstrass}]\label{ex1}
	Considering the setting of Example~\ref{ex1_quasi_weiserstrass},
Then the quadratic form defined by the operator $-\Delta(A_{\tau},B_{\tau})$ simplifies by integrating by parts and inserting the boundary conditions to become  
	\begin{equation*}
	\langle -\Delta(A_{\tau},B_{\tau})\psi,\psi\rangle = \int_{\cG} \abs{\psi^{\prime}}^2 + (1-e^{2i\tau}) \psi_2(0)\overline{\psi_2^{\prime}(0)} \quad \psi\in \Dom(-\Delta(A_{\tau},B_{\tau})).
	\end{equation*}
	In particular, the derivative term cannot be avoided, and the numerical range is entire $\IC$ for all $\tau\in(0,\pi/2]$. However, despite the numerical range for
	$\tau \in [0,\pi/2)$ the operator $-\Delta(A_{\tau},B_{\tau})$ 
	is similar to the self-adjoint Laplacian $-\Delta(A_{0},B_{0})$ and hence the generator of an analytic semigroup. Note that $A_{\tau},B_{\tau}$ define regular boundary conditions with $k$-independent
	Cayley transform.
	Hence, this is an example of quasi-sectorial boundary conditions which are not in the form ~\eqref{eq:bc_LP}.  
\end{example}

This exemplifies also the well-known fact that the numerical range of an operator is not stable under similarity transforms. It underlines that one cannot use always the usual $L^2$-scalar product to relate such operators to forms.

The similarity transform for star graphs introduced in \cite[Section 6]{HKS2015} gives for the case $\cI=\emptyset$ with quasi-sectorial boundary conditions a similarity transform to a Laplacian with boundary conditions of the form \eqref{eq:bc_LP}.  
The strategy here for the case $\cI\neq\emptyset$  is to take this similarity transform as an inspiration for a localization procedure  to construct an adjusted scalar product, and thereby to associate quasi-sectorial Laplacians with closed forms in this new setting.

To this end, consider first  cut-off functions
\begin{align*}
\chi^i_{+}, \chi^i_{-}, \chi_0^i &\colon [0,a_i] \rightarrow  [0,1] \quad \hbox{with} \quad (\chi_+^i)^2 + (\chi_-^i)^2 +(\chi_0^i)^2\equiv 1 \quad \hbox{for } i\in \cI, \hbox{ and } \\
\chi^i_{-}, \chi_0^i &\colon [0,\infty) \rightarrow [0,1] \quad \hbox{with} \quad (\chi_-^i)^2 +(\chi_0^i)^2\equiv 1 \quad \hbox{for } i\in \cE, 
\end{align*}
such that for $a_{\min}:=\min_{i\in \mathcal{I}} a_i$ if $\cI\neq \emptyset$ and $a_{\min}>0$ arbitrary if $\cI=\emptyset$, one has
\begin{align*}
\supp \chi^i_{+}\subset [a_i-a_{\min}/2, a_i], \quad  \supp \chi^i_{-}\subset [0, a_{\min}/2], \quad  \supp \chi^i_{0}\subset [a_{\min}/4,  a_i-a_{\min}/4]\quad  &\hbox{for } i\in \cI, \hbox{ and }  \\
\supp \chi^i_{-}\subset [0, a_{\min}/2], \quad  \supp \chi^i_{0}\subset [a_{\min}/4,  \infty)\quad  &\hbox{for } i\in \cE.
\end{align*}
Define the auxiliary space
\begin{align*}
\cH_{aux}:=L^2(\cG_{-})\oplus L^2(\cG_{+}) \oplus L^2(\cG_{\cE}) \oplus L^2(\cG), \quad 
L^2(\cG_{\cC}):=  \bigoplus_{j\in c(\cC)} L^2(I_j;\IC), \quad c(\cC)=\begin{cases}
\cI, & \cC\in \{+,-\}, \\
\cE, & \cC=\cE.
\end{cases}
\end{align*} 
Next, define some auxiliary operators. First, for $\chi^i_{+}, \chi^i_{-}, \chi_0^i$ as above let
\begin{align*}
\chi\colon L^2(\cG,\au) \rightarrow \cH_{aux}, \quad \psi \mapsto \chi \psi,
\quad
\chi  \psi:=
\begin{bmatrix}
\{\chi_{\cE}^i \psi_i\}_{i\in \cE} \\
\{\chi_-^i \psi_i\}_{i\in \cI} \\
\{\chi_+^i \psi_i\}_{i\in \cI} \\
\{\chi_0^i \psi_i\}_{i\in \cI\cup \cE}
\end{bmatrix},
\end{align*}
where the properties of the cut-off functions translates to $\chi^\ast \circ \chi = \mathds{1}_{L^2(\cG,\au)}$.

Second, one defines identification operators for $i,j\in \cE \cup \cI$ by 
\begin{align*}
\II_{ij}\colon  L^2(I_j;\IC)  \rightarrow L^2(I_i;\IC), \quad (\II_{ij}\psi_j)(x) = \begin{cases}
\psi_j(x),&  x\in I_i\cap I_j, \\
0, & x\in I_i\setminus I_j,
\end{cases}    
\end{align*}
that is, functions on the $j$-th edge are restricted or extended by zero and interpreted as functions on the $i$-th edge. Note that in general $\II_{li}\II_{ij}\neq \II_{lj}$. Putting these identification operators together one can define for $\cR,\cS\in \{\cE, -,+\}$ operators
\begin{align*}
\II_{\cR, \cS} = \{\II_{ij}\}_{i\in c(\cR), j\in c(\cS)}\colon  L^2(\cG_{\cR}) \rightarrow  L^2(\cG_{\cS}).
\end{align*} 
For $G \in \IC^{d\times d}$ invertible for $d$ given by ~\eqref{eq:diff}, one has a block structure $G=(G_{\cR, \cS})_{\cR,\cS \in \{\cE,-,+\}}$, where $G_{\cR,\cS}= \in \IC^{|\cR|\times |\cS|}$ with $\cC\in \{\cR,\cS\}$ and $|\cC|=|\cI|$ for $\cC\in \{-,+\}$ and $|\cC|=|\cE|$ for $\cC=\cE$. Then one sets 
\begin{align*}
M_G\colon\cH_{aux} \rightarrow \cH_{aux}, \quad M_G\psi = \begin{bmatrix}
 G_{\cE,\cE} \II_{\cE,\cE} & G_{\cE, -} \II_{\cE,-} & G_{\cE, +} \II_{\cE, +} &  0 \\
G_{-,\cE} \II_{-,\cE} &  G_{--} \II_{--} & G_{-+} \II_{-+}  &  0 \\
 G_{+,\cE} \II_{+,\cE}& G_{+-} \II_{+-} & G_{++}\II_{++} &  0 \\
0 & 0 & 0 & \mathds{1} 
\end{bmatrix}
\begin{bmatrix}
\{\psi_i\}_{i\in \cE} \\
\{\psi_i\}_{i\in \cI_-} \\
\{\psi_i\}_{i\in \cI_+} \\
\{\psi_i\}_{i\in \cI\cup \cE}
\end{bmatrix}.
\end{align*}
With the change of orientation operator on  each internal edge defined by 
\begin{align*}
(I_+\psi)_i(x) = \psi_i(a_i-x) \quad \hbox{for } i\in \cI,
\end{align*}
one defines
\begin{align*}
I, J\colon \cH_{aux}\rightarrow \cH_{aux}, \quad I= 
\begin{bmatrix}
\mathds{1} & 0 & 0 & 0 \\
0 & \mathds{1} & 0 & 0 \\
0 & 0 & I_+  & 0\\
0 & 0 & 0 & \mathds{1} 
\end{bmatrix}, \quad J= 
\begin{bmatrix}
\mathds{1} & 0 & 0 & 0 \\
0 & \mathds{1} & 0 & 0 \\
0 & 0 & -\mathds{1}  & 0\\
0 & 0 & 0 & \mathds{1} 
\end{bmatrix},
\end{align*} 
which satisfy $I^\ast=I$, $I^2=\mathds{1}$, $J^\ast=J$, $J^2=\mathds{1}$, and by the chain rule 
\begin{align*}
(I\psi)' = JI(\psi')\quad  \hbox{for}\quad \psi\in H^1(\cG_{-})\oplus H^1(\cG_{+}) \oplus H^1(\cG_{\cE}) \oplus H^1(\cG).
\end{align*}

With these preparations at hand, one  considers the sequilinear form defined by
\begin{align*}
\langle \psi, \varphi \rangle_{G, \chi} 
:=
\langle  I M_G  I \chi \psi, I  M_G  I \chi \varphi \rangle_{\cH_{aux}}, \quad \psi, \varphi \in L^2(\cG,\au).
\end{align*}
\begin{lemma}\label{lemma:skp}
	For $G$ invertible and $\chi$ as above 
	$\langle \cdot, \cdot \rangle_{G, \chi}$ defines a scalar product in $L^2(\cG,\au)$,  and the induced norm is   equivalent to the norm induced by the canonical scalar product $\langle \cdot, \cdot \rangle$.
\end{lemma}
\begin{proof}
	Linearity follows from the linearity of the scalar product in $\cH_{aux}$ along with the anti-linearity 
	 \begin{align*}
	 \overline{\langle \psi, \varphi \rangle}_{G, \chi} 
	 =
	\overline{ \langle   \psi,  \chi^\ast I  M_{G}^\ast M_{G}  I \chi \varphi \rangle}_{\cH_{aux}}= \langle     \chi^\ast I  M_{G}^\ast M_{G}  I \chi \varphi, \psi \rangle = \langle \varphi, \psi \rangle_{G, \chi}, \quad \psi, \varphi \in L^2(\cG,\au).
	 \end{align*}
	 To show the equivalence of norms observe first that
	 	 \begin{align*}
 \langle \psi, \psi \rangle_{G, \chi} &= 
	\langle \chi^\ast I  M_{G}^\ast M_{G}  I \chi \psi, \psi \rangle_{L^2(\cG)}  \leq \norm{\chi^\ast I  M_{G}^\ast M_{G}  I \chi}  \norm{\psi}^2_{L^2(\cG)} \quad \hbox{for }\psi  \in L^2(\cG,\au). 
	\end{align*} 
Then, note that on the range of $I\chi $ one has $\II_{li}\II_{ij}= \II_{lj}$, because $\supp (I\chi \psi)_i\subset [0,a_{\min/2}]$ for all $i\in \cI_{\pm} \cup \cE$ and hence $M_{H}M_{G}I\chi \psi = M_{HG}I\chi \psi$. In particular for $G$ invertible $$M_{G^{-1}}M_{G}I\chi \psi = I\chi \psi.$$ So, 
using that $I^\ast=I$, $I^2=\mathds{1}$, and $\chi^\ast \circ \chi = \mathds{1}_{L^2(\cG,\au)}$ one obtains for $\psi  \in L^2(\cG,\au)$
	 	 \begin{align*}
 \langle \psi, \psi \rangle_{G, \chi} = 
	 \langle M_G  I\chi \psi,  M_G I\chi \psi \rangle_{\cH_{aux}} 
	 \geq 
	 \norm{M_{G^{-1}}}^{-1}  \langle I\chi  \psi,    I\chi \psi \rangle_{\cH_{aux}} = \norm{M_{G^{-1}}}^{-1}  \langle  \psi,   \psi \rangle_{L^2(\cG)}. \quad\qedhere
	 \end{align*} 
\end{proof}

\begin{remark}\label{rem:metric}
	The relation between the standard scalar product $\langle \cdot, \cdot \rangle$ and $\langle \cdot, \cdot \rangle_{G, \chi}$ is given by 
	\begin{align*}
	\langle \psi, \varphi \rangle_{G, \chi} 
	=\langle \psi, \Theta \varphi \rangle \quad \hbox{where } \Theta= \chi^\ast I  M_{G}^\ast M_{G}  I \chi 
	\end{align*}
	and $\Theta$ is self-adjoint and by Lemma~\ref{lemma:skp} positive, and it plays the role of the metric operator.
\end{remark}
Now one can relate the operator $-\Delta(A,B)$ for quasi-sectorial boundary conditions $A,B$ to a sesquilinear form with respect to the scalar product $\langle \cdot, \cdot \rangle_{G, \chi}$. Consider 
\begin{align*}
\langle -\psi^{\prime\prime}, \varphi\rangle_{G, \chi} 
&=
\langle  -\psi^{\prime\prime}, \chi^\ast  I  M_{G}^\ast M_{G}  I \chi  \varphi \rangle_{L^2(\cG)} \\
&= \langle  \psi^{\prime}, (\chi^\ast  I  M_{G}^\ast M_{G}  I\chi \varphi)^\prime \rangle_{L^2(\cG)} +\langle  \underline{\psi}^{\prime}, \underline{\chi^\ast  I M_{G}^\ast M_{G}  I\chi \varphi} \rangle_{\cK}.
\end{align*}
Recall that $M_{H}M_{G}I\chi \psi = M_{HG}I\chi \psi$, in particular $M_{G}^\ast M_{G}I\chi \psi = M_{G^\ast G}  I\chi$. So,
\begin{align*}
 (\chi^\ast  I  M_{G^\ast G}  I\chi \varphi)^\prime &=  (\chi^\ast)^\prime  I  M_{G^\ast G}  I\chi \varphi +  \chi^\ast (I  M_{G^\ast G}  I\chi \varphi)^\prime  \\
 &=(\chi^\prime)^\ast  I  M_{G^\ast G}  I\chi \varphi +  \chi^\ast JI  M_{G^\ast G}  JI(\chi \varphi)^\prime \\
 &=(\chi^\prime)^\ast  I  M_{G^\ast G}  I\chi \varphi +  \chi^\ast JI  M_{G^\ast G}  IJ \chi^\prime \varphi +  \chi^\ast JI  M_{G^\ast G}  IJ \chi \varphi^\prime,
\end{align*} 
 and with $H:=G^\ast G$
 \begin{align*}
(\chi^\ast  I M_{G^\ast G}  I\chi \varphi) (x) =&  \{\chi_{\cE}^j(x_j)H^{ji}_{\cE,\cE}\II_{\cE,\cE}\chi_{\cE}^i\psi_i(x_i)\}_{i,j\in \cE} + \{\chi_{\cE}^j(x_j)H_{\cE,-}^{ji}\II_{\cE,-}\chi_{-}^i\psi_i(x_i)\}_{i\in \cI, j\in \cE} \\
+& 
\{\chi_{\cE}^j(x_j)H_{\cE,+}^{ji}\II_{\cE,+}\chi_{+}^i\psi_i(a_i-x_i)\}_{i\in \cI, j\in \cE} 
+\{\chi_{-}^j(x_j)H^{ji}_{-,\cE}\II_{-,\cE}\chi_{\cE}^i\psi_i(x_i)\}_{i\in \cE, j\in \cI}   \\
+& \{\chi_{-}^j(x_j)H_{--}^{ji}\II_{-,-}\chi_{-}^i\psi_i(x_i)\}_{i,j\in \cI} + 
\{\chi_{-}^j(x_j)H_{\cE,+}^{ji}\II_{\cE,+}\chi_{+}^i\psi_i(a_i-x_i)\}_{i,j\in \cI} \\
+&\{\chi_{+}^j(x_j)H^{ji}_{+,\cE}\II_{+,\cE}\chi_{\cE}^i\psi_i(a_i-x_i)\}_{i\in \cE, j\in \cI} 
+ \{\chi_{+}^j(x_j)H_{+-}^{ji}\II_{+,-}\chi_{-}^i\psi_i(a_i-x_i)\}_{i,j\in \cI} \\
+& 
\{\chi_{+}^j(x_j)H_{++}^{ji}\II_{+,+}\chi_{+}^i\psi_i(a_i-x_i)\}_{i,j\in \cI} 
+\{\chi_{0}^j(x_j)\chi_{0}^i\psi_i(x_i)\}_{i,j\in \cE\cup \cI}. 
 \end{align*}
 Hence
 \begin{align*}
\underline{\chi^\ast  I  M_{G^\ast G}  I\chi \varphi} = G^\ast G\underline{\varphi} \quad \hbox{and} \quad 
\langle  \underline{\psi}^{\prime}, \underline{\chi^\ast  I M_{G^\ast G}  I\chi \varphi} \rangle_{\cK} = 
\langle  G\underline{\psi}^{\prime},G\underline{\varphi}\rangle_{\cK}.
 \end{align*}
The goal of the whole construction is to eliminate the trace term $\underline{\psi}^{\prime}$ which appears after integration by parts. Now, recalling that by the quasi-Weierstrass normal form
\begin{align*}
\Dom(A,B)=\{\psi\in H^2(\cG)\colon (L+P)G \underline{\psi} + P^\perp G\underline{\psi}^{\prime}=0  \},
\end{align*}
where $L,P$ satisfy \eqref{eq:bc_LP}, one obtains for $\dom (\delta_{A,B}):= \{\psi \in H^1(\cG)\colon P^\perp G\underline{\psi}=0\}$ that
\begin{align*}
\langle  G\underline{\psi}^{\prime},G\underline{\varphi}\rangle_{\cK}=- \langle  LG\underline{\psi},G\underline{\varphi}\rangle_{\cK}, \quad \psi \in \Dom(\Delta(A,B)), \quad \varphi \in \dom (\delta_{A,B}),
\end{align*}
and one defines the sesquilinear form for $\psi, \varphi \in \dom (\delta_{A,B})$ by
 \begin{multline}\label{eq:deltaAB}
 \delta_{A,B}[\psi,\varphi]:=\langle M_G I J \chi \psi',   M_{G}  IJ \chi \varphi^\prime \rangle + 
 \langle M_G I \chi' \psi',   M_{G}  I\chi \varphi \rangle
+  \langle M_G J I \chi \psi',     M_{ G}  IJ \chi^\prime \varphi \rangle \\ - \langle  LG\underline{\psi},G\underline{\varphi}\rangle_{\cK}.
 \end{multline}
 Recall that a quadratic form $\delta$ in a Hilbert space $\cH$ with domain $\cV$ is of Lions type if there is  constant $C>0$ such that $|\Im \delta[\psi]|\leq C \norm{\psi}_{\cH}\norm{\psi}_{\cV}$ for all $\psi\in \cV$, cf. \cite[Section 6.2]{Mug_book} .
 \begin{lemma}\label{lemma:quasi-sectorial}
 	Let $A,B$ be quasi-sectorial boundary conditions, then the form $\delta_{A,B}$ is a closed densely defined sectorial form  of Lions type in $L^2(\cG,\au)$ with $\langle \cdot, \cdot\rangle_{G, \chi} $ as scalar product, and $-\Delta(A,B)$ is associated with $\delta_{A,B}$. 
 \end{lemma}
\begin{proof}
	Since $\oplus_{i\in \cI\cup \cI}C_0^\infty(I_j;\IC)\subset \dom (\delta_{A,B})$, and  $\oplus_{i\in \cI\cup \cI}C_0^\infty(I_j;\IC)\subset L^2(\cG,\au)$ is dense, also $\delta_{A,B}$ is densely defined.
	
Next, one considers the leading term $\langle M_G I J \chi \psi',   M_{G}  IJ \chi \psi^\prime \rangle$,
 and one observes that
	\begin{align*}
	\Im \langle M_G I J \chi \psi',   M_{G}  IJ \chi \psi^\prime \rangle &=0, \\
	\langle M_G I J \chi \psi',   M_{G}  IJ \chi \psi^\prime \rangle &\leq \norm{\chi^\ast JIM_{G^\ast}M_G I J \chi} \norm{\psi'}^2, \\
\langle M_G I J \chi \psi',   M_{G}  IJ \chi \psi^\prime \rangle &\geq \norm{M_{G^{-1}}}^{-1} 
\langle I J \chi \psi',   IJ \chi \psi^\prime \rangle
= \norm{M_{G^{-1}}}^{-1} \norm{\psi'}^2,
	\end{align*}
where the two estimates are analogous to the ones in the proof of Lemma~\ref{lemma:skp}. Hence this term defines a densely defined closed symmetric form on $\dom\delta_{A,B}$. The other terms will be interpreted as relatively bounded perturbations.

For the second and third term one has
	\begin{align*} 
|\Im \langle M_G I \chi' \psi',   M_{G}  I\chi \varphi \rangle| &\leq |\langle M_G I \chi' \psi',   M_{G}  I\chi \varphi \rangle|
\leq \norm{  (\chi^\prime)^\ast  I J M_{ G^\ast G}  IJ \chi}\norm{\psi'} \norm{\varphi}, \quad \hbox{and} \\
|\Im \langle M_G J I \chi \psi',     M_{ G}  IJ \chi^\prime \varphi \rangle|&\leq |\langle M_G J I \chi \psi',     M_{ G}  IJ \chi^\prime \varphi \rangle|\leq  \norm{  \chi^\ast I J M_{ G^\ast G}  IJ \chi^\prime }\norm{\psi'}\norm{\varphi}.
\end{align*}

The trace term can be estimated by Agmon's inequality in $\IR$, \ie, for $\psi_i= u + iv$ for $i\in \cI$, then 
\begin{align*}
|\psi_i(0)|^2 = u^2(0) + v^2(0) &= 2 \int_{a_i}^0 (\chi_-u)(x) (u\chi_-)'(x)   + (\chi_-v)(x) (v\chi_-)'(x) dx \\
&\leq 2\norm{\chi u} \norm{(u\chi_-)'} + 2\norm{\chi v} \norm{(v\chi_-)'} \\
&= 2\norm{\chi u} \norm{\chi_-'u + \chi u} + 2\norm{\chi v} \norm{\chi_-'v + \chi v} \\
&\leq C (\norm{u} + \norm{u'}) +  C (\norm{v} + \norm{v'}) \leq
C \norm{\psi_i}\norm{\psi_i'}. 
\end{align*}
Analogous estimates hold for $\psi_i(a_i)$ using $\chi_+$ and for $\psi_i(0)$ when $i\in \cE$ using $\chi_{\cE}$. 
Hence,
\begin{align*}
|\Im \langle  LG\underline{\psi},G\underline{\psi}\rangle_{\cK}| \leq |\langle  LG\underline{\psi},G\underline{\psi}\rangle_{\cK}| \leq \norm{G^\ast LG} \norm{\underline{\psi}}^2 \leq C\norm{G^\ast LG} \norm{\psi}\norm{\psi'}.
\end{align*}

Hence, the last three terms in \eqref{eq:deltaAB} define relatively bounded perturbations of the form defined by the first term. Therefrom, using the boundedness  of the trace operator $\psi \mapsto \underline{\psi}$, and the completeness of $H^1(\cG)$, one concludes that $\delta_{A,B}$ defines a densely defined closed form on $\dom(\delta_{A,B})$,  cf. \cite[Theorem VI.3.4]{Kato}. 
By the estimates on the imaginary part it follows that it is of Lions type.  
Moreover, the operator $-\Delta(A,B)$ is associated with $\delta_{A,B}$ by the calculation above which give that
\begin{align*}
\langle -\Delta(A,B)\psi ,  \varphi\rangle_{G, \chi} =  \delta_{A,B}[\psi,\varphi] \quad \hbox{for } \psi \in \Dom(\Delta(A,B)) \hbox{ and }\varphi \in \dom (\delta_{A,B}).
\end{align*}
In particular the assumptions of the first representation theorem \cite[Theorem VI.2.1]{Kato} are satisfied, where one verifies that $\Dom (\Delta(A,B))$ is a core of $\delta(A,B)$. 
\end{proof}

 \subsection{Spectral enclosure}\label{subsec:spec}
 The location of the spectrum is a necessary condition for an operator to be a generator of a $C_0$-semigroup. 
 \begin{proposition}[Location of the spectrum]\label{prop:spec}
 \begin{enumerate}[(a)]
 		\item If $A,B$ define regular  boundary conditions, then for any $C>0$ there exists a $c>0$ such that
 	\begin{align*}
 	\sigma(-\Delta(A,B)) \subset \{z\in \IC\colon |\Im z|+c\leq C(\Re z)\},
 	\end{align*}
 	\ie, it is contained in a sector;
 		\item If $A,B$ define quasi-sectorial boundary conditions, then 
 		\begin{align*}
 		\sigma(-\Delta(A,B)) \subset \{z\in \IC\colon c(\Im z)^2 -C\leq \Re z\} \quad \hbox{for some } C,c>0,
 		\end{align*}
 	\ie, the it is contained in a parabola around a positive half-axis.
 	\end{enumerate}
 \end{proposition}
 \begin{remark}
	In the context of non-self-adjoint operators, the location of the spectrum is not as significant as in the self-adjoint case, and instead the emphasis is laid on  pseudospectra which are related level sets of the norm of the resolvent. The notion of $\varepsilon$-spectra or pseudospectra goes back to Landau, see \cite{Landau1975}, and since it has been investigated intensively, see e.g \cite{Siegl_2015} and the references therein.   
\end{remark}
 \begin{proof}[Proof of Proposition~\ref{prop:spec}]
 	In the case $\cI= \emptyset$ one has $$\sigma(-\Delta(A,B))=\sigma_{ess}(-\Delta(A,B)) \cup \sigma_p(-\Delta(A,B)) \quad \hbox{where } \sigma_{ess}(-\Delta(A,B))=[0,\infty).$$  This follows from the fact that the resolvent of $-\Delta(A,B)$, given in Subsection~\ref{subsec:resolvent_reg} below, is a finite rank perturbation compared to a self-adjoint Laplacian,  see \cite[Proposition 4.11]{HKS2015} for more details. Moreover, $\sigma_p(-\Delta(A,B))$ is a finite set since the eigenvalue equation reduces to $\det (A+ik B)$ for the root $\Im k>0$, while the residual spectrum is empty, see \cite[Proposition 4.6]{HKS2015}. This already proves the claim for $\cI=\emptyset$.
 	
 	Consider now the case $\cI\neq \emptyset$.
 	By \cite[Proposition 4.7]{HKS2015},  $k^2\in \IC\setminus [0,\infty)$ is in the resolvent set, whenever for $k$ with $\Im k>0$ the matrix $A\pm ik B$ is invertible,  
 	and 
 	\begin{align}\label{eq:seceq}
 	\det (\mathds{1} - \mathfrak{S}(k;A,B)T(k,\au))\neq 0, \quad \hbox{where} \quad T(k,\au)=\begin{bmatrix}
 	0 & 0 & 0 \\
 	0 & 0 & e^{ik\au} \\
 	0 & e^{ik\au} & 0
 	\end{bmatrix}.
 	\end{align}
 	Note that 
 	\begin{align*}
 	\norm{T(k,\au)} \leq e^{- (\Im k) a_{\min}} \quad \hbox{for } k\in  \IC_{\Im >0}, \quad \hbox{where } a_{\min}:=\min_{i\in \cI} a_i. 
 	\end{align*}
 	Here, and in the following the notations
 	\begin{align*}
 	\IC_{\Im >0}&:= \{z\in \IC\colon \Im z>0\}, \quad \IC_{\Im < 0}:= \{z\in \IC\colon \Im z< 0\},  \\
 	\IC_{\Re >0}&:= \{z\in \IC\colon \Re z>0\}, \quad  \IC_{\Re < 0}:= \{z\in \IC\colon \Re z< 0\},
 	\end{align*}
 	are used.
 	By Lemma~\ref{lemma:Nilpotent} in case (a) and (b)
 	\begin{align*}
  \norm{\mathfrak{S}(k,A,B)T(k,\au)} \leq  |p(k)|e^{- (\Im k) a_{\min}}    \quad \hbox{and} \quad 
  	\norm{\mathfrak{S}(k,A,B)T(k,\au)} \leq Ce^{- (\Im k) a_{\min}},
 	\end{align*}
 	respectively for $k\in  \IC_{\Im >0}\setminus \{B_{\varepsilon}(p_1), \ldots, B_{\varepsilon}(p_l)\}$, where  $C>0$ is a constant and $p$ is polynomial in $k$ .
 	
 	In case (b) for $k\in \IC_{\Im>0}$ with $\Im k > \ln (C)/a_{\min}$, there are no zeros of the secular equation \eqref{eq:seceq}. Hence all roots of eigenvalues are located in a strip parallel to the real axis.
 	For $k=r e^{i\varphi }$ one has 
 	\begin{align*}
 	\Im k^2= r^2\sin(2\varphi) = 2r\sin(\varphi) r\cos(\varphi) \quad \hbox{and} \quad \Re k^2= r^2\cos(2\varphi) = r^2\cos^2(\varphi) - r^2\sin^2(\varphi), 
 	\end{align*}
 	and
 	hence
 	\begin{align*}
 	\{k\in \IC_{\Im >0}\colon \Im k > C \}^2 = \{z\in \IC\colon (\Im z)^2/C^2 -C^2 > \Re z.  \}
 	\end{align*} 
 	which is a parabola with vertex at $-C^2$, see Figure~\ref{fig:spectrum} (a) and (c).

 	In case (a), one has $|p(k)|\leq \alpha + \beta|k|^d$ for some $\alpha,\beta\geq 0$ if $p$ is a polynomial of degree $d\in \IN_0$, and assuming $C|\Re k| <  \Im k$ for given $C\in (0,1)$, one has 
 	\begin{align*}
 	\norm{\mathfrak{S}(k;A,B)T(k,\au)} \leq (\alpha + \beta \sqrt{C^2+1}\Im k )e^{- (\Im k) a_{\min}}.
 	\end{align*}
 	Then there exists a $c>0$ such that 
 		\begin{align*}
 	\norm{\mathfrak{S}(k;A,B)T(k,\au)} \leq (\alpha + \beta \sqrt{C^2+1}\Im k )e^{- (\Im k) a_{\min}} \leq \frac{1}{2} \quad  \hbox{for } \Im k \geq c,
 	\end{align*}
 	and hence by \eqref{eq:seceq} for any such $k$ one has that $k^2\in \rho(-\Delta(A,B))$.  So, 
 	\begin{multline*}
 	\{k\in \IC_{\Im >0}\colon C|\Re k| \leq  \Im k, \Im k >c \}^2 \subset \\ \{z\in \IC\colon (\Im z)^2/C^2 -C^2 > \Re z \} \cap \{z\in \IC\colon \Im z > |1/(1-C^2)\Re z| \}. 
 	\end{multline*}
 	That is the spectrum is contained in a sector as depicted in see Figure~\ref{fig:spectrum} (b) and (d).
 	 \end{proof}

 	\begin{figure}[h]
 		\begin{center}
 			\subfigure[$\Im k \leq C$]{
 				\begin{tikzpicture}[scale=2]
 				\filldraw [draw=white, fill=gray!20!white] (-1.2,0) -- (1.2,0) -- (1.2,0.3) -- (-1.2,0.3) -- (-1.2,0) -- cycle;
 				\draw[->] (-1.2,0) -- (1.2,0) node[right] {\tiny{$\Re k$}};
 				\draw[->] (0,-0.2) -- (0,0.5) node[above] {\tiny{$\Im k$}};
 				\draw (1,0.5) node[right] {\tiny{$\IC_{\Im >0}$}};
 				\end{tikzpicture}
 			}	
 			\subfigure[$\Im k \leq \max\{C |\Re k|, c\}$]{
 				\begin{tikzpicture}[scale=2]
 				\filldraw [draw=white, fill=gray!20!white] (-1.2,0) -- (1.2,0) -- (1.2,0.4) -- (0.3,0.1) -- (-0.3,0.1) -- (-1.2,0.4) -- (-1.2,0) -- cycle;
 				\draw[dashed, gray]  (1.2,0.4) -- (0,0) ;
 				\draw[dashed, gray]  (-1.2,0.4) -- (0,0);
 				\draw[->] (-1.2,0) -- (1.2,0) node[right] {\tiny{$\Re k$}};
 				\draw[->] (0,-0.2) -- (0,0.5) node[above] {\tiny{$\Im k$}};
 				\draw (1,0.5) node[right] {\tiny{$\IC_{\Im >0}$}};
 				\end{tikzpicture}
 			}	
 			\subfigure[$\Im k \leq C$]{
 				\begin{tikzpicture}[scale=2]
 				\fill [gray!20!white, domain=0:1.5, variable=\x]
 				(0, 0)
 				-- plot ({\x-0.3}, {0.38*sqrt(\x)})
 				-- (1.2, 0)
 				-- cycle;
 				\fill [gray!20!white, domain=0:1.5, variable=\x]
 				(0, 0)
 				-- plot ({\x-0.3}, {-0.38*sqrt(\x)})
 				-- (1.2, 0)
 				-- cycle;
 				\draw[->] (-1.2,0) -- (1.2,0) node[right] {\tiny{$\Re k^2$}};
 				\draw[->] (0,-0.2) -- (0,0.5) node[above] {\tiny{$\Im k^2$}};
 				\draw (1,0.5) node[right] {\tiny{$\IC$}};
 				\end{tikzpicture}
 			}	
 			\subfigure[$\Im k \leq \max\{C |\Re k|, c\}$]{
 				\begin{tikzpicture}[scale=2]
 				\fill [gray!20!white, domain=0:1.2, variable=\x]
 				(0, 0)
 				-- plot ({\x}, {0.4*\x})
 				-- (1.2, 0)
 				-- cycle;
 				\fill [gray!20!white, domain=0:1.2, variable=\x]
 				(0, 0)
 				-- plot ({\x}, {-0.4*\x})
 				-- (1.2, 0)
 				-- cycle;
 				\fill [gray!20!white, domain=0:1.3, variable=\x]
 				(0, 0)
 				-- plot ({\x-0.1}, {0.3*sqrt(\x)})
 				-- (1.2, 0)
 				-- cycle;
 				\fill [gray!20!white, domain=0:1.3, variable=\x]
 				(0, 0)
 				-- plot ({\x-0.1}, {-0.3*sqrt(\x)})
 				-- (1.2, 0)
 				-- cycle;
 				\draw[dashed, gray] (0,0) -- (1.2,0.48);
 				\draw[dashed, gray] (0,0) -- (1.2,-0.48);
 				\draw[gray] (-0.3,0) -- (1.2,0.6);
 				\draw[gray] (-0.3,0) -- (1.2,-0.6);
 				\draw[->] (-1.2,0) -- (1.2,0) node[right] {\tiny{$\Re k^2$}};
 				\draw[->] (0,-0.2) -- (0,0.5) node[above] {\tiny{$\Im k^2$}};
 				\draw (1,0.5) node[right] {\tiny{$\IC$}};
 				\end{tikzpicture}
 			}	
 			\caption{Gray indicates where roots of eigenvalues and eigenvalues can be located, respectively}\label{fig:spectrum}
 		\end{center}
 	\end{figure}
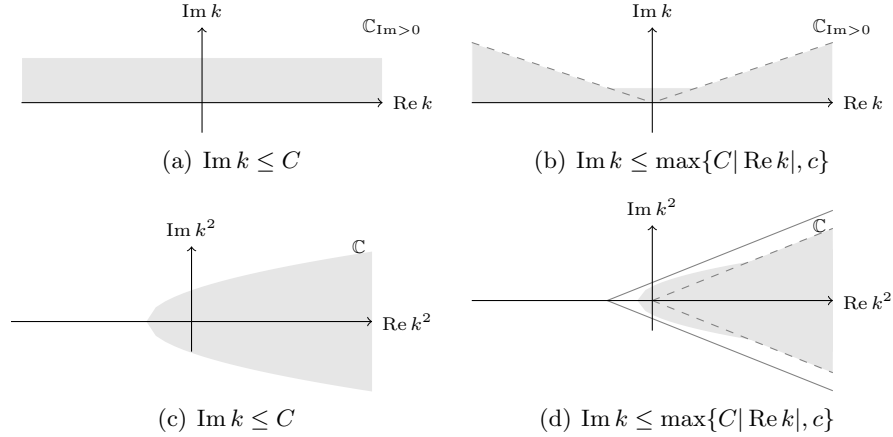

 \begin{remark}[Erratum to Lemma 4.3 in \cite{HKS2015}] To show that the resolvent set for regular boundary conditions is non-empty in \cite[Lemma 4.3]{HKS2015}, the author has  assumed implicitly that $\mathfrak{S}(i\kappa,A,B)$ is at least polynomially bounded as $\kappa\to \infty$. To justify this, one needs in fact Lemma~\ref{lemma:Nilpotent} as shown in the proof of Proposition~\ref{prop:spec}. 
 \end{remark}

 \subsection{Resolvent estimates for regular boundary condition} \label{subsec:resolvent_reg}

 The resolvent for regular boundary conditions is an integral operator with kernel  
 \begin{align*}
 &r_{\cM}(x,y;k)= r^0(x,y;k) + r^1_{\cM}(x,y;k)
 \end{align*}
 with 
 $
 \{r^0(x,y;k)\}_{j,j^{\prime}}
 = \delta_{j,j^{\prime}}\frac{i}{2k} e^{ik\abs{x_j-y_j}}
 $
 and
 \begin{align*}
 r_{\cM}^1(x,y;k) 
 =\frac{i}{2k} \Phi(x,k,\au) \left[\mathds{1}- 
 \mathfrak{S}(k,A,B) T(k;\au)\right]^{-1}\mathfrak{S}(k,A,B) 
\Phi(y,k,\au)^T,
 \end{align*}
 where
 the  $\Phi(x,k,\au)$ is given by 
 \begin{equation*}
\Phi(x,k,\au):= \begin{bmatrix} \phi_{\cE}(x,k) & 0 & 0 \\ 0 & \phi_+(x,k)& \phi_-(x,k) \end{bmatrix},
 \end{equation*}
 respectively, with diagonal matrices
 \begin{align*}
 \phi_{\cE}(x,k)=\diag\{ e^{ikx_j}\}_{j \in \cE}, \quad \phi_{+}(x,k)=\diag\{ e^{ikx_j}\}_{j\in \cI}, \quad \hbox{and} \quad \phi_{-}(x,k)=\diag\{ e^{ik(a_j-x_j)}\}_{j\in \cI},
 \end{align*}  and $\Phi(x,k)^T$ denotes the transposed of $\Phi(x,k)$, see \cite[Lemma~3.10]{KSP2008} and \cite[Proposition 4.7]{HKS2015}. 
 
 \begin{lemma}[Upper bound on the resolvent for regular boundary conditions]\label{lemma:res_upperbound}
 	Let $A,B$ define regular boundary conditions, then for $c,C$ are as in Proposition \ref{prop:spec}~(a) there exists a constant $C'>0$ such that
 	\begin{align*}
 	\norm{(-\Delta(A,B)-k^2)^{-1}} \leq C' |k|^{d-2} \quad \hbox{for all } k^2\in  \{z\in \IC\colon c(\Im z)^2 -C\leq \Re z\}.
 	\end{align*}
 \end{lemma} 
\begin{proof}
Since	
	\begin{align*}
	\norm{(-\Delta(A,B)-\lambda)^{-1}}= \norm{(-\Delta(A,B)-\lambda)^{-1} - (-\Delta(\mathds{1},0)-\lambda)^{-1}} + \norm{(-\Delta(\mathds{1},0)-\lambda)^{-1}}
	\end{align*}
	it is sufficient to estimate the operator norm  of the integral operator defined by $ r_{\cM}^1(\cdot,\cdot;k)$.
		From the proof of Proposition \ref{prop:spec}~(a)  one deduces that
	\begin{align*}
	\norm{(\mathds{1} -  \mathfrak{S}(k,A,B)T(k,\au))^{-1}} \leq 2 \quad \hbox{for all }  k\in \IC\setminus \{z\in \IC\colon c(\Im z)^2 -C\leq \Re z\},
	\end{align*}
	and from the proof of Lemma~\ref{lemma:Nilpotent} one obtains that for $k$ away from the poles of $\mathfrak{S}(k,A,B)$ 
	\begin{align*}
	\norm{\mathfrak{S}(k,A,B)} \leq C'' |k|^{d-1} 
	\end{align*}
	for some $C''>0$ Hence, for some $C'>0$
	\begin{align*}
	 \norm{\int_{\cG}r_{\cM}^1(x,y;k) f(y) dy} \leq \frac{1}{|k|} \frac{C' |k|^{d-1}}{\Im k}\norm{f} \leq C'|k|^{d-2} \norm{f}, \quad k \in 	\{k\in \IC_{\Im >0}\colon \Im k > C \}. 
	\end{align*}
\end{proof}

 \begin{lemma}[Lower bound on the resolvent for non-quasi-sectorial boundary conditions]\label{lemma:res}
 	Let $A,B$ define regular but not quasi-sectorial boundary conditions, then there exists a constant $C>0$
 	\begin{align*}
 	\norm{(-\Delta(A,B)+\kappa^2)^{-1}} \geq \frac{C}{\kappa} \quad \hbox{as} \quad \kappa \to \infty.
 	\end{align*}
 \end{lemma} 
\begin{proof}
	Note first that since the Laplacian on $L^2(\IR)$ generates a contraction semigroup
	\begin{align*}
	\norm{\int_{\cG} r^0(x,y;i\kappa)\psi(y) dy} \leq \frac{1}{\kappa^2} \norm{\psi} \quad \hbox{for}\quad \kappa>0.
	\end{align*}
Then consider $r^1_{\cM}(\cdot,\cdot;i\kappa)$ together with Lemma~\ref{lemma:Nilpotent}. For
\begin{align*}
\norm{\int_{\cG} r^1_{\cM}(\cdot,y;i\kappa) \psi(y) dy} = \sup_{0\neq \varphi\in L^2(\cG,\au)}  \frac{\abs{\langle \int_{\cG} r^1_{\cM}(\cdot,y;i\kappa) \psi(y), \varphi\rangle  dy}}{\norm{\varphi}},
\end{align*}
and since $\Phi(x,i\kappa,\au)$ has only real entries for $\kappa>0$
\begin{multline*}
\langle \int_{\cG} r^1_{\cM}(\cdot,y;i\kappa) \psi(y), \varphi\rangle  dy = \\ \frac{1}{2\kappa}\langle 
 \left[\mathds{1}- 
\mathfrak{S}(i\kappa,A,B) T(i\kappa;\au)\right]^{-1}\mathfrak{S}(i\kappa,A,B) 
\int_{\cG} \Phi(y,i\kappa,\au)^T \psi(y) dy, \int_{\cG}\Phi(x,i\kappa,\au)^T\psi(x) dx \rangle_{\cK}.
\end{multline*}
Moreover since by the proof of Lemma~\ref{lemma:Nilpotent} one has $\norm{\mathfrak{S}(i\kappa,A,B)} \leq C|\kappa|^{d-1}$ as $\kappa \to \infty$, and $\norm{T(i\kappa;\au)}=e^{-\kappa a_{\min}}$,
 for any $q\in (0,1)$ there exists $c_q>0$ such that $\norm{\mathfrak{S}(i\kappa,A,B) T(i\kappa;\au)}<q$ for $\kappa\geq c_q$, and hence using Neumann series
\begin{align*}
 \left[\mathds{1}- 
\mathfrak{S}(i\kappa,A,B) T(i\kappa;\au)\right]^{-1} = \mathds{1} + \sum_{n=1}^{\infty} \left[\mathfrak{S}(i\kappa,A,B) T(i\kappa;\au)\right]^{n}.
\end{align*}
Inserting this into the above gives
\begin{align*}
&\quad\langle \int_{\cG} r^1_{\cM}(\cdot,y;i\kappa) \psi(y), \varphi\rangle  dy\\
&=  \frac{1}{2\kappa}\langle 
\mathfrak{S}(i\kappa,A,B) 
\int_{\cG} \Phi(y,i\kappa,\au)^T \psi(y) dy, \int_{\cG}\Phi(x,i\kappa,\au)^T\varphi(x) dx \rangle_{\cK} \\
&+ \frac{1}{2\kappa}\sum_{n=1}^\infty
\langle 
\left[\mathfrak{S}(i\kappa,A,B) T(i\kappa;\au)\right]^{n}\mathfrak{S}(i\kappa,A,B)
\int_{\cG} \Phi(y,i\kappa,\au)^T \psi(y) dy, \int_{\cG}\Phi(x,i\kappa,\au)^T\varphi(x) dx \rangle_{\cK}.
\end{align*}
By a similar argument as above, there exists $c_q'>0$ such that
 \begin{align*}
 \norm{\left[\mathfrak{S}(i\kappa,A,B) T(i\kappa;\au)\right]^{n}\mathfrak{S}(i\kappa,A,B)} \leq q^{n} \quad \hbox{for} \quad \kappa \geq c_q',
 \end{align*}
 and hence for $\kappa \geq \max\{c_q', c_q\}$
 \begin{align*}
&\frac{1}{2\kappa}|\sum_{n=1}^\infty
 \langle 
 \left[\mathfrak{S}(i\kappa,A,B) T(i\kappa;\au)\right]^{n}\mathfrak{S}(i\kappa,A,B)
 \int_{\cG} \Phi(y,i\kappa,\au)^T \psi(y) dy, \int_{\cG}\Phi(x,i\kappa,\au)^T\psi(x) dx \rangle_{\cK}| \\
 &\leq 
 \frac{q}{2\kappa(1-q)}  \norm{\Phi(\cdot,i\kappa,\au)}^2 \norm{\psi} \norm{\varphi} \leq \frac{C}{\kappa^2} \quad \hbox{as} \quad \kappa \to\infty
 \end{align*}
 for some constant $C>0$, where one uses that
 \begin{align*}
 \norm{\phi_{\cE}(\cdot, i\kappa)}^2= \frac{1}{2\kappa}, \quad  \norm{\phi_{+}(\cdot, i\kappa)}^2= \frac{1}{2\kappa}(1-e^{-2\kappa a_{\min}}), \quad \norm{\phi_{+}(\cdot, i\kappa)}^2= \frac{1}{2\kappa}(1-e^{-2\kappa a_{\min}}).
 \end{align*}
 So, one can now focus on the leading term
 \begin{align*}
 \frac{1}{2\kappa}\langle 
 \mathfrak{S}(i\kappa,A,B) 
 \int_{\cG} \Phi(y,i\kappa,\au)^T \psi(y) dy, \int_{\cG}\Phi(x,i\kappa,\au)^T\varphi(x) dx \rangle_{\cK}.
 \end{align*}
 Consider for $\alpha^{\cE} = \{\alpha_j^{\cE}\}_{j \in \cE}$, $\alpha^{-} = \{\alpha_j^-\}_{j \in \cI}$, $\alpha^{+} = \{\alpha_j^+\}_{j \in \cI}$ 
 \begin{align*}
\underline{\alpha} = \begin{bmatrix}\alpha^{\cE} \\ \alpha^- \\ \alpha^+
\end{bmatrix} \quad \hbox{and} \quad
 \psi_{\underline{\alpha}}(x) = \begin{bmatrix}
 \{\alpha_j^{\cE} e^{-\kappa x_j}\}_{j \in \cE} \\
 \{\alpha_j^{-} e^{-\kappa x_j}\}_{j \in \cI} +
 \{\alpha_j^{+} e^{-\kappa(a_j-x_j)}\}_{j \in \cI}
 \end{bmatrix}.
 \end{align*}
 Then $\psi_{\underline{\alpha}}\in L^2(\cG,\au)$,
 \begin{align}\label{eq:underlinealpha}
 \begin{split}
 \norm{\psi_{\underline{\alpha}}}^2 &= \frac{1}{2\kappa}\sum_{j\in \cE}|\alpha_j^{\cE}|^2 + \frac{1}{2\kappa}\sum_{j\in \cI} (1-e^{-2\kappa a_{j}})|\alpha_j^{+}|^2  + \frac{1}{2\kappa}(1-e^{-2\kappa a_{j}})|\alpha_j^{-}|^2 \\
&\qquad + \sum_{j\in \cI} a_j e^{-\kappa a_j}2 \Re \alpha_j^+\overline{\alpha_j^-},
\end{split}
 \end{align}
 and
 \begin{align*}
  \int_{\cG} \Phi(y,i\kappa,\au)^T  \psi_{\underline{\alpha}}(y) dy  = H(i\kappa,\au) \underline{\alpha}, \quad H(i\kappa,\au) = \begin{bmatrix}
  \tfrac{1}{2\kappa}\mathds{1} & 0 & 0 \\
  0 & \tfrac{1}{2\kappa}(\mathds{1}-e^{-2\kappa \au})&  \au e^{-\kappa a} \\
  0 & \au e^{-\kappa a} &  \tfrac{1}{2\kappa}(\mathds{1}-e^{-2\kappa \au})
  \end{bmatrix}.
 \end{align*}
 Now, for $N_B\neq 0$ from the quasi-Weierstrass normal form one chooses for $v\in \cK$ with $\norm{v}=1$, $N_Bv\neq 0$, $N_B^2 v=0$, and $G$ as in \eqref{eq:S_blockmatrix} for $\kappa>0$
 \begin{align*}
 \underline{\alpha}=\underline{\alpha}(\kappa)= H(i\kappa,\au)^{-1}G^{-1}v \quad \hbox{and}\quad 
 \underline{\beta}=\underline{\beta}(\kappa)= H(i\kappa,\au)^{-1}G^{\ast}N_Bv, 
 \end{align*}
 then using \eqref{eq:SN}
  \begin{multline*}
 \frac{1}{2\kappa}\langle 
 \mathfrak{S}(i\kappa,A,B) 
 \int_{\cG} \Phi(y,i\kappa,\au)^T \psi_{\underline{\alpha}}(y) dy, \int_{\cG}\Phi(x,i\kappa,\au)^T\psi_{\underline{\beta}}(x) dx \rangle_{\cK} 
 =-\frac{1}{2\kappa}\langle 
 v +2\kappa N_Bv,N_Bv \rangle_{\cK}.
 \end{multline*}
 Using that  $H(i\kappa,\au)^{-1}$ for $\kappa>0$ has the eigenvalues
 \begin{align*}
 2\kappa, \quad \left(\tfrac{1}{2\kappa}(\mathds{1}-e^{-2\kappa a_j}+ 2\kappa a_j e^{-\kappa a_j})\right)^{-1} \quad \hbox{and} \quad \left(\tfrac{1}{2\kappa}(\mathds{1}-e^{-2\kappa a_j} - 2\kappa a_j e^{-\kappa a_j})\right)^{-1} \quad \hbox{for } j \in \cI
 \end{align*}
it follows using the symmetry of $H(i\kappa,\au)$  for $\kappa$ large that
\begin{align*}
 \norm{H(i\kappa,\au)^{-1}} \simeq \kappa, \quad \hbox{and hence also} \quad  \norm{\underline{\alpha}}, \norm{\underline{\beta}} \simeq \kappa.
\end{align*}
Then for $\kappa \to \infty$ using the exponential decay of the term $a_j e^{-\kappa a_j}2 \Re \alpha_j^+\overline{\alpha_j^-}$ in \eqref{eq:underlinealpha} it follows for $\kappa$ large that 
 \begin{align*}
\norm{\psi_{\underline{\alpha}}}, \norm{\psi_{\underline{\beta}}} \simeq \tfrac{\kappa}{\sqrt{\kappa}}= \sqrt{\kappa} \quad \hbox{and} \quad  \norm{\psi_{\underline{\alpha}}}^{-1}, \norm{\psi_{\underline{\beta}}}^{-1} \simeq \sqrt{\kappa}.
\end{align*}
 Normalizing $\varphi_{\underline{\alpha}}=\psi_{\underline{\alpha}}/\norm{\psi_{\underline{\alpha}}}$ and $\varphi_{\underline{\beta}}=\psi_{\underline{\beta}}/\norm{\psi_{\underline{\beta}}}$, one obtains as $\kappa \to \infty$
 \begin{align*}
 \tfrac{1}{\kappa}
 &\lesssim 
\frac{1}{2\kappa^{2}} |\langle 
 v +2\kappa N_Bv,N_Bv \rangle_{\cK}|\\
 &\lesssim 
 \frac{1}{2\kappa}\frac{1}{\norm{\psi_{\underline{\alpha}}}\norm{\psi_{\underline{\beta}}}}|\langle 
 v +2\kappa N_Bv,N_Bv \rangle_{\cK}|\\
&=
 \frac{1}{2\kappa}|\langle 
 \mathfrak{S}(k,A,B) 
\int_{\cG} \Phi(y,k,\au)^T \varphi_{\underline{\alpha}}(y) dy, \int_{\cG}\Phi(x,k,\au)^T\varphi_{\underline{\beta}}(x) dx \rangle_{\cK}|.
 \end{align*}
In particular,  since the other terms discussed above go to zero of order $\cO(\tfrac{1}{\kappa^2})$ or faster  
 \begin{align*}
 \tfrac{1}{\kappa} + \cO(\tfrac{1}{\kappa^2})
 &\lesssim 
\abs{\langle \int_{\cG} r_{\cM}(\cdot,y;i\kappa) \varphi_{\underline{\alpha}}(y) dy, \varphi_{\underline{\beta}}\rangle} \\
 &\leq \sup_{0\neq \varphi\in L^2(\cG,\au)}  \frac{\abs{\langle \int_{\cG} r_{\cM}(\cdot,y;i\kappa) \varphi_{\underline{\alpha}}(y)dyk, \varphi\rangle  }}{\norm{\varphi}} \\
 &=  \norm{\int_{\cG} r_{\cM}(\cdot,y;i\kappa) \varphi_{\underline{\alpha}}(y) dy} 
 \leq \norm{(-\Delta(A,B)+\kappa^2)^{-1}} \quad \hbox{as } \kappa \to \infty.
 \end{align*}

\end{proof}

  \subsection{Resolvent estimates for irregular boundary condition} \label{subsec:irr}
Set as before $a_{\min}:=\min_{i\in \cI}a_i$ if $\cI\neq\emptyset$ and $a_{\min}>0$ arbitrary if $\cI=\emptyset$.
 \begin{lemma}[Lower bound on the resolvent for irregular boundary conditions]\label{lemma:res_irr}
 	Let $A,B$ be irregular boundary conditions, then
 	\begin{align*}
\frac{|e^{\Im k a_{\min}/2}-3a_{\min}|}{a_{\min}|k^2|}   \lesssim	\norm{(-\Delta(A,B)-k^2)^{-1}}  \quad \hbox{for } \Im k>0,
 	\end{align*}
 	where one uses the convention $\norm{(-\Delta(A,B)-k^2)^{-1}}=\infty$ if $k^2\in \sigma(-\Delta(A,B))$.
 \end{lemma}
 
 The following example is an essential ingredient for the proof of Lemma~\ref{lemma:res_irr} which is given below after Example~\ref{ex:emptyspec}.
 \begin{example}\label{ex:emptyspec}
 	Consider the interval $[0,a]$ and the irregular boundary conditions defined by 
 	\begin{eqnarray*}
 		A=\begin{bmatrix}1 & 0 \\ 0 & 0  \end{bmatrix}  & \mbox{and} & B=\begin{bmatrix} 0 & 0 \\ 1 & 0  \end{bmatrix}.
 	\end{eqnarray*}
 	Then $\dim\cM(A,B)=2$ and the boundary conditions correspond to
 	\begin{eqnarray*}
 		\psi(0)=0 & \mbox{and} & \psi^{\prime}(0)=0.
 	\end{eqnarray*}
 	This example is discussed in \cite[Sec.~XIX.6(b)]{DSIII} as \emph{totally degenerate boundary conditions}.
    Note that
 	\begin{align*}
 	-\Delta(A,B)= -D_0^2, \quad \hbox{where} \quad
 	D_0\psi = \psi' \quad \hbox{with} \quad \Dom(D_0)=\{\psi\in H^1([0,a])\colon \psi(0)=0\},
 	\end{align*}  
 	and $\sigma(D_0)=\emptyset$.  
 	Hence for any $k\in \IC$
 	\begin{align*}
 	(-\Delta(A,B)-k^2)^{-1} =[-(D_0-ik)(D_0+ik)]^{-1} = - (D_0+ik)^{-1}(D_0-ik)^{-1}.
 	\end{align*}
 	The resolvent for $D_0$ is given by the well-known variation of constants formula
 	\begin{align*}
 	(D_0\pm ik)^{-1}\psi(x) = \int_0^x e^{\mp ik (x-y)} \psi(y) dy, \quad k\in \IC. 
 	\end{align*}
 	Inserting for instance $\psi\equiv 1$, one obtains
 	\begin{align*}
 	(-\Delta(A,B)-k^2)^{-1}\psi= \frac{1}{k^2}\left(  \cos(kx)-1   \right), \quad k \in \IC\setminus \{0\}.
 	\end{align*}
 \end{example}

 \begin{proof}[Proof of Lemma~\ref{lemma:res_irr}]
 	Note that for $\cI=\emptyset$ one has $\rho(-\Delta(A,B))=\emptyset$, since the eigenvalue equation reduces to $\det(A+ik B)=0$ for $\Im k>0$, and hence for irregular boundary conditions $\sigma(-\Delta(A,B))=\IC$, see \cite[Section 3.4]{HKS2015}, and hence the claim follows.
 	
 	So let $\cI\neq \emptyset$.
 Then, the first observation is that the function $\psi_{a}$ defined by
\begin{align*}
\psi_{a}(x;k)=\frac{1}{k^2}\left(  \cos(k(a-x))-1   \right)\chi_{[0,a]},  \quad  \psi_{a}(\cdot;k)\in H^2([0,\infty)) \quad\hbox{for} \quad k\in \IC\setminus\{0\}, \quad a>0,
\end{align*}
because $\psi_{a}(a;k)=\psi_{a}'(a;k)=0$.
Now, let 
\begin{align*}
v\in (\Ker A \cap \Ker B)\setminus\{0\} \quad \hbox{with } v=[\{ v_j^{\cE}\}_{j\in \cE}, \{ v_j^{-}\}_{j\in \cI}, \{ v_j^{+}\}_{j\in \cI}] \hbox{ and } \norm{v}=1,
\end{align*}
and then consider the function $\psi_v\in H^2(\cG)$ defined by
 \begin{align*}
 \psi_v(x_j)=\begin{cases}
 v_j^{\cE} \psi_{a_{\min/2}}(x;k), & j\in \cE, \\
 v_j^{-} \psi_{a_{\min/2}}(x_j;k) + v_j^{+} \psi_{a_{\min/2}}(a_j-x_j;k), & j\in \cI.
 \end{cases}
 \end{align*}
This satisfies
\begin{align*}
\underline{\varphi_v} =  \psi_{a}(0;k) \cdot v \quad \hbox{and} \quad
\underline{\varphi_v}' = \psi_{a}'(0;k) \cdot v, 
\end{align*}
 and hence $\varphi_v\in \Dom(\Delta(A,B))$, and moreover it solves
 \begin{align*}
 (-\Delta(A,B)-k^2)  \psi_v = \varphi_v \quad \hbox{for} \quad 
 \varphi_v(x_j)=
 \begin{cases}
 v_j^{\cE} \chi_{[0,a_{\min}/2]}, & j\in \cE, \\
 v_j^{-} \chi_{[0,a_{\min}/2]} + v_j^{+} \chi_{[a_j-a_{\min}/2], a_j}, & j\in \cI.
 \end{cases} 
 \end{align*}
 So, if $k^2\in \sigma(-\Delta(A,B))$, then one sets $\norm{(-\Delta(A,B)-k^2)^{-1}}=\infty$ and if $k^2\in \rho(-\Delta(A,B))$, then $\psi_v=(-\Delta(A,B)-k^2)^{-1}\varphi_v$, and
 \begin{align*}
 \norm{\varphi_v}=(a_{\min}/2)\norm{v} \quad \hbox{and}\quad 
 \norm{v} \frac{|e^{\Im k a_{\min}}-3a_{\min}|}{2|k^2|}   \lesssim  \norm{\psi_v}.
 \end{align*}
 Hence the claim follows as in Example~\ref{ex:emptyspec}.
 \end{proof}

  \subsection{Proof of Theorem~\ref{thm:semigroup}} \label{subsec:proof_thm_semigroup}
  By the classifications of boundary conditions one can discuss now each possible case to show that only quasi-sectorial boundary conditions lead to the desired properties. 
  
  First, if $\dim \cM\neq d$, then $$\sigma(-\Delta(\cM))=\IC$$ according to \cite[Proposition~4.2]{HKS2015}, and hence $\Delta(\cM)$ cannot be a generator. For $A,B$ with $\dim \cM(A,B)= d$ defining irregular boundary conditions or regular boundary conditions which are not quasi-sectorial by Lemma~\ref{lemma:res_irr} 
  and Lemma~\ref{lemma:res}, respectively, one has
  \begin{align*}
  \norm{(-\Delta(A,B)+\kappa^2)^{-1}}\geq C/\kappa \quad \hbox{as} \quad \kappa \to \infty,
  \end{align*} 
  and therefore  it cannot be a generator. Now, if $A,B$ define quasi-sectorial boundary conditions, then by Lemma~\ref{lemma:quasi-sectorial}, $\Delta(A,B)$ is 
  the unique operator associated with the densely defined, closed, sectorial form $\delta_{A,B}$. Hence it generate an analytic semigroup, cf. e.g. \cite[Theorem 6.15]{Mug_book}, in $(L^2(\cG,\au), \langle \cdot,\cdot \rangle_{G,\chi})$ and by the equivalence of norms, see Lemma~\ref{lemma:skp}, it generates an analytic semigroup in $L^2(\cG,\au)$ with the standard scalar product, and 
  in particular a $C_0$-semigroup.
 Since the form is of Lions type it follows that $-\Delta(A,B)$
  is the the generator of an $C_0$-cosine operator function, cf. \cite[Theorem 6.18]{Mug_book}, and by equivalence of norms this carries over to $L^2(\cG,\au)$.
  \qed
  
 Combining Proposition~\ref{prop:spec} with Theorem~\ref{thm:semigroup} one obtains 
  \begin{corollary}\label{cor:normal}
  	If $-\Delta(A,B)$ is similar to a normal operator, then $A,B$ define quasi-sectorial boundary conditions.
  \end{corollary}
  \begin{proof}
  	Assume that $-\Delta(A,B)$ is similar to a normal operator, then there is an equivalent norm $\norm{\cdot}_{n}$ such that 
  	\begin{align*}
  	\norm{(-\Delta(A,B)-k^2)^{-1}}_n = \dist(k^2, \sigma(-\Delta(A,B))), \quad k^2\in \rho(-\Delta(A,B)).
  	\end{align*}
  	since $\sigma(-\Delta(A,B))$ by Proposition~\ref{prop:spec} is contained  in a sector, this implies that $-\Delta(A,B)$ -- taking into account equivalence of norms -- generates an analytic semigroup in $L^2(\cG,\au)$ which by Theorem~\ref{thm:semigroup} implies that $-\Delta(A,B)$ is quasi-sectorial. 
  \end{proof}
 \begin{remark}[Normal Laplacians are already self-adjoint]\label{rem:normalbounded}
 	A normal Laplacians in $L^2(\cG,\au)$ satisfies
 	\begin{align*}
 	\Delta(A,B)^\ast\Delta(A,B) = \Delta(A,B)\Delta(A,B)^\ast,
 	\end{align*}
 	where both operators are self-adjoint bi-Laplacians. By the extension theory for self-adjoint bi-Laplacians which is analogous to the one for Laplacians, see e.g. \cite{MugnoloGregorio_2020} and the references therein, it follows that this holds if and only if $\Dom(\Delta(A,B)^\ast)=\Dom(\Delta(A,B))$, \ie, the operator is already self-adjoint. This seems to be a general feature for extensions of closed symmetric operators.
\end{remark}

\begin{remark}
	To prove the generation of analytic semigroups and $C_0$-cosine operator function, instead of using the form $\delta_{A,B}$ constructed in Subsection~\ref{subsec:forms}, one can use also the explicit formula for the Green's function given in Subsection~\ref{subsec:resolvent_reg} to prove suitable resolvent estimates. 
\end{remark}

\subsection{Spectra on compact graphs}
If $\cG$ is compact, \ie, $\cE=\emptyset$,  then  $\sigma(-\Delta(A,B))$ is discrete if $\rho(-\Delta(A,B))\neq \emptyset$ because of the compact embedding $\Dom(-\Delta(A,B))\hookrightarrow L^2(\cG,\au)$. 
For an unbounded non-self-adjoint operator it is in general not self-evident that the resolvent set or the spectrum is non-empty. For irregular boundary conditions, there both cases can occur, compare e.g. \cite[Example 3.2 and Example 3.1]{HKS2015}, and for regular boundary conditions one has by Proposition~\ref{prop:spec} that the resolvent set is non-empty.

The \emph{spectral projection} on $\gamma$ for a compact set $\gamma\subset \sigma(-\Delta(A,B))$ is
defined by
\begin{align}\label{eq:EAB}
E_{A,B}(\gamma) =  \frac{1}{2\pi i}\int_{\Gamma} (-\Delta(A,B)-z)^{-1} dz,
\end{align}
where $\Gamma$ is a closed Jordan curve constructed such that it is
positively  oriented, and it bounds a finite
domain containing every point of $\gamma$ and no point of $\sigma(-\Delta(A,B))\setminus \gamma$. 
The range of the finite dimensional range of $E_{A,B}(\lambda_i)$ consists of the span of the eigenvectors and generalized eigenvectors to $\lambda_i\in \sigma_p(-\Delta(A,B))$.

\begin{proposition}\label{thm:riesz}
	Let $\cI=\emptyset$, \ie, $(\cG,\au)$ is compact.
	If the boundary conditions defined by $A,B$ are regular, then the set of generalized eigenvectors of $-\Delta(A,B)$ is complete in $L^2(\cG,\au)$.
\end{proposition}

\begin{proof}
	There is a family of results on the completeness of the set of root vectors of Hilbert-Schmidt operators, see for instance \cite[Corollary IX.31]{DSIII} or \cite[Theorem 2.6.2]{Locker2000} based on the decay or at least polynomial growth of the resolvent. Here, the version of \cite[Lemma 2]{Agranovich1994} is applied to $-\Delta(A,B)$ in the separable Hilbert space $L^2(\cG,\au)$.

	For $\cE=\emptyset$, the resolvent $(-\Delta(A,B)-k^2)^{-1}$ is compact for $k^2\in \rho(-\Delta(A,B))$, and the singular values of $(-\Delta(A,B)-k^2)^{-1}$ are the eigenvalues 
	\begin{align*}
	\{s_j(k^2;A,B)\}_{j\in \IN} \quad \hbox{with}\quad s_1(k^2;A,B) \geq s_2(k^2;A,B) \geq \ldots \geq 0 \quad \hbox{counting multiplicties}
	\end{align*}
	of the non-negative self-adjoint compact operator
	\begin{align*}
	|(-\Delta(A,B)-k^2)^{-1}|=(((-\Delta(A,B)-k^2)^{-1})^\ast(-\Delta(A,B)-k^2)^{-1})^{1/2} \quad \hbox{for }  k^2\in \rho(-\Delta(A,B)).
	\end{align*}
	
	\begin{lemma}\label{lemma:hilbertschmidt}
		If $\cE=\emptyset$, and  $A,B$ define regular boundary conditions, then 
		$(-\Delta(A,B)-k^2)^{-1}$ is a Hilbert-Schmidt operator for any $k^2\in \rho(-\Delta(A,B))$, and 
		\begin{align*}
		s_j(k^2;A,B)  = \cO(1/j^2) \quad \hbox{as} \quad j\to \infty.
		\end{align*}
	\end{lemma}	
	\begin{proof}
		This follows already from the fact that the integral kernel of the resolvent, see Subsection~\ref{subsec:resolvent_reg}, is continuous and for $\cE=\emptyset$ the graph is compact.
		
		Note that for self-adjoint boundary conditions $A_{sa},B_{sa}$ one can apply the classical Dirichlet-Neumann bracketing method, compare e.g. \cite[Proposition 4.2]{Bolte2009}, to obtain that the eigenvalues of $-\Delta(A_{sa},B_{sa})$ behave as $\cO(j^2)$ as $j\to \infty$, and hence the singular values for any $k\in \rho (-\Delta(A_{sa},B_{sa}))$ behave as $\cO(1/j^2)$ as $j\to \infty$.
		If $A,B$ defining regular boundary conditions, it follows by Proposition~\ref{prop:spec} that 
		for any self-adjoint boundary conditions $\rho(-\Delta(A,B))\cap \rho(-\Delta(A_{sa},B_{sa}))\neq \emptyset$. 
		For $k^2\in \rho(-\Delta(A,B))\cap \rho(-\Delta(A_{sa},B_{sa}))$ the resolvent formula from Subsection~\ref{subsec:resolvent_reg} applies, and 
		\begin{align*}
		(-\Delta(A,B)-k^2)^{-1} - (-\Delta(A_{sa},B_{sa})-k^2)^{-1}\psi = \int_{\cG} (r^1_{A,B}(\cdot,y;k)-r^1_{A_{sa},B_{sa}}(\cdot,y;k))\psi(y) dy
		\end{align*}
		is a finite rank operator with dimension of its range smaller equal to $d$ since 
		\begin{align*}
		\int_{\cG} r^1_{A,B}(\cdot,y;k)\psi(y) dy \in \hbox{span} \{\psi\in L^2(\cG,\au)\colon \psi_j=\alpha_j^+ e^{ik_\lambda x_j}+\alpha_j^-e^{ik_\lambda x_j}, \quad \alpha_j^\pm \in \IC\}.
		\end{align*}
		Therefore, compare e.g. \cite[Corollary II.2.1]{GohbergKrein}, also 
		\begin{align*}
		s_{j-d}(k^2;A_{sa},B_{sa}) \leq s_j(k^2;A,B) \leq s_{j+d}(k^2;A_{sa},B_{sa})\quad \hbox{for} \quad j\geq d
		\end{align*}
		and hence
		$s_j(k^2;A,B)  = \cO(1/j^2)$ as $j\to \infty$, and by the  first resolvent identity   this follows for all $k^2\in \rho(-\Delta(A,B))$.
	\end{proof}

	Continuing the proof of Proposition~\ref{thm:riesz}, by Lemma~\ref{lemma:hilbertschmidt} one has
	\begin{align*}
	\liminf_{j\to \infty} s_j(k^2;A,B) j^{2} >0.
	\end{align*}
	Hence, one needs 5 rays $\gamma_1, \ldots, \gamma_5$ with angles between adjacent lines being smaller than $\pi/2$ such that there exists an $N\geq -1$ such that
	\begin{align*}
	\norm{(-\Delta(A,B)-\lambda)^{-1}}=\cO(N) \quad \hbox{as} \quad \lambda\to \infty  \quad \hbox{along each ray } \gamma_j, \quad j=1, \ldots, 5.
	\end{align*}
	These can be chosen by Lemma~\ref{lemma:res} and Proposition~\ref{prop:spec} as illustrated in Figure~\ref{fig:rays}. Hence by \cite[Lemma 2]{Agranovich1994} the system of root vectors of $-\Delta(A,B)$ is complete.
\end{proof}

\begin{figure}[h]
	\begin{center}
		\begin{tikzpicture}[scale=2]
		\fill [gray!20!white]
		(-0.5, 0) -- (1.2, 0.4) -- (1.2, -0.4) -- (-0.5,0) 
		-- cycle;
		\draw[black, thick, ->] (0,0) -- (0.6,0.3);
		\draw[black, thick] (0.6,0.3) -- (1.2,0.6);
		\draw[black, thick, ->] (0,0) -- (0.6,-0.3);
		\draw[black, thick] (0.6,-0.3) -- (1.2,-0.6);
		\draw[black, thick, ->] (0,0) -- (-0.1, 0.4);
		\draw[black, thick] (-0.1,0.4) -- (-0.2, 0.8);
		\draw[black, thick, ->] (0,0) -- (-0.1, -0.4);
		\draw[black, thick] (-0.1, -0.4) -- (-0.2, -0.8);
		\draw[black, thick, ->] (0,0) -- (-0.6, 0);
		\draw[black, thick] (-0.6, 0) -- (-1.2, 0);
		\draw[->] (-1.2,0) -- (1.2,0) node[right] {\tiny{$\Re k^2$}};
		\draw[->] (0,-0.2) -- (0,0.5) node[above right] {\tiny{$\Im k^2$}};
		\draw (1,0.5) node[right] {\tiny{$\IC$}};
		\end{tikzpicture}
	\end{center}
	\caption{Five rays with angles smaller $\pi/2$}\label{fig:rays}
\end{figure}
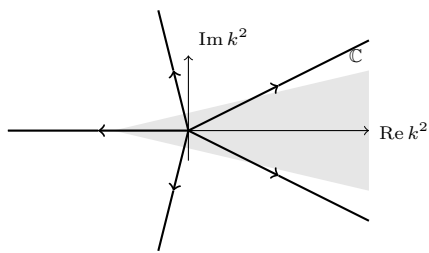    

\begin{corollary}\label{cor:nonemptyspec}
	Let $A,B$ be regular boundary conditions, then $\sigma(-\Delta(A,B))\neq \emptyset$, and $\sigma_p(-\Delta(A,B))$ is an at most countable infinite set, and for $\cE=\emptyset$, it is a countable infinite set. 
\end{corollary}
\begin{proof}
	The non-zero eigenvalues correspond to the zeros of the non-constant holomorphic function $k\mapsto \det(Z(k;A,B,\au))$, see \cite[Section 4.1]{HKS2015} and Lemma~\ref{prop:spec}, and hence the set of eigenvalues is at most countable infinite. 
	For $\cE=\emptyset$, by Theorem~\ref{thm:riesz} (a) the set of generalized eigenvalues is complete. Since $L^2(\cG,\au)$ is infinite dimensional, it follows that the dimension of the space spanned by generalized eigenvectors is infinite. As the resolvent is compact, the multiplicity of each eigenvalue is finite and hence also the set of eigenvalues. Since for $\cE\neq \emptyset$ one has that $[0,\infty)\subset \sigma(-\Delta(A,B))$, see \cite[Proposition 4.11]{HKS2015}, it follows in particular that the spectrum is non-empty.
\end{proof}

Combing this corollary with the results from \cite[Section 3 and 4]{HKS2015} one can summarize the basic spectral properties of non-self-adjoint Laplacians with regular boundary conditions on finite metric  graphs in Table~\ref{Tab:summary}.
\begin{table}[h]
	\begin{tabular}{c c c c}
		$\cI=\emptyset$	& $\cE=\emptyset$ & $\cI\neq\emptyset, \cE\neq\emptyset$ & \\ \hline 
		$\sigma_r(-\Delta(A,B))=\emptyset$ & $\sigma_r(-\Delta(A,B))=\emptyset$ & $\sigma_r(-\Delta(A,B))\subset [0,\infty)$  & \\
		& & or $\sigma_r(-\Delta(A,B))=\emptyset$ & 
		\\
		$\sigma_{e}(-\Delta(A,B)) =[0,\infty)$ & $\sigma_{e}(-\Delta(A,B))=\emptyset$ & $\sigma_{e}(-\Delta(A,B)) =[0,\infty)$ & \\
		$\sigma_{p}(-\Delta(A,B))\subset \IC\setminus[0,\infty)$ & $\sigma(-\Delta(A,B))$ discrete & $\sigma_{p}(-\Delta(A,B))$ discrete & \\
		$\sigma_{p}(-\Delta(A,B))$ finite & countable infinite  & at most countable infinite & \\
		and finite multiplicity  & eigenvalues& eigenvalues \\
		& & &
	\end{tabular}
	\caption{Summary of spectral properties for regular boundary conditions}\label{Tab:summary}
\end{table}

\begin{example}
	Consider the interval $[0,1]$, and the regular, non-quasi-sectorial boundary conditions 
	\begin{eqnarray*}
		A=\begin{bmatrix}1 & 0 \\ 0 & 1  \end{bmatrix}  & \mbox{and} & B=\begin{bmatrix} 0 & 0 \\ -1 & 0  \end{bmatrix},
	\end{eqnarray*}
discussed in Subsection~\ref{exDS}.
	The spectrum of $-\Delta(A,B)$ consists only of eigenvalues of geometric multiplicity one, where each eigenvalue is a solution of 
	\begin{align*}
	\sin(k)=k, \quad k\in \IC,
	\end{align*}
	with eigenfunction $\sin(kx)$ for $k\neq 0$ and $x$ for $k=0$,
	and a direct computation shows that there are no cyclic vectors. One can ask if the set of eigenvectors forms even a Riesz basis. 
	According to \cite[Ex.~XIX.6(d)]{DSIII} the eigenvalues are located asymptotically  at  the points
	\begin{align*}
	an^2 + ib n\ln n + \ldots , \quad n\in \IN,
	\end{align*}
where the ratio  of $a$ and $b$ is real. The question when on an interval functions of the form $e^{ ikx}_{k\in \cC}$ form a Riesz basis for some set $\cC\subset \IC$ has been characterized by Pavlov in  \cite{Pavlov1979}, and  one of the conditions is that $|\Im k| \leq c$ for some $c>0$ and all $k\in \cC$. So, 
the eigenvectors to this problem are complete but do not form a Riesz basis. 
\end{example}

\begin{example}[Laplacian with nilpotent part]
	One can ask if for regular boundary conditions cyclic vectors can occur at all. To construct  a Laplacian with nilpotent part 
	let $N\in \IC^{n\times n}$ be a nilpotent matrix with
	\begin{align*}
	N^n=0 \quad \hbox{and} \quad N^k\neq 0 \quad \hbox{for } k=1, \ldots, n-1. 
	\end{align*} 
	Consider the ``pumpkin graph'' with $n$ edges each of length $a>0$ sketched in Figure~\ref{fig:punpkin_graph}. Then one has $d=2|\cI|=2n$, and one can define quasi-sectorial boundary conditions in the block form
	\begin{align*}
	A= \begin{bmatrix}
	-N & 0 \\ 0 & N 
	\end{bmatrix} \quad \hbox{and}\quad B= \begin{bmatrix}
	\mathds{1} & 0 \\ 0 & \mathds{1}
	\end{bmatrix}.
	\end{align*}
	For $v \in \IC^n$ one can define the function by
	\begin{align*}
	\varphi_i(x) = (e^{N x}v)_i, \quad i \in \cI,  \quad \hbox{where} \quad \partial_{x}^{m}\varphi= N^{m} \varphi, \quad m\in \IN_0.
	\end{align*}
	Then
	\begin{align*}
	\underline{\varphi} = \begin{bmatrix}
	v \\ \{e^{N a}v_i\}_{i\in I} 
	\end{bmatrix}\quad \hbox{and}\quad 
	\underline{\varphi}^\prime = \begin{bmatrix}
	Nv \\ -N\{e^{N a}v_i\}_{i\in I} 
	\end{bmatrix}.
	\end{align*}
	Hence $A\underline{\varphi} + B \underline{\varphi}^\prime=0$ and similarly $A\underline{\partial_{x}^{m}\varphi} + B \underline{\partial_{x}^{m}\varphi}^\prime=0$. So, 
	\begin{align*}
	\Delta(A,B)^m\varphi \in D(\Delta(A,B)) \quad \hbox{for all } m\in \IN_0, \quad \hbox{and} \quad \Delta(A,B)^{ [k/2]}\varphi = N^{2[k/2]} \varphi =0 \hbox{ for } k\geq n. 
	\end{align*}
	For $n\geq 3$, $\Delta(A,B)\varphi= N^{2}\varphi\neq 0$, so $\varphi$ defines a cyclic vector to the eigenvalue $0$.
\end{example}

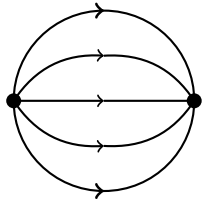
\begin{figure}[h]
	\begin{center}
		\begin{tikzpicture}[scale=0.6]
		\fill[black] (0,0) circle (1ex);
		\fill[black] (-4,0) circle (1ex);
		\draw[->, black, thick] (-4,0) -- (-2,0);
		\draw[black, thick] (-2,0) -- (0,0);
		\draw[thick, ->] (-4,0) to[bend left] (-2,1);
		\draw[thick] (-2,1) to[bend left] (0,0);
		\draw[thick, ->] (-4,0) to[bend right] (-2,-1);
		\draw[thick] (-2,-1) to[bend right] (0,0);
		\draw[<-, black, very thick] (-2,2) -- (-2.1,2);	
		\draw[thick] (0,0) arc (0:-180:2);
		\draw[thick] (0,0) arc (0:180:2);
		\draw[<-, black, very thick] (-2,-2) -- (-2.1,-2);
		\end{tikzpicture}
		\caption{``Pumpkin graph''}\label{fig:punpkin_graph}
	\end{center}
\end{figure}

\section{Similarity to self-adjoint operators}\label{sec:similarity}
The Schrödinger equation 
\begin{align*}
\partial_t\psi  + i\Delta(A,B)\psi&=0, \quad t>0,\quad \psi(0)=\psi_0, 
\end{align*}
is well-posed if $-\Delta(A,B)$ is the generator of a bounded $C_0$-group or equivalently if $-\Delta(A,B)$ is similar to a self-adjoint operator. For $\delta$-potentials on the real line such  characterizations have been pioneered by the work of Mostafazadeh \cite{Mostafazadeh2006}, and further results for point interactions have been obtained in \cite{GrodKuzhel2014} and \cite{KuzhelZnojil2017}. These results -- obtained by a variety of methods such as Krein spaces and classical criteria for quasi-self-adjointness from \cite{Naboko1984, Malamud1985, Casteren1983}   -- point into a similar direction where similarity depends on the spectrum of $L$ in the parametrization \ref{eq:bc_LP}, and the orders of some spectral singularities. In this sense the following characterization is a generalization of Mostafazadeh's criterion, where the admissible range of spectra of $L$ is illustrated in Figure~\ref{fig:similarity}. A particular class of rotational invariant boundary conditions on star graphs  has been studied in \cite{Astudillo2015}.

\begin{theorem}[Similarity on star graphs]\label{thm:similarity}
	Let $\cG$ be a star graph with only external edges, \ie, $\cI=\emptyset$. Then $-\Delta(A,B)$ is similar to a self-adjoint operator if and only if $A,B$ define quasi-sectorial boundary conditions, for the matrix $L$ in the quasi-Weierstrass normal form
	\begin{align*}
	\sigma(L) \subset \{z\in \IC\colon \Re z< 0\} \cup [0,\infty),
	\end{align*}
	and for each $l\in \sigma(L) \cap [0,\infty)$ the algebraic and geometric multiplicity agree.
\end{theorem}
The proof of this theorem is given in Subsection~\ref{subsec:proof_thm_similarity} below.

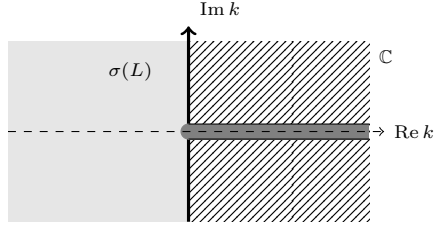
\begin{figure}[h]
	\begin{center}
		\begin{tikzpicture}[scale=2]
		\fill[gray!20!white]
		(0, -0.6) -- (-1.2, -0.6) -- (-1.2, 0.6) -- (0, 0.6) -- (0, -0.6)
		-- cycle;
		\pattern[pattern=north east lines]
		(0, -0.6) -- (1.2, -0.6) -- (1.2, 0.6) -- (0, 0.6) -- (0, -0.6)
		-- cycle;
		\draw[->, very thick] (0,-0.6) -- (0,0.7) node[above right] {\tiny{$\Im k$}};
		\draw (1.2,0.5) node[right] {\tiny{$\IC$}};
		\draw (-0.6,0.4) node[right] {\tiny{$\sigma(L)$}};
		\draw[gray, fill=gray] (0,0) circle (0.05);
		\fill [gray]
		(0, -0.05) -- (0, 0.05) -- (1.2, 0.05) -- (1.2, -0.05) -- (0, -0.05)
		-- cycle;
		\draw [thin]
		(0, -0.05) -- (1.2, -0.05);
		\draw [thin]
		(0, 0.05) -- (1.2, 0.05);
		\draw[dashed, ->] (-1.2,0) -- (1.3,0) node[right] {\tiny{$\Re k$}};
		\end{tikzpicture}
	\end{center}
	\caption{Location of $\sigma(L)$ for quasi-sectorial Laplacians similar to self-adjoint operators in grey and dark grey where conditions on the multiplicity are required}\label{fig:similarity}
\end{figure}

\begin{example}[Complex $\delta$-interaction]\label{ex:Cdelta}
	Consider a graph with $\cI=\emptyset$ and $\abs{\cE}\geq 2$. Assume that the boundary conditions are defined  by 
	\begin{align*}
	A= \left[
	\begin{array}{cccccc}
	1 & -1 & 0 &\cdots & 0 & 0 \\
	0 & 1 & -1 &\cdots & 0 & 0  \\
	0 & 0 & 1 &\cdots & 0 & 0  \\
	\vdots &\vdots  & \vdots & & \vdots  & \vdots \\
	0 & 0 & 0 &\cdots & 1 & -1  \\
	-\gamma & 0 & 0 &\cdots & 0 & 0
	\end{array}
	\right] \quad \hbox{and}\quad B= \left[
	\begin{array}{cccccc}
	0 & 0 & 0 &\cdots & 0 & 0 \\
	0 & 0 & 0 &\cdots & 0 & 0  \\
	0 & 0 & 0 &\cdots & 0 & 0  \\
	\vdots &\vdots  & \vdots & & \vdots  & \vdots \\
	0 & 0 & 0 &\cdots & 0 & 0  \\
	1 & 1 & 1 &\cdots & 1 & 1 
	\end{array}
	\right], \quad \gamma\in \IC.
	\end{align*}  
	One can represent the boundary conditions by equivalent boundary conditions of the form ~\eqref{eq:bc_LP} with $P=\mathds{1}-P^{\perp}$, where $P^{\perp}$ is the rank one projector onto $(\Ker B)^{\perp}$, and $L= -\frac{\gamma}{\abs{\cE}} P^{\perp}$, 
	\cf~\cite[Section~3.2.1]{Kuchment2004} for the case of real $\gamma$. 
	Theorem~\ref{thm:similarity} translates the result from \cite{Mostafazadeh2006}, see also \cite[Example III]{GrodKuzhel2014}, for the case $|\cE|=2$ to $|\cE|=d$, that is, $-\Delta(A,B)$ is similar to a self-adjoint operator if and only if 
	\begin{align*}
	\gamma \in \{z\in \IC\colon \Re z>0\} \cup (-\infty,0].
	\end{align*}
\end{example}

\begin{example}[Complex $\delta'$-interactions]\label{Cdeltaprime}
	As in Example~\ref{ex:Cdelta} consider a graph with $\cI=\emptyset$ and $\abs{\cE}\geq 2$, and now interchanging the boundary conditions, one considers
	\begin{align*}
	A=\left[
	\begin{array}{cccccc}
	0 & 0 & 0 &\cdots & 0 & 0 \\
	0 & 0 & 0 &\cdots & 0 & 0  \\
	0 & 0 & 0 &\cdots & 0 & 0  \\
	\vdots &\vdots  & \vdots & & \vdots  & \vdots \\
	0 & 0 & 0 &\cdots & 0 & 0  \\
	1 & 1 & 1 &\cdots & 1 & 1 
	\end{array}
	\right]\quad \hbox{and}\quad B= \left[
	\begin{array}{cccccc}
	1 & -1 & 0 &\cdots & 0 & 0 \\
	0 & 1 & -1 &\cdots & 0 & 0  \\
	0 & 0 & 1 &\cdots & 0 & 0  \\
	\vdots &\vdots  & \vdots & & \vdots  & \vdots \\
	0 & 0 & 0 &\cdots & 1 & -1  \\
	-\gamma & 0 & 0 &\cdots & 0 & 0
	\end{array}
	\right], \quad \gamma\in \IC,
	\end{align*}  
    compare \cite[Section 3.2.3]{Kuchment2004} for the case of real $\gamma$, and there are further $\delta^\prime$-type boundary conditions discussed for instance in \cite{Manko2010, Manko2015}.
	Equivalent boundary conditions are defined by $A=P_{\delta}^\perp$ and $B=L_{\delta} + P_{\delta}$, where $P_{\delta}=\mathds{1}-P_{\delta}^{\perp}$ with $P_{\delta}^{\perp}$ being the rank one projector from Example~\ref{ex:Cdelta}, and $L_{\delta}= -\frac{\gamma}{\abs{\cE}} P_{\delta}^{\perp}$. Hence, equivalent boundary conditions are 
	\begin{align*}
	A_{\delta'}= -\frac{\abs{\cE}}{\gamma} P_{\delta}^{\perp}  \quad \hbox{and} \quad B_{\delta'} = \mathds{1},  
	\end{align*}
	and therefore $L_{\delta'}= -\frac{\abs{\cE}}{\gamma} P_{\delta}^{\perp}$. 
	Theorem~\ref{thm:similarity} states that $-\Delta(A,B)$ is similar to a self-adjoint operator if and only if 
	\begin{align*}
	\gamma \in \{z\in \IC\colon \Re z>0\} \cup (-\infty,0],
	\end{align*}
	because $\gamma$ is in this region in $\IC$ if and only if $1/\gamma$ lies there.

	In \cite{GrodKuzhel2014} the couplings on the real line with
	\begin{align}\label{eq:Kuzhel}
	\frac{1}{2}\begin{bmatrix}
	a & b \\ c & d
	\end{bmatrix}
	\begin{bmatrix}
	\psi(0+) + \psi(0-) \\
	-\psi'(0+) - \psi'(0-)
	\end{bmatrix} =
		\begin{bmatrix}
		\psi'(0+) - \psi'(0-)\\
	\psi(0+) - \psi(0-)
	\end{bmatrix}, \quad a,b,c,d \in \IC,
	\end{align}
	are considered, see \cite[Lemma 2.1]{GrodKuzhel2014}. Identify $\IR$ with a graph with $\cI=\emptyset$ and $\cE=\{e_1, e_2\}$ with $e_1$ taking the role of $[0,\infty)$ and $e_2$  the one of $(-\infty,0]$ (where both $e_1,e_2$ are identified with $[0,\infty)$). Then $\psi(0+)=\psi_1(0)$, $\psi(0-)=\psi_2(0)$, and $\psi'(0+)=\psi_1'(0)$, $\psi'(0-)=-\psi_2'(0)$, and the coupling \eqref{eq:Kuzhel} translates to
	\begin{align*}
	A=\begin{bmatrix}
	\tfrac{a}{2} & \tfrac{a}{2} \\ \tfrac{c}{2}-1 & \tfrac{c}{2}+1
	\end{bmatrix}
	\quad \hbox{and}\quad 	B=\begin{bmatrix}
	-\tfrac{b}{2}-1 & \tfrac{b}{2}-1 \\ -\tfrac{d}{2} & \tfrac{d}{2}
	\end{bmatrix}.
	\end{align*}  
	   In particular the case of $\delta'$-type conditions for $a=b=c=0$ and $d\neq 0$ translates to
	\begin{align*}
	A=\begin{bmatrix}
0 & 0 \\ -1 & 1
\end{bmatrix}, 	B=\begin{bmatrix}
-1 & -1 \\ -\tfrac{d}{2} & \tfrac{d}{2}
\end{bmatrix}
	\quad \hbox{and equivalently} \quad
A'=\frac{1}{d}\begin{bmatrix}
1 & -1 \\ -1 & 1
\end{bmatrix}, 
B'=\begin{bmatrix}
1 & 0 \\ 0 & 1
\end{bmatrix}
\end{align*}
which is $A'= \frac{\sqrt{2}}{d} P_{\delta}^{\perp}$ and $B'= \mathds{1}$. This is
similar to a self-adjoint operator if and only if 
\begin{align*}
d \in \{z\in \IC\colon \Re z<0\} \cup [0,\infty)
\end{align*}
which is consistent with \cite[Example IV]{GrodKuzhel2014}, and the case $d=0$ is included since it defines the  self-adjoint Laplacian on $\IR$.
\end{example}

\subsection{Geometry of a graph and similarity transforms}\label{subsec:geometry}
Semigroup generation, compare Theorem~\ref{thm:semigroup},  such as self-adjointness, see e.g. \cite{KS1999}, can be verified locally considering boundary conditions at each vertex independent of the geometry.
However, similarity to self-adjointness involves not only the local boundary conditions but also and foremost the geometry, and the symmetry of the graph played an essential role in the proof of Theorem~\ref{thm:similarity}.  The following examples illustrate that similarity to self-adjoint operators can be caused by an interplay of the graph's geometry and the local boundary conditions, and the similarity can be prohibited as well by breaking such symmetries. In short for the generation of bounded $C_0$-groups the geometry matters!

\begin{example}[Broken symmetry]
	Consider the metric graph consisting of one internal edge of length $a>0$ and one external edge. Impose the following boundary conditions
	defined by 
	\begin{eqnarray*}
		A_{\tau}=\begin{bmatrix} 
			1 & -e^{i\tau} & 0\\0&  0 & 0 \\
			0&  0 & 1
		\end{bmatrix} &\mbox{and} & B_{\tau}=\begin{bmatrix}  0 & 0 & 0\\ 1 & e^{-i\tau} & 0 \\
			0&  0 & 0
		\end{bmatrix} \quad \hbox{for } \tau \in [0,\pi/2).
	\end{eqnarray*} 
	In this case the Cayley transform is similar to a unitary matrix, because the upper block corresponds to Example~\ref{ex1_quasi_weiserstrass}, and the lower block to Dirichlet boundary conditions. However, this similarity does not carry over to similarity of operators.
	The eigenvalue equation becomes
	\begin{align*}
	ik  e^{-ika}\cos(\tau) - ik  e^{ika}i\sin(\tau)=0 \quad \hbox{for }k\in  \{z\in \IC\colon \Im z>0 \hbox{ or } z\in (0,\infty)\}.
	\end{align*} 
	For $\tau\in (0,\pi/2)$ and $k=x+iy$ this is equivalent to
	\begin{align*}
	i\tan(\tau)= e^{2ika}=[\cos(2ax) - i\sin(2ax)]e^{2ya},
	\end{align*}
	and therefore $k^2$ is an eigenvalue if and only if $\tan(\tau) = \sin(2ax)e^{2ya}$ and $\cos(2ax)=0$.
	Hence all
	\begin{align*}
	x\in \frac{\pi}{2a}(3/2+2\IZ) \quad  \hbox{and} \quad y= \frac{1}{2a}\ln (\tan(\tau))
	\end{align*} 
	define solutions to the secular equation with non-trivial imaginary part. So, although there is a similarity relation for the local Cayley transforms, the operator cannot be similar to a self-adjoint one.  
\end{example}

\begin{example}[Similarity by geometry]
	Consider the interval $[0,a]$ and impose the boundary condition 
	\begin{eqnarray*}
		\psi^{\prime}(0)+ (i\alpha-\beta) \psi(0) & \mbox{and} & \psi^{\prime}(a)+ (i\alpha+\beta) \psi(a)=0, \quad \mbox{for } \alpha,\beta\in \IR,
	\end{eqnarray*}
	\cf~\cite[Sec.~6.3]{Znojil2006} which
	in matrix notation becomes 
	\begin{eqnarray*}
		A=\begin{bmatrix} i\alpha-\beta & 0 \\ 0 & -(i\alpha+\beta)  \end{bmatrix} & \mbox{and} & B=\begin{bmatrix} 1 & 0 \\ 0 & 1  \end{bmatrix}.
	\end{eqnarray*}
	In \cite{Znojil2006} it has been shown that
	$\beta=0$ the spectrum is real, and if $\alpha\neq n\pi/ a$, $n\in\IN$, then $-\Delta(A,B)$ is similar to a self-adjoint operator. 
	This highlights that similarity of Laplacians on graphs can be achieved also without  similarity of the local boundary conditions.
	
	These boundary conditions were introduced in~\cite{Znojil2006} 
	and studied further in \cite{Krej2008}. The case with $\beta\not=0$
	is studied in~\cite{Krej2010}, and the general case of Robin boundary conditions on each endpoint, including a discussions of the previously mentioned cases, is studied in \cite{Krej2014} where some similarity transforms are computed explicitly, see also \cite{Krej2012, Krej2011, Krej2018} for an extension of this model.
	A generalisation of this example to metric graphs was proposed 
	in \cite{Znojil2015}. 
	The eigenvalues and eigenvectors have already been studied in \cite{Miha1962} and  \cite[Sec.XIX.3]{DSIII}. 
\end{example}

\subsection{Proof of Theorem \ref{thm:similarity}}\label{subsec:proof_thm_similarity}
For the proof of Theorem~\ref{thm:similarity}, it is essential that for star graphs a generalized reflection principle applies. Odd and even refection along a hyperplane is often applied to pass from second order elliptic differential operators on the whole space to half space problems with Dirichlet and Neumann boundary conditions, respectively. The generalized reflection principle for star graphs proposed in \cite[Section 6]{HKS2015} gives that similarity of the boundary conditions implies similarity of the Laplace operators. More precisely, if $\cI=\emptyset$, then for
\begin{align*}
A= G^{-1}A'G \quad \hbox{and} \quad B= G^{-1}B'G, 
\end{align*}
one can consider the map
\begin{align*}
\Phi_G\colon \Dom(\Delta(A,B)) \rightarrow \Dom(\Delta(A',B')), \quad \Phi_G\{\psi_i\}_{i\in \cE} =  \{G \II_{\cE,\cE}\psi_i\}_{i\in \cE}
\end{align*} 
which induces the similarity
\begin{align*}
\Delta(A,B) = \Phi_{G^{-1}}\Delta(A',B') \Phi_G.
\end{align*}

Note that if $-\Delta(A,B)$ is similar to a self-adjoint operator, \ie, in particular a normal one, then by Corollary~\ref{cor:normal}, the boundary conditions are quasi-sectorial. Using the similarity transform $\Phi_G$, see \cite[Section 6]{HKS2015}, one can consider without loss of generality boundary conditions of the form ~\eqref{eq:bc_LP}. Then only  a combination of the three cases (a), (b), (c), or (d) can occur, where
\begin{enumerate}[(a)]
	\item $\sigma(L)\cap \IC_{\Re>0}\setminus (0,\infty)\neq \emptyset$,
	\item 
	$\sigma(L)\cap i \IR\setminus\{0\} \neq \emptyset$,
	\item  there exists $l\in \sigma(L)\cap [0,\infty) \neq \emptyset$ such that there is a cyclic vector to the eigenvalue $l$,
	\item $\sigma(L)\subset \IC_{\Re<0}\cup [0,\infty)$, where $l\in \sigma(L)\cap[0,\infty)$ has no cyclic vector. 	
\end{enumerate}
	The key observation of the proof of Theorem~\ref{thm:similarity} is that poles of the Cayley transform are also poles of the resolvent, or of some meromorphic extension  of the resolvent to the unphysical sheet.

For case (a), note that $k^2\in \sigma_p(-\Delta(A,B))$ if for $\Im k>0$ $\det (A+ik B)=0$. For quasi-sectorial boundary conditions this means that $\det(L+ik)=0$, where $L$ is considered as an operator in $\IC^m$ with $L$ and $m$ from the quasi-Weierstrass normal form, see Proposition~\ref{prop:quasiweierstrass},  and if $l\in \sigma(L)\cap \IC_{\Re>0}\setminus (0,\infty)$, then $k=il$ is a zero of $\det(L+ik)=0$ with $\Im k>0$, and hence $-l^2\in \IC\setminus \IR$ is an eigenvalue
 of $-\Delta(A,B)$. Since the spectrum is stable under similarity transforms, $-\Delta(A,B)$ cannot be similar to a self-adjoint operator.

For case (b) and (c), recall that for $\cI=\emptyset$ one has $\sigma_r(-\Delta(A,B))=\emptyset$, see \cite[Proposition 4.6]{HKS2015}, $\sigma_p(-\Delta(A,B))\subset \IC\setminus [0,\infty)$ consists of finitely many eigenvalues with finite multiplicity, and moreover,  $\sigma_{ess}(-\Delta(A,B))=[0,\infty)$. 
If this operator is similar to a self-adjoint one, then there exists an equivalent norm $\norm{\cdot}_n$ in $\L^2(\cG)$ induced by a scalar product with respect to which $-\Delta(A,B)$ is self-adjoint. Hence, using that for self-adjoint operators the norm of the resolvent with spectral parameter $\lambda\pm i\varepsilon$ is given by the distance $\dist(\lambda\pm i\varepsilon, \sigma(-\Delta(A,B)))$ to the spectrum, there exists a $C>0$ such that
\begin{align*}
C\norm{(-\Delta(A,B)-\lambda\pm i\varepsilon)^{-1}} \leq \norm{(-\Delta(A,B)-\lambda\pm i\varepsilon)^{-1}}_n = \varepsilon^{-1}
\end{align*} 
for $\lambda\in \sigma(-\Delta(A,B))\subset \IR$ and $\varepsilon>0$.
This would imply that for any  $-\Delta(A',B')$ which is self-adjoint with respect to the standard scalar product 
\begin{align}\label{eq:caseb}
\varepsilon\norm{(-\Delta(A,B)-\lambda\pm i\varepsilon)^{-1} -(-\Delta(A',B')-\lambda\pm i\varepsilon)^{-1} } \leq  1/C+1   
\end{align} 
for $\lambda\in \sigma(-\Delta(A,B)) \cap \sigma(-\Delta(A',B'))$ and $\varepsilon>0$,
and this carries over to any $\lambda\in \IR$ and $\varepsilon>0$ sufficiently small because if $\lambda \notin \sigma(-\Delta(A,B))$ or $\lambda \notin \sigma(-\Delta(A',B'))$, then
\begin{align*}
\varepsilon\norm{(-\Delta(A,B)-\lambda\pm i\varepsilon)^{-1}}\to 0 \hbox{ or }  \varepsilon\norm{(-\Delta(A',B')-\lambda\pm i\varepsilon)^{-1} } \to 0 \quad \hbox{as } \varepsilon \to 0, 
\end{align*}
respectively.

The strategy is now to show that the poles of the Cayley transform discussed in Lemma~\ref{lemma:poles} lead to singularities of the resolvent which cause a contractions to \eqref{eq:caseb}.

In case (b), let first $l\in \sigma(L)\cap i(-\infty,0)$, then $\det (A+ik B)$ has a zero at $k=il\in(0,\infty)$, 
since it is a zero of $\det(L+ik)=0$, where $L$ is considered as an operator in $\IC^m$ with $L$ and $m$ from the quasi-Weierstrass normal form in Proposition~\ref{prop:quasiweierstrass}.
Hence $\mathfrak{S}(k,A,B)$ has a pole at $il\in (0,\infty)$ by Lemma~\ref{lemma:poles}.

If in the case of (b), the second case $l\in \sigma(L)\cap i(0,\infty)$ occurs, then $\overline{l}\in \sigma(L^\ast)\cap i(-\infty,0)$.
%
Using the generalized reflection principle, one can assume without loss of generality that $G=\mathds{1}$ in \eqref{eq:S_blockmatrix}, and then $-\Delta(A,B)$ can be equivalently described by $A=L +P$ and $B=P^\perp$ according to the quasi-Weierstrass normal form in Proposition~\ref{prop:quasiweierstrass}. For the adjoint one then has
\begin{align*}
-\Delta(A,B)^{\ast} = -\Delta(A',B') \quad \hbox{with } A'=L^{\ast} +P \hbox{ and } B'=P^\perp,
\end{align*}
compare \cite[Corollary 3.8]{HKS2015}, and therefore
the Cayley transform for the adjoint operator $\mathfrak{S}(k,A',B')$ has a pole at $i\overline{l}\in (0,\infty)$. Since similarity of $-\Delta(A,B)$ and $-\Delta(A,B)^\ast=-\Delta(A',B')$ to a self-adjoint operator is equivalent, it is sufficient to consider one of the two operators, and one can assume without loss of generality that $l\in \sigma(L)\cap i(-\infty,0)$.

So, let $l\in \sigma(L)\cap i(-\infty,0)$ and let $v$ be an eigenvector to $l$, then consider the normalized function $\psi$ defined by $\psi(x):=(1/a)\chi_{[0,a]}v$ for some $a>0$. Now comparing the resolvents of $-\Delta(A,B)$  and of the Neumann Laplacian $-\Delta(0,\mathds{1})$, one obtains for $\Im k>0$ using the representation of the resolvent kernel given in Subsection~\ref{subsec:resolvent_reg} that
\begin{align*}
\varphi_k:&=(-\Delta(A,B)-k^2)^{-1}\psi- (-\Delta(0,\mathds{1})-k^2)^{-1}\psi \\
&= \frac{i}{2k}\phi_{\cE}(x,k) (\mathfrak{S}(k,A,B)-\mathds{1}) \left(\int_{0}^a e^{iky}  dy\right) v \\
 &=\frac{i}{2k}\phi_{\cE}(x,k) \left(-\frac{l-ik}{l+ik}-1\right)  \left(\int_{0}^a e^{iky}  dy\right) v.
\end{align*}
For $k=\Re k + i \Im k$ one obtains for $k^2$ that
\begin{align*}
\lambda_k := \Re  k^2=(\Re k)^2 - (\Im k)^2\quad \hbox{and} \quad \varepsilon_k:=\Im k^2 = 2 (\Re k)\cdot (\Im k).
\end{align*}
Let $l=-i\xi$ with $\xi>0$, then one considers $k$ with $\Re k = \xi$, and then for $\Im k>0$ using that  $\norm{\phi_{\cE}(\cdot,k)v}=\norm{v}/\sqrt{2\Im k}$ for $v\in \IC^{d}$
\begin{align*}
\varepsilon_k\norm{\varphi_k} &= 2 \xi\cdot (\Im k)\frac{1}{2|k|}\frac{1}{\sqrt{2\Im k}}\frac{2\xi}{|\Im k|}  \frac{|e^{ika}-1|}{|k|} \norm{v} \\
 &= \frac{|e^{ika}-1|}{|k|^2}\frac{2\xi^2}{\sqrt{2\Im k}} \norm{v} \\
 &\geq c\frac{\xi^2}{\xi^2 + (\Im k)^2}\frac{\sqrt{2}}{\sqrt{\Im k}} \to \infty \quad \hbox{as}\quad \Im k \to 0,
\end{align*}
where one uses that for $k=\Re k +i \Im k$ with $\Re k, \Im k\in \IR$ and $\Re k\neq 0$ there exists since $\xi>0$ a $c>0$ such that  
\begin{align*}
|e^{ika}-1|=|e^{i\xi a - (\Im k) a}-1|>c \quad \hbox{as }\Im k \to 0.
\end{align*}
Note that the singular term $1/\sqrt{\Im k}$ is due to a singularity of the Cayley transform.
This is a contradiction to \eqref{eq:caseb}, because $\lambda_k\in [0,\infty)\subset \sigma(-\Delta(A,B))\cap \sigma(-\Delta(0,\mathds{1}))$ for $\Im k$ sufficiently small, and for $\Im k \to 0$ by the computation given above
\begin{align*}
\infty\leftarrow\varepsilon_k\norm{\varphi_k}
\leq \varepsilon_k\norm{(-\Delta(A,B)-\lambda_k\pm i\varepsilon_k)^{-1} -(-\Delta(0,\mathds{1})-\lambda_k\pm i\varepsilon_k)^{-1}}.
\end{align*}

Consider now the case (c) where $l\in \sigma(L)\cap [0,\infty)$ where $v$ with $\norm{v}=1$ is a cyclic vector to the eigenvalue $l$ with $L v = (l+N) v$ for a nilpotent matrix $N$ with $Nv\neq 0$ and $N^2v=0$. From \eqref{eq:S_NL} and \eqref{eq:S_N0} in the proof of Lemma~\ref{lemma:poles} one deduces that then
\begin{align*}
\mathfrak{S}(k,A,B)v = \left(-\mathds{1} + \frac{2ik}{l +ik}  \left(\mathds{1} + \frac{N}{-(l +ik)} \right)\right)v\quad \hbox{for }  l +ik\neq 0.
\end{align*}  
Then similar to the above but using this time the Dirichlet Laplacian $-\Delta(\mathds{1},0)$ for comparison
\begin{align*}
\varphi_k:&=(-\Delta(A,B)-k^2)^{-1}\psi-(-\Delta(\mathds{1},0)-k^2)^{-1}\psi \\
&= \frac{i}{2k}\phi_{\cE}(x,k) (\mathfrak{S}(k,A,B)+\mathds{1}) \left(\int_{0}^a e^{iky}  dy\right) v \\
&=\frac{i}{2k}\phi_{\cE}(x,k) \frac{2ik}{l +ik}  \left(\mathds{1} + \frac{N}{-(l +ik)} \right) \left(\int_{0}^a e^{iky}  dy\right) v.
\end{align*}

For $l=0$, assume that $\Re k = \Im k$, \ie, $\lambda_k=0$, $\varepsilon_k= 2(\Im k)^2$, and $|k^2|= 2(\Im k)^2$. Then 
\begin{align*}
\varepsilon_k\norm{\varphi_k} &= 2  (\Im k)^2\frac{1}{2|k|}\frac{1}{\sqrt{2\Im k}} \frac{2|k|}{|k|}\norm{v-\tfrac{1}{ik}Nv} \frac{|e^{ika}-1|}{|k|} \\
&= \frac{|e^{ika}-1|}{|k|}\frac{1}{\sqrt{2\Im k}} \norm{ik v-Nv}  \to \infty \quad \hbox{as}\quad \Im k \to 0,
\end{align*}
where $\frac{|e^{ika}-1|}{|k|}\to a$ and $\norm{ik v-Nv}\to \norm{Nv}$  as $k\to 0$. Since $0\in \sigma(-\Delta(A,B))\cap \sigma(-\Delta(\mathds{1},0))$, one concludes as above that this is a contradiction to \eqref{eq:caseb}.

For $l>0$, consider $k$ with $\Im k = l$ and $\Re k >0$, \ie, $\lambda_k= (\Re k)^2-l^2$, $\varepsilon_k= 2l(\Re k)$. Then
\begin{align*}
\varepsilon_k\norm{\varphi_k} &= 2 l (\Re k)\frac{1}{2|k|}\frac{1}{\sqrt{2 l}} \frac{2|k|}{|l +ik|} \frac{1}{|l+ik|} \norm{-(l+ik)v + Nv} \frac{|e^{ika}-1|}{|k|} \\
&= \frac{1}{\sqrt{2 l}} \frac{1}{(\Re k)} \norm{-(i\Re k)v + Nv} \frac{|e^{ika}-1|}{|k|} \\
&\simeq \frac{1}{(\Re k)} \frac{1}{\sqrt{2 l}}  \norm{Nv} \frac{|e^{-la}-1|}{|l|}
\to \infty \quad \hbox{as}\quad \Re k \to 0,
\end{align*}
 where the singular term $1/(\Re k)$ again comes from a singularity of the Cayley transform, thus contradicting  \eqref{eq:caseb}.

Now, consider case (d). First, note that the spectrum is real, and it consists of $\sigma_{ess}(-\Delta(A,B))=[0,\infty)$ plus an at most finite set of negative eigenvalues with finite multiplicity. Characterizations for the similarity of operators with real spectrum to self-adjoint operators have been 
proven more or less simultaneously in different variants by van Casteren \cite[Theorem 3.1]{Casteren1983}, Naboko
\cite[Theorem 2]{Naboko1984}, and Malamud \cite[Theorem 1]{Malamud1985}. 
\begin{theorem}[cf. Theorem 2 in 
	\cite{Naboko1984} and Theorem 1 in \cite{Malamud1985}]\label{thm:naboko}
	A closed densely defined linear operator $T$  in a Hilbert space $\cH$  is similar to a self-adjoint operator if and only if its  spectrum is real and there exists a constant $C>0$ such that for all $g\in \cH$
	\begin{align*}
	\sup_{\varepsilon>0}  \int_{\IR+i\varepsilon} \norm{(T-z)^{-1} g}^2 dz &\leq C \norm{g}^2, \\
	\sup_{\varepsilon>0}  \int_{\IR+i\varepsilon} \norm{(T^\ast-z)^{-1} g}^2 dz &\leq C \norm{g}^2.
	\end{align*}
\end{theorem}
This theorem has been applied in a similar situation in \cite{GrodKuzhel2014}.
	Here, the estimate in Theorem~\ref{thm:naboko} could be verified directly since the resolvent kernel is given explicitly, but it involves a rather lengthy computation which can be avoided by the series of comparison arguments.

In the situation of (d), in the Jordan normal form for the matrix $L$ from the quasi-Weierstrass normal form, see Proposition~\ref{prop:quasiweierstrass}, the eigenvalues in $[0,\infty)$ correspond to a  part $L_{\Re L \geq 0}$, which due to the assumption on the algebraic and geometric multiplicities is diagonal, while the eigenvalues in $\IC_{\Re<0}$ correspond to a possibly non-diagonal block
$L_{\Re<0}$. 
Hence, the matrices $A,B$ are similar to
\begin{align*}
A'= \begin{bmatrix}
L_{\Re L \geq 0} & 0 & 0 \\ 0 & L_{\Re<0}
&0 \\
0 & 0 & \mathds{1}
 \end{bmatrix} \quad \hbox{and}\quad 
 B'= \begin{bmatrix}
 \mathds{1} & 0 & 0 \\ 0 & \mathds{1}
 &0 \\
 0 & 0 & 0
 \end{bmatrix} 
\end{align*}
with $\sigma(L_{\Re L \geq 0})\subset (0,\infty)$ and $L_{\Re L \geq 0}$ self-adjoint, and $\sigma(L_{\Re<0})\subset \IC_{\Re<0}$. Due to the generalized reflection principle on star graphs, $-\Delta(A,B)$ is similar to $-\Delta(A',B')$, cf. \cite[Theorem 6.2]{HKS2015}
Now define
\begin{align*}
A_{sa}= \begin{bmatrix}
L_{\Re L \geq 0} & 0 & 0 \\ 0 & 0
&0 \\
0 & 0 & \mathds{1}
\end{bmatrix} \quad \hbox{and}\quad 
B_{sa}= \begin{bmatrix}
\mathds{1} & 0 & 0 \\ 0 & \mathds{1}
&0 \\
0 & 0 & 0
\end{bmatrix}. 
\end{align*} 
 Then $\Delta(A_{sa},B_{sa})$ is a self-adjoint operator, and
\begin{align*}
\mathfrak{S}(k,A',B') - \mathfrak{S}(k,A_{sa},B_{sa})= \begin{bmatrix}
0 & 0 & 0\\
0  & \mathfrak{S}(k,L_{\Re<0},\mathds{1})& 0 \\
0 & 0 & 0
\end{bmatrix}.
\end{align*}
The poles of $\mathfrak{S}(k,L_{\Re<0},\mathds{1})$ ly in the lower half plane with $\IC_{\Im <0}$, compare Lemma~\ref{lemma:poles}, and since $L_{\Re<0},\mathds{1}$ define quasi-sectorial boundary conditions one has by Lemma~\ref{lemma:Nilpotent} that $\mathfrak{S}(k,L_{\Re<0},\mathds{1})$ is uniformly bounded on  $\IC_{\Im >0}$. Hence, for some $C>0$
\begin{align}
\label{eq:diff}
\begin{split}
&\norm{(-\Delta(A',B')-k^2)^{-1}\psi-(-\Delta(A_{sa},B_{sa})-k^2)^{-1}\psi}^2 \\ =&\norm{\frac{i}{2k}\phi_{\cE}(\cdot,k) \left( \mathfrak{S}(k,A',B') - \mathfrak{S}(k,A_{sa},B_{sa})\right) \int_{\cG} \phi_{\cE}(y,k) \psi(y) dy}^2 \\
\leq&  C \frac{1}{4|k|^2} \frac{d}{2\Im k}  \norm{\int_{\cG} \phi_{\cE}(y,k) \psi(y) dy}^2
\end{split}
\end{align}
for any $\psi\in L^2(\cG)$  since $\norm{\phi_{\cE}(\cdot, k)}^2=d/(2\Im k)$ for $\Im k>0$.
Next note that since $\mathfrak{S}(k,\mathds{1},0)=-\mathds{1}$ and $\mathfrak{S}(k,0,\mathds{1})=\mathds{1}$
 \begin{align*}
 \norm{(-\Delta(\mathds{1},0)-k^2)^{-1}\psi-(-\Delta(0,\mathds{1})-k^2)^{-1}\psi}^2 &= \norm{\frac{i}{k}\phi_{\cE}(x,k) \int_{0}^\infty \phi_{\cE}(y,k) \psi(y_j) dy}^2 \\
 &= \frac{1}{|k|^2}\frac{1}{2\Im k} \norm{\int_{0}^\infty \phi_{\cE}(y,k) \psi(y_j) dy}^2,
 \end{align*}
 and since both operators  $-\Delta(\mathds{1},0)$ and $\Delta(0,\mathds{1})$ are self-adjoint, by Theorem~\ref{thm:naboko}, using the branch of the square root with $\Im \sqrt{\cdot}>0$, there exists another constant $C>0$ such that
  	\begin{align*}
\sup_{\varepsilon>0}  \int_{\IR+i\varepsilon}
\frac{1}{2|z| \cdot \Im \sqrt{z}}  \norm{\int_{\cG} \phi_{\cE}(y,\sqrt{z}) \psi(y) dy}^2 dz
	\leq C  \norm{\psi}^2  \quad \hbox{for all }\psi\in L^2(\cG).
	\end{align*}
 Consequently, by \eqref{eq:diff} there exists a possibly larger constant $C>0$ such that for all $\psi\in L^2(\cG)$
 \begin{multline*}
\sup_{\varepsilon>0}  \int_{\IR+i\varepsilon}  \norm{(-\Delta(A',B')-z)^{-1}\psi}^2 dz \leq  
\sup_{\varepsilon>0}  \int_{\IR+i\varepsilon}
\norm{(-\Delta(A',B')-z)^{-1}\psi-(-\Delta(A_{sa},B_{sa})-z)^{-1}\psi}^2 dz   \\ + 
\sup_{\varepsilon>0}  \int_{\IR+i\varepsilon}
\norm{(-\Delta(A',B')-z)^{-1}\psi}^2 dz
\leq C \norm{\psi}^2.
 \end{multline*}
 For $(-\Delta(A,B)^\ast-z)^{-1}$ an analogous argument applies, where one uses that the by \cite[Proposition 3.7]{HKS2015} boundary conditions for the adjoint operator $-\Delta(A_{ad}, B_{ad})= -\Delta(A, B)^\ast$ are given by
 \begin{eqnarray*}
 	A_{ad}:= - \frac{1}{2} \left(\mathfrak{S}(k,A,B)^{\ast} -\mathds{1}\right) & \mbox{and} & B_{ad}:=  \frac{1}{-2i\overline{k}} \left(\mathfrak{S}(k,A,B)^{\ast} +\mathds{1}\right).
 \end{eqnarray*}
Hence 
 \begin{align*}
 \mathfrak{S}(-\overline{k},A_{ad},B_{ad}) =\mathfrak{S}(k,A,B)^{\ast},
 \end{align*}
 and its poles behave as the ones of the adjoint since for $A_{ad},B_{ad}$ the matrix $L_{ad}$ in the quasi-Weierstrass normal form can be chosen as $L_{ad}=L^\ast$ with $L$ determined by the quasi-Weierstrass normal form of $A,B$.
  Hence by Theorem~\ref{thm:naboko} $-\Delta(A',B')$  is similar to a self-adjoint operator and by the generalized reflection principle used before also $-\Delta(A,B)$ is similar to a self-adjoint operator.
\qed

\subsection*{Acknowledgements} I would like to express my gratitude to David Krej\v{c}i\v{r}\'{i}k, Delio Mugnolo, and
Petr Siegl for many helpful discussions and insights on the subtleties of non-self-adjointness in operator theory. Moreover, I am very thankful to the lecturers of the 23rd Internet Seminar, Christian Seifert, Sascha Trostorff, and Marcus Waurick, who brought the quasi-Weierstrass normal form to my knowledge and which has been the ``missing link'' to complete this work. 
Also, I would like to thank the anonymous referees for their valuable comments. 


\bibliographystyle{alpha}
\bibliography{literatur}

\end{document}